%% file: floerbraidspapersubmission.tex
\documentclass[12pt,reqno,palatino]{amsart}

\usepackage{amsfonts,bm,mathrsfs}
\usepackage{subfigure}
\usepackage{latexsym}
\usepackage{afterpage}
\usepackage{mathptm}
\usepackage{amsmath,amscd,textcomp,amssymb}
\usepackage{stmaryrd}
\usepackage{graphicx,eucal,times,hyperref}

\input{macros-theorems.tex}
\input{macros-gvvw.tex}

\numberwithin{equation}{subsection}
\numberwithin{figure}{section}

\newcommand{\style}[1]{{\sc{#1}}}    

\begin{document}

\begin{sloppypar}

\title[Braid Floer homology]{}

\author[J.B. van den Berg,
R. Ghrist, R.C. Vandervorst and W. W\'ojcik]{}

\maketitle \noindent {\huge {\bf Braid Floer homology}\footnote{RG supported in part by DARPA
HR0011-07-1-0002 and by ONR N000140810668.
RV supported in part by
NWO Vidi 639.032.202.}} \vskip.8cm

\noindent
J.B. van den Berg{$^a$},
R. Ghrist{$^b$}, R.C. Vandervorst{$^a$} and W. W\'ojcik{$^a$}
\vskip.3cm

\noindent {\it {$^a$}Department of Mathematics, VU Universiteit
Amsterdam, the Netherlands.}

\noindent {\it {$^b$}Departments of Mathematics and
Electrical/Systems Engineering, University of Pennsylvania,
Philadelphia, USA.}

\vskip1cm

\noindent {\bf Abstract.} Area-preserving diffeomorphisms of a 2-disc
can be regarded as time-1 maps of (non-autonomous) Hamiltonian
flows on $\Sb^1\times \D^2$, periodic flow-lines of which define braid
(conjugacy) classes, up to full twists. We examine the dynamics
relative to such braid classes and define a braid Floer homology. This
refinement of the Floer homology originally used for the Arnol'd
Conjecture yields a Morse-type forcing theory for periodic points of
area-preserving diffeomorphisms of the 2-disc based on braiding.

Contributions of this paper include (1) a monotonicity lemma for the
behavior of the nonlinear Cauchy-Riemann equations with respect to
algebraic lengths of braids; (2) establishment of the topological
invariance of the resulting braid Floer homology; (3) a shift theorem
describing the effect of twisting braids in terms of shifting the braid
Floer homology; (4) computation of examples; and (5) a forcing
theorem for the dynamics of Hamiltonian disc maps based on braid
Floer homology.

\vspace{0.1in}
\noindent {\em AMS Subject Class:} 37B30, 57R58, 37J05.



 \section{Motivation}
 \label{sec:motivation}

The interplay between dynamical systems and algebraic topology is
traceable from the earliest days of the qualitative theory: it is no
coincidence that Poincar\'e's investigations of invariant manifolds and
(what we now know as) homology were roughly coincident. Morse
theory, in particular, provides a nearly perfect mirror in which
qualitative dynamics and algebraic topology reflect each other. Said
by Smale to be the most significant single contribution to
mathematics by an American mathematician, Morse theory gives a
relationship between the dynamics of a gradient flow on a space $X$
and the topology of this space. This relationship is often expressed as
a homology theory \cite{Milnor,Schwarz}. One counts (nondegenerate) fixed
points of $-\nabla f$ on a closed manifold $M$ (with $\Z_2$
coefficients), grades them according to the dimension of the
associated unstable manifold, then constructs a boundary operator
based on counting heteroclinic connections. Careful but
straightforward analysis shows that this boundary operator yields a
chain complex whose corresponding (Morse) homology ${\rm HM}_*(f)$
is isomorphic to $H_*(M;\Z_2)$, the (singular, mod-2) homology of
$M$, a topological invariant.

Morse's original work established the finite-dimensional theory and
pushed the tools to apply to the gradient flow of the energy function
on the loop space of a Riemannian manifold, thus using closed
geodesics as the basic objects to be counted. The problem of
extending Morse theory to a fully infinite-dimensional setting with a
strongly indefinite functional remained open until, inspired by the
work of Conley and Zehnder on the Arnol'd Conjecture, Floer
established the theory that now bears his name.

Floer homology considers a formal gradient flow and studies its set of
bounded flowlines. Floer's initial work studied the elliptic nonlinear
Cauchy-Riemann equations, which occur as a formal $L^2$-gradient
flow of a (strongly indefinite) Hamiltonian action. The key idea is
that no locally defined flow is needed: generically, the space of
bounded flow-lines has the structure of an invariant set of a gradient
flow. As in the construction of Morse homology one builds a complex
by grading the critical points via the Fredholm index and constructs a
boundary operator by counting heteroclinic flowlines between points
with difference one in index. The homology of this complex --- Floer
homology ${\rm HF}_*$ --- satisfies a continuation principle and
remains unchanged under suitable (large) perturbations. Floer
homology and its descendants have found use in the solution of the
Arnol'd Conjecture \cite{Floer1}, in instanton homology \cite{Floer5},
elliptic systems \cite{AV}, heat flows \cite{SalWeb}, strongly
indefinite functionals on Hilbert spaces \cite{AM1}, contact topology
and symplectic field theory \cite{EGH}, symplectic homology
\cite{FH1}, \cite{Oancea}, \cite{V1}, and invariants of knots, links, and
3-manifolds \cite{OS1}.

The disconnect between practitioners of Floer theory and applied
mathematicians is substantial, in large part due to the lack of
algorithms for computing what is in every respect a truly
infinite-dimensional construct. We suspect and are convinced that better
insights into the computability of Floer homology will be
advantageous for its applicability. It is that long-term goal that
motivates this paper.

The intent of this paper is to define a Floer homology related to the
dynamics of time-dependent Hamiltonians on a 2-disc. We build a
relative homology, for purposes of creating a dynamical forcing
theory in the spirit of the Sharkovski theorem for 1-d maps or Nielsen
theory for 2-d homeomorphisms \cite{Jiang}. Given one or more
periodic orbits of a time-periodic Hamiltonian system on a disc, which
other periodic orbits are forced to exist? Our answer, in the form of a
Floer homology, is independent of the Hamiltonian. We define the
Floer theory, demonstrate topological invariance, and connect the
theory to that of braids.

Any attempt to establish Floer theory as a tool for applied dynamical
systems must address issues of computability. This paper serves as a
foundation for what we predict will be a computational Floer theory
--- a highly desirable but challenging goal. By combining the results of
this paper with a spatially-discretized braid index from \cite{GVV}, we
hope to demonstrate and implement algorithms for the computation
of braid Floer homology. We are encouraged by the potential use of
the Garside normal form to this end: Sec.\ \ref{subsec:FMC} outlines
this programme.

\section{Statement of results}
\label{sec:intro-floer}

\subsection{Background and notation}
\label{subsec:area1}

Recall that a smooth orientable 2-manifold $M$ with area form
$\omega$ is an example of a symplectic manifold, and an
area-preserving diffeomorphism between two such surfaces is an
example of a symplectomorphism. Symplectomorphisms of
$(M,\omega)$ form a group $\symp(M,\omega)$
with respect to composition. The standard unit 2-disc in the plane
$\D^2 = \{x\in \R^2~|~ |x| \le 1\}$ with area form $\omega_0=
dp\wedge dq$ is the canonical example, with the
area-preserving diffeomorphisms as $\symp(\D^2,\omega_0)$.

Hamiltonian systems on a symplectic manifold are defined as follows.
Let $X_H(t,\cdot)$ be 1-parameter family of vector fields given  via
the relation $\iota_{X_H} \omega = -dH$, where $H(t,\cdot):M\to \R$
is a smooth family of functions, or Hamiltonians, with the property
that $H$ is constant on $\partial M$. This boundary condition is
equivalent to $i^*\iota_{X_H}\omega=0$ where $i: \partial M \to M$ is
the  inclusion. As a consequence $X_H(t,x)\in T_x\partial M$ for $x\in
\partial M$, and the differential equation
 \begin{equation}
 \label{eqn:HE}
 {dx(t)\over dt} = X_H(t,x(t)),\quad x(0) = x,
 \end{equation}
defines diffeomorphisms $\psi_{t,H}:M\to M$ via $\psi_{t,H}(x) \bydef
x(t)$ with $\partial M$ invariant.  Since $\omega$ is closed it holds that
$\psi_{t,H}^*\omega = \omega$ for any $t$, which implies that
$\psi_{t,H}\in\symp(M,\omega)$.
Symplectomorphisms of the form $\psi_{t,H}$ are called Hamiltonian,
and the subgroup of such is denoted $\ham(M,\omega)$.

The dynamics of Hamiltonian maps are closely connected to the
topology of the domain. Any map of $\D^2$ has at least one fixed
point, via the Brouwer theorem. The content of the Arnol'd
Conjecture is that the number of fixed points of a Hamiltonian map of
a (closed) $(M,\omega)$ is at least $\sum_k\dim H_k(M;\R)$, the sum
of the Betti numbers of $M$. Periodic points are more delicate still. A
general result by Franks \cite{Franks1} states that an area-preserving
map of the 2-disc has either one or infinitely many periodic points
(the former case being that of a unique fixed point, e.g., irrational
rotation about the center). For a large class of closed symplectic
manifolds $(M,\omega)$ a similar result was proved by Salamon and
Zehnder \cite{SalZehn1} under appropriate non-degeneracy
conditions; recent results by Hingston for tori \cite{Hingston} and
Ginzburg  for closed, symplectically aspherical manifolds show that
any Hamiltonian symplectomorphism has infinitely many
geometrically distinct periodic points \cite{Ginzburg}. These latter
results hold without non-degeneracy conditions.

In this article we develop a more detailed Morse-type theory for
periodic points utilizing the fact that orbits of non-autonomous
Hamiltonian systems form links in $\Sb^1\times \D^2$. Let  $A_m
=\{y^1,\cdots,y^m\} \subset \D^2$ be a discrete invariant set for a
Hamiltonian map which preserves the set $A_m$: are there additional
periodic points? We build a forcing theory based of Floer homology to
answer this question.

\subsubsection{Closed braids and Hamiltonian systems}
\label{subsec:Ham1}
We focus our attention on $\D^2_m$, the unit disc with $m$ disjoint
points removed from the interior.

\begin{lemma}
\label{lem:hamdisc} [Boyland \cite{Boyland2} Lemma 1(b)]
For every area-preserving diffeomorphism $f$ of $\D^2_m$ there
exists a Hamiltonian isotopy $\psi_{t,H}$ on $\D^2$ such that
$f=\psi_{t,H}$.
\end{lemma}

A detailed proof of this is located in the appendix, for completeness.
We note that (contrary to the typical case in the literature)
$\partial\D^2_m$ is assumed invariant, but not pointwise fixed. As a
consequence of the proof, the Hamiltonian $H(t,x)$ can be assumed
$C^\infty$, $1$-periodic in $t$, and vanishing on $\partial\D^2$. We
denote this class of Hamiltonians by $\cH$.

\begin{figure}[hbt]
\label{fig:braids}
\begin{center}
\includegraphics[width=6in]{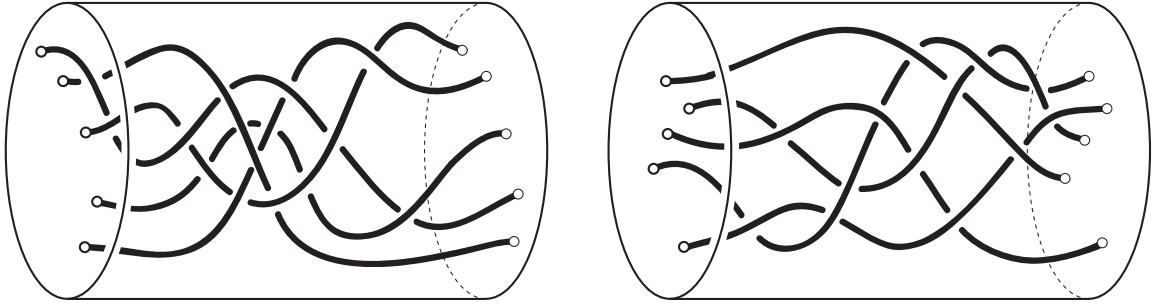}
\caption{A pair of braids on five strands are in the same connected component
of $\Omega^5$.}
\end{center}
\end{figure}

Periodic orbits of a map of $\D^2$ are described in terms of
configuration spaces and braids. The \style{configuration space}
$C^n(\D^2)$ is the space of all subsets of $\D^2$ of cardinality $n$,
with the topology inherited from the product $(\D^2)^n$. The free
loop space $\Omega^n$ of maps $\Sb^1\to C^n(\D^2)$ captures the
manner in which periodic orbits `braid' themselves in the Hamiltonian
flow on $\Sb^1 \times \D^2$, as in Fig. \ref{fig:braids}. Recall that
the classical \style{braid group} $\cB_n$ on $n$ strands is
$\pi_1(C^n(\D^2))$ (pointed homotopy classes of loops). Although
elements of $\Omega^n$ are not themselves elements of $\cB_n$, we
abuse notation and refer to such as braids.

To build a forcing theory, we work with \style{relative braids} ---
braids split into `free' and `skeletal' sub-braids. Denote by
$\Omega^{n,m}$ the embedded image of $\Omega^n \times
\Omega^m \hookrightarrow \Omega^{n+m}$. Such a relative braid is
denoted by $\x\rel\y$; its braid class $[\x\rel\y]$ is the connected
component in $\Omega^{n,m}$. A \style{relative braid class fiber}
$[\x]\rel \y$ is defined to be the subset of $\x'\in\Omega^n$ for which
$\x'\rel\y\in[\x\rel\y]$. This class represents all possible free braids
which stay in the braid class, keeping the \style{skeleton} $\y$ fixed:
see Fig. \ref{fig:relative}.

\subsubsection{The variational approach}
\label{subsec:var1}
Fix a Hamiltonian $H\in\cH$ and assume that $\y\in\Omega^n$ is a
(collection of) periodic orbit(s) of the Hamiltonian flow. We will
assign a Floer homology to relative braid classes $[\x]\rel\y$. Define
the \style{action} of $H\in \cH$ via
\begin{equation}
\label{eqn:action1a}
 \f_{H}(\x)
 =
 \int_0^1 \theta(\x(t))
 =
 \int_0^1 \overline \alpha_0 (\x_t(t)) dt - \int_0^1 \oH(t,\x(t))dt,
\end{equation}
where $\overline \alpha_0 = \sum_k p^kdq^k$, $\oH(t,\x(t)) = \sum_k
H(t,x^k(t))$; $\x_t = \frac{d\x}{dt}$; and $\theta = \overline \alpha_0
- \overline H dt$. The set of critical points of $\f_{H}$ in $[\x]\rel\y$
is denoted by $\P_H([\x]\rel\y)$.

\begin{figure}[hbt]
\label{fig:relative}
\begin{center}
\includegraphics[width=6in]{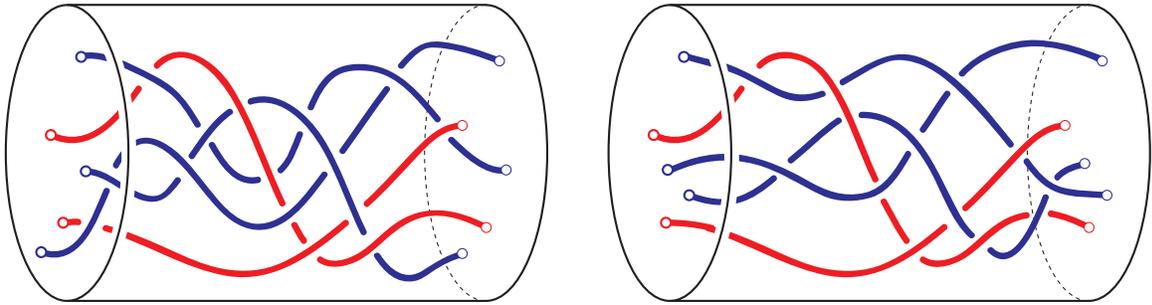}
\caption{Shown are two relative braids in the fiber $[\x]\rel\y$:
the fixed strands $\y$ are called the skeleton.}
\end{center}
\end{figure}

We investigate critical points of $\f_H$ via nonlinear Cauchy-Riemann
equations as an $L^2$-gradient flow of $\f_{H}$ on $[\x]\rel\y$.
Choose a smooth $t$-dependent family of compatible almost-complex
structures $J$ on $(\D^2,\omega_0)$: each being a mapping
$J:T\D^2\to T\D^2$, with $J^2 = -{\rm Id}$,
$\omega_0(J\cdot,J\cdot) = \omega_0(\cdot,\cdot)$ and
$g(\cdot,\cdot) = \omega_0(\cdot,J\cdot)$ a metric on $\D^2$. We
denote the set of $t$-dependent almost complex structures on $\D^2$
by $\cJ = \cJ(\Sb^1\times \D^2)$. For functions $u:\R\times \R \to
\D^2$ the nonlinear \style{Cauchy-Riemann equations} (CRE) are
defined as
\begin{equation}
\label{eqn:CR1}
\bigl(\partial_{J,H}(u)\bigr) (s,t)
\bydef
\frac{\partial u(s,t)}{\partial s}
-
J(t,u(s,t))\frac{\partial u(s,t)}{\partial t}
-
\nabla_g H(t,u(s,t)),
\end{equation}
where $\nabla_g$ is the gradient with respect to the metric $g$.
Stationary, or $s$-independent solutions $u(s,t) = x(t)$ satisfy  Eq.\
\rmref{eqn:HE} since $J\nabla_g H = X_H$. In order to find closed
braids as solutions we lift the CRE to the space of closed braids: a
collection $\u(s,t) =\{u^k(s,t)\}$ satisfies the CRE if each component
$u^k$ satisfies Eq.\ (\ref{eqn:CR1}).

\subsection{Result 1: Monotonicity}
There is a crucial link between bounded solutions of CRE and
algebraic-topological properties of the associated braid classes: {\em
braids decrease in word-length over time}. For $\x\in \Omega^n$, the
associated braid can be represented as a \style{conjugacy class} in
the braid group $\cB_n$, using the standard generators
$\{\sigma_i\}_{i=1}^{n-1}$. The \style{length} of the braid word (the
sum of the exponents of the $\sigma_i$'s) is well-defined and is a
braid invariant. Geometrically, this length is the total \style{crossing
number} $\Cr(\x)$ of the braid: the algebraic sum of the number of
crossings of strands in the standard planar projection (see Fig.
\ref{fig:generators}).

\begin{figure}[hbt]
\label{fig:generators}
\begin{center}
\includegraphics[width=6in]{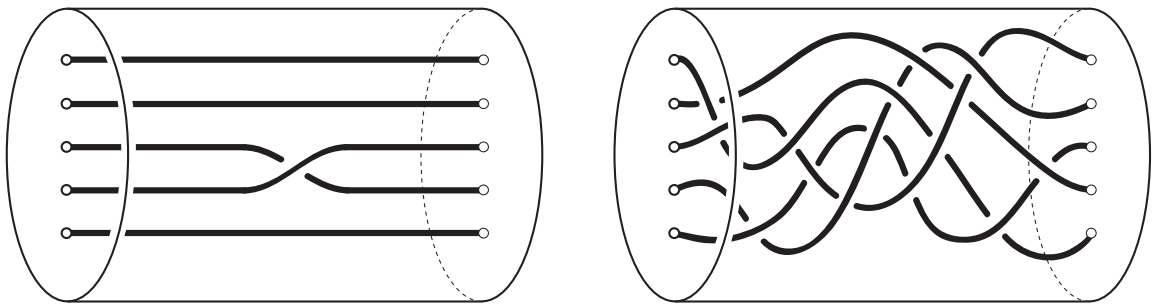}
\caption{The standard positive generator $\sigma_i$ for the braid group
$\cB_n$ sends the $i^{th}$ strand over the $(i+1)^{st}$.
Pictured in $\cB_5$ is (left) $\sigma_2$ and (right) the braid
$\sigma_4^{-1}\sigma_3\sigma_1\sigma_3\sigma_2^{-1}\sigma_1\sigma_2\sigma_3^{-1}
\sigma_4^{-1}\sigma_1\sigma_2\sigma_3\sigma_4^{-1}\sigma_1\sigma_2^{-1}$. This
braid has length (algebraic crossing number) equal to $+3$.}
\end{center}
\end{figure}

To make sense of this statement, we first assemble braid classes into
a completion. Denote by $\overline\Omega^n$ the space of maps
$\Sb^1\to(\D^2)^n/S_n$, the configuration space of $n$ {\em not
necessarily distinct} unlabeled points in $\D^2$. The discriminant of
\style{singular braids}
$\Sigma^n=\overline\Omega^n\backslash\Omega^n$ partitions
$\overline\Omega^n$ into braid classes; the relative versions of these
(i.e.\ $\x\rel\y$) arise when the skeleton $\y$ is fixed.

The Monotonicity Lemma \ref{lem:dislyap3} says, roughly, that the
flow-lines --- bounded solutions ---  of the CRE are transverse to the
singular braids in a manner that decreases crossing number.

\begin{lemma}[{\bf Monotonicity Lemma.}]
 \label{lem:dislyap3}
Let $\u(s) \in \bOn$ be a local solution of
the Cauchy-Riemann equations. If $\u(s_0,\cdot) \in \Sigma^n \rel
\y\bydef\bOn\backslash \Omega^n\rel \y$, then there exists an
$\epsilon_0>0$ such that
\begin{equation}
    \Cr(\u(s_0-\epsilon,\cdot))>\Cr(\u(s_0+\epsilon,\cdot))
    \quad
    \quad
    \forall \, 0<\epsilon\le \epsilon_0 .
\end{equation}
\end{lemma}


This consequence of the maximum principle follows from the
positivity of intersections of $J$-holomorphic curves in
almost-complex 4-manifolds \cite{McDuff1,McDuff2}; other
expressions of this principle arise in applications to heat equations in
one space dimension, see e.g. \cite{Ang3,Ang2}.

As a consequence, the local flowlines induced by the CRE are
topologically transverse to the singular braids off of the (maximally
degenerate) set of collapsed singular braids. This leads to an isolation
property for certain relative braid classes which makes a
Morse-theoretic approach viable. Denote by $\cM([\x]\rel\y)$ the set
of bounded solutions $\u(s,t)$ of the CRE contained a relative braid
class $[\x]\rel\y$. Consider braid classes which satisfy the following
topological property: for any representative $[\x]\rel\y$ the strands in
$\x$ cannot collapse onto strands   in $\y$, or onto strands in $\x$,
nor can the strands in $\x$ collapse onto the boundary $\partial
\D^2$. Such braid classes are called \style{proper}: see Fig.
\ref{fig:proper}. From elliptic regularity and the Monotonicity Lemma
we obtain compactness and isolation: for a proper relative braid
class, the set of bounded solutions is compact and isolated in the
topology of uniform convergence on compact subsets in $\R^2$.

\begin{figure}[hbt]
\label{fig:proper}
\begin{center}
\includegraphics[width=6in]{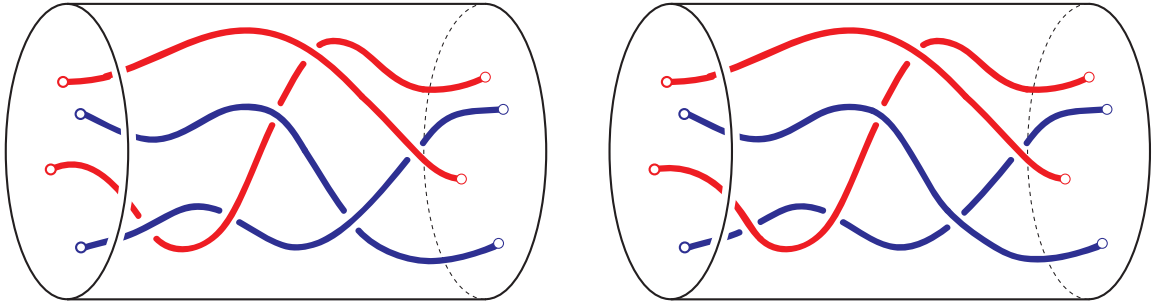}
\caption{Proper (left) and improper (right) relative braid classes.
On the right, the free strands can collapse onto the skeleton.}
\end{center}
\end{figure}

\begin{proposition}[{\bf Isolation and Compactness Theorem.}]
\label{prop:compiso1}
Let $[\x]\rel\y$ be a proper braid class.
Then, in the topology of uniform convergence on compact subsets in
$\R^2$, the set of bounded solutions $\cM([\x]\rel\y)$ is compact
and isolated in $[\x]\rel\y$.
\end{proposition}

The proof is a combination of Proposition \ref{prop:CR4} and
Proposition \ref{prop:isol1}.

\subsection{Result 2: Braid Floer Homology}
The above proposition is used to define a Floer homology (cf. \cite{Floer1})
for proper relative braid classes.
Define $\cN = [\x]\rel\y$, then by Proposition \ref{prop:compiso1} the
set of bounded solutions is in $\cN$ is compact and $|\u|<1$ for all
$\u \in \cM(\cN)$. We say that $\cM(\cN)$ is isolated. In order to
define Floer homology for $\cN$, the system needs to be embedded
into a generic family of systems. The usual approach is to establish
that (for generic choices of Hamiltonians $H\in \cH$, for which $\y\in
\P_H$) the critical points of the action are non-degenerate and the
sets of connecting orbits $\cM_{\x_-,\x_+}([\x]\rel\y)$ are finite
dimensional manifolds.

The Fredholm theory for CRE yields an index function $\mu$ on
stationary points of the action $\f_H$ and
 $\dim \cM_{\x_-,\x_+}([\x]\rel\y) = \mu(\x_-)-\mu(\x_+)$.
Following Floer \cite{Floer1} we can define a chain complex
$C_k([\x]\rel\y) = \bigoplus \Z_2 \langle \x\rangle$, where the direct
sum is taken over critical points  of index $k$. The boundary operator
$\partial_k: C_k \to C_{k-1}$ is the linear operator generated by
counting orbits (modulo 2) between critical points of the correct
indices. The structure of the space of bounded solutions reveals that
$(C_*,\partial_*)$ is a chain complex. The homology of this chain
complex --- the \style{braid Floer homology} --- is denoted ${\rm
HF}_*([\x] \rel \y;J,H)$. This is finite dimensional for all $k$ and
nontrivial for only finitely many values of $k$. Independence of
choices is our first major result. Any Hamiltonian dynamics which
leaves $\y$ invariant yields the same Floer homology, ${\rm
HF}_*([\x] \rel \y;J,H)$ is independent of  $J$ and $H$; homotopies of
$\y$ within the braid class also leave the braid Floer homology
invariant. To be more precise, for $\y,\y' \in [\y]$, the Floer homology
groups for the relative braid classes fibers $[\x]\rel\y$ and
$[\x']\rel\y'$ of $[\x\rel\y]$ are isomorphic, ${\rm
HF}_k([\x]\rel\y)\cong {\rm HF}_k([\x']\rel\y')$. This allows us to
assign the Floer homology to the entire product class $[\x\rel\y]$.

\begin{THM}[{\bf Braid Floer Homology Theorem.}]
 The braid Floer homology of a
proper relative braid class,
\begin{equation}
\label{eqn:defn-floer-1}
    \bH_*([\x\rel\y])\bydef {\rm HF}_*([\x]\rel\y),
\end{equation}
is a function of the braid class $[\x\rel\y]$ alone, independent of
choices for $J$ and $H$, and the representative skeleton $\y$.
\end{THM}

\subsection{Result 3: Shifts \& Twists}
The braid Floer homology $\bH_*$ entwines topological braid data
with dynamical information about braid-constrained Hamiltonian
systems. One example of the braid-theoretic content of $\bH_*$
comes from an examination of twists. Recall that the braid group
$\cB_n$ has as its group operation concatenation of braids. This does
not extend to a well-defined product on conjugacy classes; however,
$\cB_n$ has a nontrivial center $Z(\cB_n)\cong\Z$ generated by
$\Delta^2$, the \style{full twist} on $n$ strands. Thus, products with
full twists are well-defined on conjugacy classes. These full twists
have a well-defined impact on the braid Floer homology (Sec.
\ref{subsec:twists}). Twists shift the grading: see Fig. \ref{fig:twist}.

\begin{THM}[{\bf Shift Theorem.}]
\label{thm:garside-floer1}
Let $[\x\rel\y]$ denote a braid class with $\x$ having $n$ strands. Then
\[
    \bH_*\bigl([(\x\rel\y)\cdot \Delta^{2}]\bigr) \cong
    \bH_{*-2n}([\x\rel\y]) .
\]
\end{THM}

%
\begin{figure}[hbt]
\label{fig:twist}
\begin{center}
\includegraphics[width=6in]{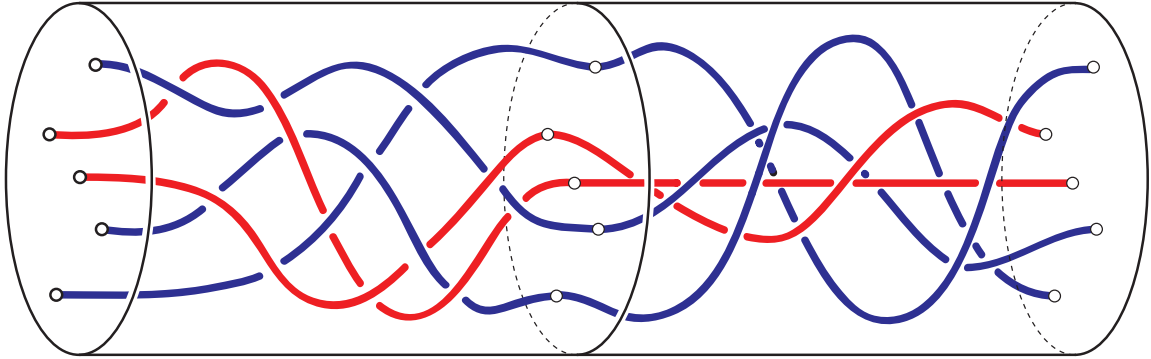}
\caption{The addition of a full positive twist (on right) to a proper
relative braid class (on left) shifts the Floer homology up by degree two.}
\end{center}
\end{figure}

In any Floer theory, computable examples are nontrivial. We compute
examples in Sec. \ref{sec:cyclic}, including the following.

\begin{example}
\label{ex:rotations-1} Consider a skeleton $\y$ consisting of two
braid components $\{\y^1,\y^2\}$, with $\y^1$ and $\y^2$  defined by
$$
\y^1 =   \Bigl\{  r_1 e^{\tfrac{2\pi n}{m} i t},\cdots, r_1 e^{\tfrac{2\pi n}{m} i (t-m+1)}\Bigr\},\quad
\y^2 =   \Bigl\{  r_2 e^{\tfrac{2\pi n'}{m'} i t},\cdots, r_2 e^{\tfrac{2\pi n'}{m'} i (t-m'+1)}\Bigr\}.
$$
where $0<r_1<r_2\le 1$, and $(n,m)$ and $(n',m')$ are relatively
prime integer pairs with $n\not = 0$, $m\ge 2$, and $m'>0$. A free
strand is given by $\x=\{x^1\}$, with $x^1(t) = re^{2\pi \ell it}$,
for $r_1<r<r_2$ and some $\ell\in\Z$, with either $n/m<\ell<n'/m'$ or $n/m>\ell>n'/m'$,
depending on the ratios of $n/m$ and
$n'/m'$. The relative braid class
$[\x\rel\y]$ is defined via the representative $\x\rel\y$.
The associated braid class is proper, and
$\bH_*([\x\rel\y])$ is non-zero in exactly two dimensions:
$2\ell$ and $2\ell\pm 1$, see Sec. \ref{sec:cyclic}.
\end{example}

\begin{figure}[hbt]
\label{fig:examples}
\begin{center}
\includegraphics[width=5.5in]{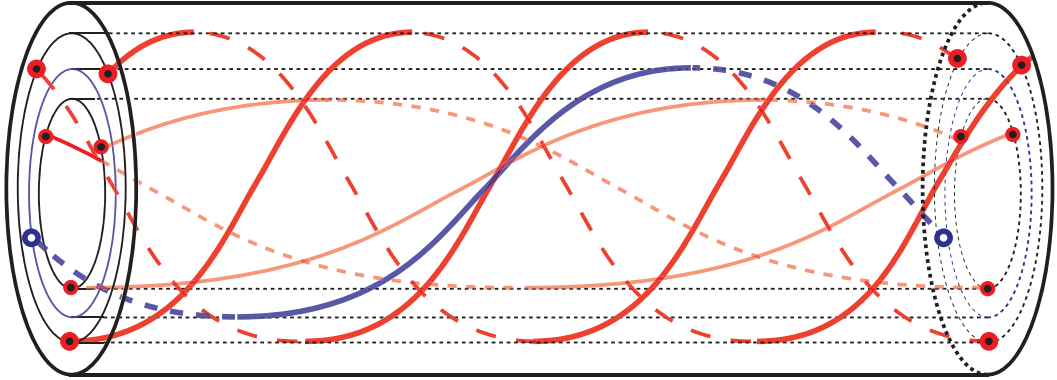}
\caption{The braid class of Example \ref{ex:rotations-1} has one free strand
surrounded by a pair of skeletal braids whose rotation numbers bound from above and below.}
\end{center}
\end{figure}

This example agrees with a similar computation of a
(finite-dimensional) Conley index of positive braid classes in
\cite{GVV}. Indeed, we believe that the braid Floer  homology agrees
with that index on positive braid classes. We anticipate using
Theorem \ref{thm:garside-floer1} combined with Garside's Theorem on
normal forms for braids as a means of algorithmically computing
$\bH_*$, see e.g. Sec.\ \ref{subsec:FMC}.

\subsection{Result 4: Forcing}
\label{subsec:mth1}
Floer braid homology $\bH_*$ contains information about the
existence of periodic points or invariant sets $A_m$ of
area-preserving diffeomorphisms $f$. Recall that an invariant set
$A_n = \{y^1,\cdots,y^n\}$ for $f$ determines a braid class $[\y]$ via
$\y(t) = \left(\psi_{t,H}(y^1),\cdots,\psi_{t,H}(y^n)\right)$. The
representation as an element $\beta(\y)$ in the braid group $\cB_n$,
given by a choice of a Hamiltoian $H$, is uniquely determined modulo
full twists.

\begin{THM}[{\bf Braid Forcing Theorem.}]
Let $f\in \symp(\D^2,\omega_0)$ have an invariant set $A_m$
representing the $m$-strand braid class $[\y]$. Then, for any proper
relative braid class $[\x\rel\y]$ for which  $\bH_*([\x\rel\y]) \not = 0$,
there exists an invariant set $A_n'$ for $f$ such that the union
$A_m\cup A_n'$  represents the relative braid class $[\x\rel\y]$.
\end{THM}

\begin{example}
\label{ex:rotations-2} Consider the braid class $[\x\rel\y]$ defined in
Ex. \ref{ex:rotations-1}. For any area-preserving diffeomorphism $f$
of the (closed) disc with invariant set represented (up to full twists)
by the braid class $[\y]$, there exist infinitely many distinct periodic
points. To prove this statement we invoke the invariant $\bH_*$,
computed in Ex. \ref{ex:rotations-1}, and use Theorem \ref{thm:garside-floer1}.
In particular, this result implies that if $f$ has
a fixed point at the boundary and  a periodic point of period larger
than two in the interior, then $f$ has periodic points of arbitrary
period, and thus infinitely many periodic points. In Sec.\
\ref{sec:cyclic} we give more details: the main results are presented in
Theorem \ref{thm:fixed1}.
\end{example}


\section{Background: configuration spaces and braids}
\label{sec:closedbraids}

Experts who want to get to the Floer homology are encouraged to skip
over the following background sections to Sec. \ref{sec:dislyap}.

\subsection{Configuration spaces and braid classes}
\label{subsec:config}

The \style{configuration space} $C^n(\D^2)$ is the space of subsets of
$\D^2$ of cardinality $n$, with the topology inherited from the
product $(\D^2)^n/S_n$, where $S_n$ is the permutation group.
Configuration spaces lead naturally to braids \cite{Birman}. For any
fixed basepoint, $\pi_1C^n(\D^2)\cong \cB_n$, Artin's braid group on
$n$ strands. For our purposes (finding periodic orbits of Hamiltonian
maps of the disc), the basepoint is unnatural, and we work with free
loop spaces.

The loop space $\Omega^n$ is the space of continuous mappings $\x:
\R/\Z \to C^n(\D^2)$ under the standard (strong) metric topology
induced by  $C^0(\R/\Z;\D^2)$. We abuse notation to indicate points
in $C^n(\D^2)$ and loops in $\Omega^n$ by the same symbol $\x$,
often referred to as a braid. In some contexts we denote loops by
$\x(t)$ and consider them as a union of disjoint `strands' $x^k(t)$.
Two loops $\x(t)$ and $\tilde\x(t)$ are close in the topology of
$\Omega^n$ if and only if for some permutation $\theta\in S_n$ the
strands $x^{\theta(k)}$ and $\tilde x^k$ are $C^0$-close for all $k$
and $\tilde \sigma = \theta^{-1} \sigma\theta$.

The \style{braid class} of $\x$ is the connected component $[\x] \in
\pi_0(\Omega^n)$. These braid classes can be completed to $\bOn$,
the $C^0$-closure of $\Omega^n$ in $(\D^2)^n/S_n$. The
discriminant $\Sigma^n = \bOn\backslash \Omega^n$ defines the
singular braids. A special subset of singular braids are those that can
be regarded as in $\overline \Omega^{n'}$ with $n'<n$. Such collapsed
braids are denoted by $\Sigma_-^n\subset \Sigma^n$.

Given a fixed skeleton braid $\y\in\Omega^m$, we define the relative
braids as follows. There is a natural inclusion
$\iota_{\y}:\Omega^n\hookrightarrow \overline\Omega^{n+m}$ taking
$\x$ to the union of $\x$ and $\y$ --- a potentially singular braid.
Braids rel $\y$ are partitioned by the singular relative braids
$\Sigma\rel\y = \iota_{\y}^{-1}(\Sigma^{n+m})$ into connected
components. These connected components in $\Omega^{n,m}$ are
the relative braid classes $[\x\rel\y]$; in this class, both the strands
of $\x$ and $\y$ are permitted to deform (though without
intersecting themselves or each other). The \style{relative braid
class} $[\x]\rel \y$ is defined as the fiber
$\iota_{\y}^{-1}(\Omega^{n+m})\cap[\x]$; in this fiber, strands of
$\x$ may deform, but not strands of $\y$.

\subsection{The Hamiltonian action}
\label{subsec:action1}
For $H\in \cH$ the Hamiltonian 1-form on $\Sb^1\times \D^2$ is
defined by $\theta = pdq - Hdt $. For a loop $\x\in \Omega^1$ one
defines the \style{action} functional
\begin{equation}
\label{eqn:action1}
 \f_H(x) = \int_0^1 \theta(x(t)) = \int_0^1 \alpha_0 (x_t(t)) dt - \int_0^1 H(t,x(t))dt,
\end{equation}
where $\alpha_0 = pdq$ and $x_t = \frac{dx}{dt}$. The first variation
with respect to variations $\xi\in   C^1$ defines an exact 1-form
$d\f_H$:
\begin{eqnarray*}
d\f_H(x)\xi &=& -\int_0^1 \omega_0\bigl(x_t(t),\xi(t)\bigr)dt - \int_0^1 dH(t,x(t))\xi(t) dt\\
&=& -\int_0^1 \omega_0\Bigl(x_t(t)-X_{t,H}(t,x(t)),\xi(t)\Bigr)dt,
\end{eqnarray*}
which gives  the variational principle
for Eq.\  (\ref{eqn:HE}).
For closed $n$-braids we can extend the Hamiltonian action to
$\bOn$. Given any $\x\in \bOn$ define its action by $\f_{H} (\x) =
\sum_{k=1}^n \f_H(x^k)$. We may  regard $\f_H(\x)$ as the action for
a Hamiltonian system on $(\D^2)^n$, with $\overline\omega_0 =
\omega_0\times \cdots \times \omega_0$, and $\overline H(t,\x) =
\sum_k H(t,x^k)$, i.e.\ an uncoupled system with coupled boundary
conditions  $x^k(t+1) = x^{\sigma(k)}(t)$ for all $t$ and all
$k=1,\cdots,n$, where $\sigma\in S_n$ is a permutation. We abuse
notation by   denoting the action for the product system by $\f_H(\x)$
and the action is well-defined for all $\x \in \bOn\cap C^1$. The
stationary, or critical closed braids, including singular braids, are
denoted by $ \P_H(\bOn) = \bigl\{\x\in \bOn\cap C^1~|~d\f_H(\x)
=0\bigr\}. $ By the boundary conditions  given above, the first
variation of the action yields that the individual strands satisfy Eq.\
(\ref{eqn:HE}). For the critical points of $\f_H$ on $\bOn$ the
following compactness property holds.
\begin{lemma}
\label{lem:br4}
The set $\P_H(\bOn)$ is compact in $\bOn$.
As a matter of fact $\P_H(\bOn)$ is compact in
 the $C^r$-topology for any $r\ge 1$.
\end{lemma}

\begin{proof} From Eq.\  \rmref{eqn:HE} we derive that $|x_t^k|
= |\nabla H(t,x^k)| \le C$, by the assumptions on $H$. Since
$|x^k|\le 1$, for all $k$, we obtain the a priori estimate $ \Vert
\x\Vert_{W^{1,\infty}}\le C, $ which holds for all $\x\in \P_H(\bOn)$.
Via compact embeddings (Arzela-Ascoli) we have that a sequence
$\x_n\in \P_H$ converges in $C^0$, along a subsequence, to a limit
$\x\in\bOn$. Using the equation we obtain the convergence in $C^1$
and $\x$ satisfies the equation with the boundary conditions given
above. Therefore $\x\in \P_H(\bOn)$, thereby establishing the
compactness of $\P_H(\bOn)$ in $\bOn$. The $C^r$-convergence is
achieved by differentiating Eq.\ \rmref{eqn:HE} repeatedly. This
concludes the compactness of $\P_H(\bOn)$ in $C^r$.
\end{proof}

\begin{remark}
\label{br5} The critical braids in $\P_H(\bOn)$ have one additional
property that plays an important role. For the strands $x^k$ of $\x\in
\P_H(\bOn)$ it holds that either $|x^k(t)|=1$, for all $t$, or
$|x^k(t)|<1$, for all $t$. This is a consequence of the uniqueness of
solutions for the initial value problem for \rmref{eqn:HE}. We say that
a braid $\x$ in supported in $\Int(\D^2)$ if $ |x^k(t)|<1, $
 for all $t$  and for all $k$.
\end{remark}

In the same spirit, we can define the subset of stationary braids
restricted to a braid class $[\x]$, notation $\P_H([\x])$, or in the case
of a relative braid class $\P_H([\x]\rel\y)$.


\section{Background: Cauchy-Riemann equations}
\label{sec:CR1}

\subsection{Almost complex structures}
\label{subsec:almost}

The standard almost complex matrix $J_0$ is defined by the relation
$\langle \cdot,\cdot\rangle = \omega_0(\cdot,J_0\cdot)$, where $
\langle \cdot,\cdot\rangle$ is the standard inner product defined by
$dp\otimes dp + dq\otimes dq$. Therefore, $J_0$ defines an  almost
complex structure on $\D^2$ which corresponds to complex
multiplictation with $i$ in $\CC$. In general an almost complex
structure on $\D^2$ is a mapping $J:T\D^2 \to T\D^2$, with the
property that $J^2=-{\rm id}$. An almost complex structure is
compatible with $\omega_0$ if $\omega_0(\cdot,J\cdot)$ defines a
metric on $\D^2$ and $g=\omega_0(\cdot,J\cdot)$ is $J$-invariant.
The space of $t$-families of almost complex structures is denoted by
$\cJ = \cJ(\Sb^1\times\D^2)$.

In terms of the standard inner product $ \langle \cdot,\cdot\rangle$
the metric $g$ is given by $g(\xi,\eta) = \langle -J_0 J
\xi,\eta\rangle$, where $-J_0J$ is a positive definite symmetric
matrix function. With respect to the metric $g$ it holds that $J
\nabla_gH = X_H$.

\subsection{Compactness}
\label{subsec:compactness}

In order to study 1-periodic  solutions of  Eq.\  (\ref{eqn:HE}) the
variational method due to Floer and Gromov explores the perturbed
nonlinear Cauchy-Riemann equations (\ref{eqn:CR1}) which can be
rewritten as
$$
\cre(u) = u_s - J(t,u)\bigl[ u_t - X_{H}(t,u)\bigr] =0,
$$
for functions $u: \R\times \R \to \D^2$ (short hand notation). In the
case of 1-periodic solutions we invoke the boundary conditions
$u(s,t+1)=u(s,t)$.

In order to find closed braids as critical points of $\f_H$ on $\bOn$ we
invoke the CRE for $u^k$ in $\u=\{u^k\}$. A collection of
$C^1$-functions $\u(s,t) =\{u^k(s,t)\}$ is said to satisfy the CRE if its
components $u^k$ satisfy Eq.\ (\ref{eqn:CR1}) for all $k$ and the
periodicity condition
\begin{equation}
\label{e:percond}
  \{u^1(s,t+1),\dots, u^n(s,t+1) \} = \{u^1(s,t) ,\dots, u^n(s,t) \}.
\end{equation}
for all $s,t$. We use the almost complex structure $\overline J$
defined via $\overline g(\cdot,\cdot) = \overline
\omega_0(\cdot,\overline J \cdot)$, $\overline \omega_0 = \omega_0
\times \cdots \times \omega_0$, and $\overline g(\cdot,\cdot) =
\overline \omega_0(\cdot,\overline J \cdot)$, with $\overline H(t,\x)
= \sum_k H(t,x^k)$ as per Sect. \ref{subsec:var1}. The equations
become
\begin{equation}
\label{eqn:CR3}
\partial_{\overline J,\overline H}(\u) =
\u_s - \overline J(t,\u) \bigl[ \u_t - X_{\oH} (t,\u)\bigr] =0,
\end{equation}
where  $X_{\oH}$ is defined via the relation $\iota_{X_{t,\oH}}
\overline \omega_0= -d\oH$ and the periodicity condition in
(\ref{e:percond}). This requirement is fulfilled precisely by braids
$\u(s,\cdot) \in \bOn$ for all $s$. We therefore define the space of
bounded solutions in the space of $n$-braids by:
\begin{equation*}
\M^{J,H} = \M^{J,H}(\bOn) =
\bigl\{\u = \{u^k\}\in C^1\bigl(\R\times \R/\Z;C^n(\D^2)\bigr)
~|~
\partial_{J,H}(u^k) =0,~~\forall k
\bigr\}
\end{equation*}
Note that solutions in $\M^{J,H}$ extend to $C^1$-functions $u^k:
\R\times\R\to \D^2$ by periodic extension in $t$. If there is no
ambiguity about the dependence on $J$ and $H$ we abbreviate
notation by writing  $\M$ instead of $\M^{J,H}$.

The following statement is Floer's compactness theorem
adjusted to the present situation.

\begin{proposition}
\label{prop:CR4} The space $\M^{J,H}$ is compact in the topology of
uniform convergence on compact sets in $(s,t)\in \R^2$, with
derivatives up to any order.  The action $\f_H$ is uniformly bounded
along trajectories $\u\in \M^{J,H}$, and
\begin{eqnarray*}
& &\lim_{s\to\pm \infty} |\f_H(\u(s,\cdot))| = |c_\pm (\u)| \le C(J,H),\\
& &\int_\R\int_0^1 |\u_s|_g^2 dt ds = \sum_{k=1}^n\int_\R\int_0^1 |u_s^k|_g^2 dt ds \le C'(J,H),
\end{eqnarray*}
for all $\u\in \M^{J,H}$ and some  constants $c_\pm(\u)$. The constants
$C,\, C'$ depend only on $J\in \cJ$ and $H\in \cH$.
\end{proposition}

\begin{proof}
Define the operators
$$
\partial_J  = \frac{\partial}{ \partial s} - J \frac{\partial}{ \partial t},\quad
\bar\partial_J  = \frac{\partial}{ \partial s} + J \frac{\partial}{ \partial t}.
$$
Eq.\  \rmref{eqn:CR1} can now be written as $\partial_J u^k =
\nabla_g H(t,u^k) \bydef f^k(s,t)$ for all $k$. By the hypotheses on
$H$ and the fact that $|u^k|\le 1$ for all $k$ we have that $f^k(s,t)
\in L^\infty(\R^2)$.  The latter follows from the fact that the
solutions $u^k$ can be regarded as functions on $\R^2$ via periodic
extension in $t$. Using these crucial a priori estimates, the remainder
follows as in Floer's compactness proof: see \cite{Sal, Sal2}. In brief,
the $L^\infty$-estimates yield $C^{1,\lambda}$-estimates, which
then give the desired compactness. By the smoothness of $J$ and $H$
such estimates can be found in any $C^{r,\lambda}$.

Due to the a priori bound in $C^{1,\lambda}$ it holds that
$|\f_H(\u(s,\cdot))| \le C(J,H)$ and since
$$
\frac{d}{ds} \f_H(\u(s,\cdot)) = -\int_0^1 |\u_s|_g^2 dt\le 0,
$$
it follows  that the limits $\lim_{s\to \pm \infty} \f_H(\u(s,\cdot)) =
c_\pm$ exist and are a priori bounded by the same constant $C(J,H)$.
Finally, for any $T_1,T_2>0$,
\begin{equation*}
\int_{-T_1}^{T_2}\int_0^1 |\u_s|_g^2 dt ds =
\sum_{k=1}^n\int_{-T_1}^{T_2} \int_0^1 |u_s^k|_g^2 dt ds =
\f_H(\u(T_2,\cdot)) -\f_H(\u(-T_1,\cdot)).
\end{equation*}
By the uniform boundedness of the action along all orbits $\u\in
\M_{J,H}$ we obtain the estimate $ \int_\R\int_0^1 |\u_s|_g^2 dt ds
\le 2 C(J,H), $ which completes our proof.
\end{proof}

\begin{remark}
\label{rmk:CR5a} In the above proof we only use $C^2$-regularity of
the Hamiltonian $H$ in order to obtain compactness in $C^1$. Since
we assume Hamiltonians to be $C^\infty$-smooth we can improve the
compactness result to hold up to derivatives of any order.
\end{remark}

\subsection{Additional compactness}
\label{subsec:comp2} Consider the non-autonomous Cauchy-Riemann
equations:
\begin{equation}
\label{eqn:CR2s}
u_s - J(s,t,u) u_t - \nabla_{g_s} H(s,t,u) =0,
\end{equation}
where $s\mapsto J(s,\cdot)$ is a smooth path in $\cJ$ and $s\mapsto
H(s,\cdot,\cdot)$ is a smooth path in $\cH$. Both  paths are assumed
to have the property that the  limits as $s\to \pm \infty$ exists. The
path of metrics $s\mapsto g_s$ is defined via the relation
$g_s(\cdot,\cdot) = \omega_0(\cdot,J(s,\cdot)\cdot)$. Assume that
$|H_s| \le \kappa(s) \to 0$  as $s\to\pm\infty$ uniformly in $(t,x) \in
\R/\Z\times \D^2$, with $\kappa\in L^1(\R)$. For the equation
$\partial_{J} v  = f(s,t)$, the analogue of Proposition \ref{prop:CR4}
holds via the $L^\infty$-estimates on the right hand side, see
\cite{Sal,Sal2}. We sketch the main idea.

Define $s\mapsto \f_H(s,x)$ as the action path with Hamiltonian path
$s\mapsto H(s,\cdot,\cdot)$. The first variation with respect to $s$
can be computed as before:
\begin{eqnarray*}
\frac{d}{ds} \f_H(s,\u(s,\cdot)) &=& \frac{\partial \f_H}{\partial s}
+ \sum_{k=1}^n \int_0^1 \omega_0\bigl( u_t^k-X_H(t,u^k),u^k_s\bigr) dt\\
&=&  \frac{\partial \f_H}{\partial s} -  \sum_{k=1}^n \int_0^1\vert u_t^k-X_H(t,u^k)\vert^2_{g_s} dt\\
&=&  \frac{\partial \f_H}{\partial s} -   \int_0^1\vert \u_s \vert^2_{g_s} dt.
\end{eqnarray*}
The partial derivative with respect to $s$ is given by
$$
 \frac{\partial \f_H}{\partial s}
 =
 \sum_{k=1}^n \int_0^1 \frac{\partial H}{\partial s}\left(s,t,u^k(s,t)\right)dt,
 $$
and $\left| \frac{\partial \f_H}{\partial s}\right| \le C \kappa(s) \to 0$
as $s\to\pm\infty$. For a non-stationary solution $\u$ it holds that
$\int_0^1\vert \u_s \vert^2_{g_s} dt>0$, and thus for $|s|$
sufficiently large $\frac{d}{ds} \f_H(s,\u(s,\cdot)) - \frac{\partial}{\partial s} \f_H<0$, proving
that the limits $\lim_{s\to\pm\infty}   \f_H(s,\u(s,\cdot)) = c_\pm$
exist. Since $\kappa\in L^1(\R)$ we also obtain the integral
$\int_\R\int_0^1 \vert \u_s \vert^2_{g_s} dtds \le C(J,H)$. This
non-autonomous CRE will be used to establish continuation for Floer
homology.


\section{Crossing numbers, a priori estimates and isolation}
\label{sec:dislyap}

\subsection{The crossing number}
\label{subsec:cross}
We begin with an important property of the (linear) Cauchy-Riemann
equations in dimension two. We consider Eq.\ \rmref{eqn:CR1}, or
more generally Equation \rmref{eqn:CR2s}, and local solutions of the
form $u: G\subset \R^2 \to \R^2$, where $G = [\sigma,\sigma']\times
[\tau,\tau']$. For two local solutions $u,u':G\to \R^2$ of
\rmref{eqn:CR1} assume that
$$
u(s,t) \not =u'(s,t),~~~{\rm for~all}~~
(s,t) \in \partial G.
$$
Intersections of $u$ and $u'$, where $u(s_0,t_0)=u'(s_0,t_0)$ for
some $(s_0,t_0)\in G$, have constrained evolutions. Consider the
difference function $w(s,t) = u(s,t) - u'(s,t)$. By the assumptions on
$u$ and $u'$ we have that $w|_{\partial G}\not = 0$, and
intersections are given by $w(s_0,t_0)=0$. The following lemma is a
special feature of CRE in dimension two and is a manifestation of the
well-known positivity of intersection of $J$-holomorphic curves in
almost complex 4-manifolds \cite{McDuff1,McDuff2}.
\begin{lemma}
\label{lem:dislyap1a} Let $u,u'$ and $G$ be as defined above. Assume
that $w(s_0,t_0)=0$ for some $(s_0,t_0)\in G$. Then $(s_0,t_0)$ is
an isolated zero and $\deg(w,G,0)<0$.
\end{lemma}
\begin{proof}
Taylor expand: $\nabla_g H(t,u') = \nabla_g H(t,u) +
R_1(t,u,u'-u)(u'-u)$, where $R_1$ is continuous. Substitution yields
$$
w_s - J(s) w_t - A(s,t) w =0,\quad w(s_0,t_0) =0,
$$
where $A(s,t) = R_1(t,u,-w)$ is continuous on $G$. Define complex
coordinates $z=s-s_0 + i(t-t_0)$. Then by \cite[Appendix A.6]{HZ},
there exists a $\delta<0$, sufficiently small, a disc
$D_\delta=\{z~|~|z|\le \delta\}$, a holomorphic map $h: D_\delta \to
\CC$, and  a continuous mapping $\Phi: D_\delta \to {\rm
GL}_{\R}(\CC)$ such that
$$
\det~\Phi(z)>0,
\qquad J(z) \Phi(z) = \Phi(z) i,\qquad w(z) = \Phi(z) \bar h(z),
$$
for all $z\in D_\delta$. Clearly, $\Phi$ can be represented by
a real $2\times 2$ matrix function of invertible matrices.

Since $w=\Phi \bar h$, it holds that the condition $w(z_0) =0$ implies
that $\bar h(z_0) = h(z_0)=0$. The analyticity of $h$ then implies
that either $z_0$ is an isolated zero in $D_\delta$, or $h\equiv 0$ on
$D_\delta$. If the latter holds, then also $w\equiv 0$ on $D_\delta$.
If we repeat the above arguments we conclude that $w\equiv 0$ on
$G$ (cf. analytic continuation), in contradiction with the boundary
conditions. Therefore, all zeroes of $w$ in $G$ are isolated, and
there are finitely many zeroes $z_i\in \Int(G)$.

For the degree we have that, since $\det \Phi(z)>0$,
$$
\deg(w,G,0) = \sum_{i=1}^m \deg(w,B_{\epsilon_i}(z_i),0)
=\sum_{i=1}^m \deg(\bar h,B_{\epsilon_i}(z_i),0)
= -\sum_{i=1}^m \deg(h,B_{\epsilon_i}(z_i),0),
$$
and for an analytic function with an isolated zero $z_i$ it holds that
$\deg(h,B_{\epsilon_i}(z_i),0) = n_i\ge 1$; thus $\deg(w,G,0)<0$.
\end{proof}

For a curve $\Gamma:I \to \R^2 \setminus \{(0,0)\}$, with $I$ a
bounded interval, one can define the \style{winding number} about
the origin by
$$
  W(\Gamma,0) \bydef  \frac{1}{2\pi}   \int_I \Gamma^* \alpha
  =\frac{1}{2\pi}\int_\Gamma \alpha,
$$
for $\alpha = (-qdp+pdq)/(p^2+q^2)$. In particular, for curves
$w(s,\cdot):[\tau,\tau'] \to \R^2\setminus \{(0,0)\}$ and for
$s=\sigma,\sigma'$ the (local) winding number is
$$
W(w(s,\cdot),0) \bydef \frac{1}{2\pi} \int_{[\tau,\tau']} w^*\alpha = \frac{1}{2\pi} \int_w \alpha.
$$
We denote these winding numbers by $W_\sigma^{[\tau,\tau']}(w)$
and $W_{\sigma'}^{[\tau,\tau']}(w)$ respectively. In the case that
$[\tau,\tau']=[0,1]$ we simply write $W_\sigma(w) \bydef
W_\sigma^{[0,1]}(w)$. Similarly, we have winding numbers for the
curves $w(\cdot,t):[\sigma,\sigma']\to \R^2\setminus\{(0,0)\}$ and for
$t=\tau,\tau'$, which we denote by $W_\tau^{[\sigma,\sigma']}(w)$
and $W_{\tau'}^{[\sigma,\sigma']}(w)$ respectively. These local
winding numbers are related to the degree of the map $w:G\to \R^2$.
\begin{lemma}
\label{lem:dislyap1} Let $u,u': G \to \R^2$ be local solutions of
Equation \rmref{eqn:CR1}, with   $w|_{\partial G}\not = 0$. Then
\begin{equation}
\label{eqn:dislyap2}
    \Bigl[W_{\sigma'}^{[\tau,\tau']}(w) -W_{\sigma}^{[\tau,\tau']}(w)\Bigr] -
    \Bigl[W_{\tau'}^{[\sigma,\sigma']}(w) -W_{\tau}^{[\sigma,\sigma']}(w)\Bigr]
    = \deg(w,G,0).
\end{equation}
In particular, for each zero $(s_0,t_0)\in \Int(G)$, there exists an
$\epsilon_0>0$ such that
$$
W_{s_0+\epsilon}^{[\tau,\tau']}(w) - W_{s_0-\epsilon}^{[\tau,\tau']}(w)
< W_{\tau'}^{[s_0-\epsilon,s_0+\epsilon]}(w) -W_{\tau}^{[s_0-\epsilon,s_0+\epsilon]}(w),
$$
for all $0<\epsilon\le \epsilon_0$.
\end{lemma}
\begin{proof}
We abuse notation by regarding $w$ as a map from the complex plane
to itself. Let the contour $\gamma = \partial G$ be positively
oriented (counterclockwise in the $(s,t)$ plane).
The winding number of the
contour $w(\gamma)$ about $0\in \CC$ in complex notation is given
by
$$
W(w(\gamma),0) = \frac{1}{2\pi i} \oint_{w(\gamma)} \frac{dz}{z} = \deg(w,G,0),
$$
which is equal to the degree of $w: G\to \R^2$ with respect to the
value $0$. Using the special form of the contour $\gamma$ we can
write out the the Cauchy integral using the 1-form $\alpha$:
\begin{eqnarray*}
\frac{1}{2\pi i} \oint_{w(\gamma)} \frac{dz}{z} &=&
\frac{1}{2\pi} \int_{w(\sigma',\cdot)} \alpha -
 \frac{1}{2\pi} \int_{w(\cdot,\tau')} \alpha -
 \frac{1}{2\pi} \int_{w(\sigma,\cdot)} \alpha +
 \frac{1}{2\pi} \int_{w(\cdot,\tau)} \alpha \\
& = &
 \Bigl[W_{\sigma'}^{[\tau,\tau']}(w) -W_{\sigma}^{[\tau,\tau']}(w)\Bigr] -
\Bigl[W_{\tau'}^{[\sigma,\sigma']}(w) -W_{\tau}^{[\sigma,\sigma']}(w)\Bigr],
\end{eqnarray*}
which proves the first statement.

Lemma \ref{lem:dislyap1a}  states that all zeroes of $w$ are isolated
and have negative degree. Therefore, there exists an $\epsilon_0>0$
such that $G_{\epsilon}=[s_0-\epsilon,s_0+\epsilon]\times
[\tau,\tau']$ contains no zeroes on the boundary, for all
$0<\epsilon\le\epsilon_0$, from which the second statement follows.
\end{proof}

On the level of comparing two local solutions of Equation
\rmref{eqn:CR1}, the winding number behaves like a discrete
Lyapunov function with respect to the time variable $s$. This can be
further formalized for solutions of the CRE on $\bOn$. For a closed
braid $\x \in \Omega^n$, one defines the \style{total crossing
number}
\begin{equation}
\Cr(\x) \bydef  \sum_{k , k'} W\bigl( x^k-x^{k'},0 \bigr) =
2 \sum_{\{k,k'\}\atop k\neq k'}W\bigl( x^k-x^{k'},0 \bigr),
\end{equation}
where the second sum is over all unordered pairs $\{k,k'\}$, using the
fact that the winding number is invariant under the inversion $(p,q)
\to (-p,-q)$. The number $\Cr(\x)$ is equal to the total
linking/self-linking number of all components in a closed braid $\x$.
The local winding number as introduced above is not necessarily an
integer. However, for closed curves the winding number is integer
valued. It is clear that the number $\Cr(\x)$ as defined above is also
an integer, one interpretation of which is via the associated braid
diagrams as the \style{algebraic crossing number}:

\begin{lemma}
\label{lem:cross1} The number $\Cr(\x)$ is an integer, and
\begin{eqnarray*}
\Cr(\x) &=& \#\{\hbox{positive crossings}\} - \#\{\hbox{negative
crossings}\}.
\end{eqnarray*}
This is a braid class invariant; i.e., $\Cr(\x)=\Cr(\x')$ for all $\x,\x'\in
[\x]$.
\end{lemma}


This result is standard: we include a self-contained proof.

\begin{proof} The expression for $\Cr(\x)$ is twice the sum of $\tfrac{n!}{2!(n-2)!}$
local winding numbers. On the unordered pairs $\{k,k'\}$ there exists
the following equivalence relation. Two pairs $\{k,k'\}$ and $\{h,h'\}$
are equivalent if for some integer $d\ge 0$, $\{x^k(d),x^{k'}(d)\} =
\{x^h(0),x^{h'}(0)\}$ as unordered pairs. The equivalence classes of
unordered pairs $\{k,k'\}$ are denoted by $\pi_j$ and the number of
elements in $\pi_j$ is denoted by $|\pi_j|$. For each class $\pi_j$
define $w_{\pi_j} = x^k-x^{k'}$, for some representative $\{k,k'\} \in
\pi_j$. For $t\in  [0,2|\pi_j|]$, the functions $w_{\pi_j}(t)$ represent
closed loops in $\R^2$ regardless of the choice of the representative
$\{k,k'\}\in \pi_j$. Namely, note that
$\{x^k(|\pi_j|),x^{k'}(|\pi_j|)\}= \{x^k(0),x^{k'}(0)\}$ as unordered
pairs, which imlies that
\begin{equation}
\label{e:xk}
  x^k(|\pi_j|)=x^k(0)\qquad \text{or}\qquad x^k(|\pi_j|)=x^{k'}(0) .
\end{equation}
For the crossing number we have
\begin{equation}
\label{e:CrW}
\Cr(\x) = 
 2 \sum_j \Bigl( \sum_{\{k,k'\}\in \pi_j} W\bigl( x^k-x^{k'},0 \bigr)\Bigr)
= 2 \sum_j W^{[0,|\pi_j|]}(w_{\pi_j},0) = \sum_j W^{[0,|2\pi_j|]}(w_{\pi_j},0),
\end{equation}
where the (outer) sum is over all equivalence classes $\pi_j$.
For the final equality we have used (\ref{e:xk}) and the invariance
of the winding number under the inversion $w \to -w$.
Since the latter winding numbers are winding numbers for closed
loops about $0$ (linking numbers), they are all integers, and
thus $\Cr(\x)$ is an integer.

As for the expression in terms of positive and negative crossings we
argue as follows. By inspection, $W\bigl(x^k-x^{k'},0 \bigr)$ equals all
positive minus negative crossings between the two strands. The
invariance of $\Cr(\x)$ with respect to $[\x]$ follows from the
homotopy invariance of the winding number.
\end{proof}

Using the representation of the crossing number for a braid in terms
of winding numbers, we can prove a Lyapunov property. Note that
elements $\u$ of $\M$ are not necessarily in $\Omega^n$ for all $s$.
Therefore $\Cr(\u(s,\cdot))$ is only well-defined whenever
$\u(s,\cdot)\in \Omega^n$.

Combining these results leads to the crucial step in setting up a Floer
theory for braid classes.

\begin{lemma}
\label{prop:dislyap3}
{\bf [Monotonicity Lemma]} For $\u\in \M$, $\Cr(\u(s,\cdot))$ is (when well-defined)
non-increasing in $s$. To be more precise, if $u^k(s_0,t_0) =
u^{k'}(s_0,t_0)$ for some $(s_0,t_0)\in \R\times \R/\Z$, and $k\not =
k'$, then either there exists an $\epsilon_0>0$ such that
$$
\Cr(\u(s_0-\epsilon,\cdot)) > \Cr(\u(s_0+\epsilon,\cdot)),
$$
for all $0<\epsilon\le \epsilon_0$, or $u^k \equiv u^{k'}$.
\end{lemma}
\begin{proof}
Given $\u=\{u^k\}\in\M$, $\Cr(\u(s,\cdot))$ is well-defined for all
$s\in\R$ for which $\u(s,\cdot) \in \Omega^n$. As in the proof of
Lemma \ref{lem:cross1} we define $w_{\pi_j}(s,t) =
u^k(s,t)-u^{k'}(s,t)$ for some representative $\{k,k'\}\in \pi_j$. From
the proof of Lemma \ref{lem:dislyap1a} we know that $(s_0,t_0)$ is
either isolated, or $u^k \equiv u^{k'}$. In the case that $(s_0,t_0)$ is
an isolated zero there exists an $\epsilon_0>0$, such that $(s_0,t_0)$
is the only zero in $[s_0-\epsilon,s_0+\epsilon]\times
[t_0-\epsilon,t_0+\epsilon]$, for all $0<\epsilon \le \epsilon_0$. By
periodicity it holds that $w_{\pi_j}(s,t+|\pi_j|) = w_{\pi_j}(s,t)$, for
all $(s,t)\in \R^2$, and therefore $W_{t_0-\epsilon
+|\pi_j|}^{[\sigma,\sigma']}(w_{\pi_j}) =
W_{t_0-\epsilon}^{[\sigma,\sigma']}(w_{\pi_j})$, for any
$\sigma<\sigma'$. From Lemma \ref{lem:dislyap1} it then follows that
$$
W_{s_0-\epsilon}^{[t_0-\epsilon,t_0-\epsilon+|\pi_j|]}(w_{\pi_j})
>
W_{s_0+\epsilon}^{[t_0-\epsilon,t_0-\epsilon+|\pi_j|]}(w_{\pi_j}),
$$
and, since these terms make up the expression for $\Cr(\u(s,\cdot))$
in Equation~(\ref{e:CrW}), we obtain the desired inequality.
\end{proof}

\subsection{A priori bounds}
\label{subsec:aprioiri}
From Lemma \ref{lem:dislyap1} we can also derive the following
a priori estimate for solutions of the Cauchy-Riemann equations.

\begin{proposition}
\label{prop:dislyap4} Let $u: G \to \D^2$ be a local solution of
Equation \rmref{eqn:CR1}, then either
$$
|u(s,t)| =1,\qquad {\rm or}\qquad  |u(s,t)|<1,
$$
for  all $(s,t)\in G$. In particular, solutions $\u \in \M$ have the
property that components $u^k$ either lie entirely  on $\partial
\D^2$, or entirely in the interior of $\D^2$.
\end{proposition}

\begin{proof}
By assumption, the boundary of the disc is invariant for $X_H$
and thus consists of solutions $x(t)$ with $|x(t)|=1$. Assume that
$u(s_0,t_0) = x(t_0)$ for some $(s_0,t_0)$ and some boundary
trajectory $x(t)$. For convenience, we write $u'(s,t) = x(t)$ and we
consider the difference $w(s,t) = u'(s,t)-u(s,t)=x(t)-u(s,t)$. By the
arguments presented in the proof of Lemma~\ref{lem:dislyap1a}, we
know that either all zeroes of $w$ are isolated, or $w \equiv 0$. In
the latter case $u\equiv x$, hence $|u(s,t)| \equiv 1$. Consider the
remaining possibility, namely that $(s_0,t_0)$ is an isolated zero of
$w$, which  leads to a contradiction.

Indeed, choose a rectangle $G = [\sigma,\sigma']\times [\tau,\tau']$
containing $(s_0,t_0)$, such that $w|_{\partial G} \not = 0$. With
$\gamma=\partial G$ positively oriented, we derive from Lemma
\ref{lem:dislyap1} that
$$
W(w(\gamma),0) = \deg(w,G,0)\le -1.
$$
The latter is due to the assumption that $G$ contains a zero. Consider
on the other hand the loops $u(\gamma)$ and $u'(\gamma)$. By
assumption $|(u'-w)(\gamma)|= |u(\gamma)|<|u'(\gamma)|=1$. If
we now apply the `Dog-on-a-Leash' Lemma\footnote{
If two closed planar paths $\Gamma(t)$ (dog) and $\Gamma'(t)$ (walker) satisfy $|\Gamma'(t)-\Gamma(t)|<|\Gamma'(t)-0|$ (i.e., 
the leash is shorter then the walkers distance to the origin)
then $W(\Gamma,0) = W(\Gamma',0)$. } from the theory of winding numbers \cite{Fulton}, we conclude that $ -1\ge W(w(\gamma),0)  = W(u'(\gamma),0) = 0, $
which  contradicts the assumption that $u$ touches $\partial \D^2$.
Hence $|u(s,t)|<1$ for all $(s,t)$.
\end{proof}

As a consequence of this proposition we have following result for
connecting orbit spaces. For $\x_\pm \in \P(\bOn)$, define
$$
\M_{\x_-,\x_+}^{J,H}
=
\bigl\{\u\in \M^{J,H}~|~\lim_{s\to\pm \infty} \u(s,\cdot)
=
\x_\pm\bigr\}.
$$

\begin{corollary}
\label{cor:dislyap5} For $\u\in \M_{\x_-,\x_+}$, with $|\x_\pm|<1$, it
holds that $ |\u(s,t)|<1, $ for all $(s,t)\in \R\times \R/\Z$.
\end{corollary}
This is an isolating property of the connecting orbit spaces.

\subsection{Isolation for proper relative braid classes}
\label{sec:isol-2} In order to assign topological invariants to relative
braid classes we consider proper braid classes as introduced in Sect.\
\ref{sec:intro-floer}. To be more precise:
\begin{definition}
\label{defn:proper-1} A relative braid class $[\x\rel\y]$ is called
\style{proper} if for any fiber $[\x]\rel\y$ it holds that (i):
$|x^k(t)|\not \equiv 1$, and (ii): ${\rm cl}([\x]\rel\y) \cap
(\Sigma_-^n\rel\y)=\varnothing$. The elements of a proper braid class
are called proper braids.
\end{definition}

Under the flow of the Cauchy-Riemann equations, proper braid classes
isolate the set of bounded CRE solutions inside a relative braid class.
Following Floer \cite{Floer1} we define the set of bounded solutions
inside a proper relative braid class $[\x]\rel \y$ by
$$
\M^{J,H}\bigl([\x]\rel \y\bigr) \bydef
\Bigl\{\u\in \M^{J,H}(\bOn)~|~ \u(s,\cdot) \in [\x]\rel \y,
~~\forall~s\in \R\Bigr\}.
$$
We are also interested in the paths traversed (as a function of $s$) by
these bounded solutions in phase space. Hence we define
$$
\cS^{J,H}\bigl([\x]\rel \y\bigr)\bydef
\Bigl\{ \x  = \u(0,\cdot)~|~\u\in \M^{J,H}\bigl([\x]\rel \y\bigr)\Bigr\}.
$$
If there is no ambiguity about the relative braid class we write
$\cS^{J,H}$. Recall that $\M$ carries the $C^r_{\loc}(\R \times
[0,1];\D^2)^n$ topology, while $\cS^{J,H}$ is endowed with the
$C^r([0,1],\D^2)^n$ topology, since $H$ is $C^\infty$.
\begin{proposition}
\label{prop:isol1} For any fiber $[\x]\rel \y$  of a proper relative braid
class $[\x\rel\y]$ the set $\M^{J,H}\bigl([\x]\rel \y\bigr)$ is compact,
and $\cS^{J,H}$ is a compact isolated set in $[\x]\rel \y$, i.e. (i)
$|\u(s,t)|<1$, for all $s,t$ and (ii) $\u(s,\cdot) \cap \Sigma^n\rel\y =
\varnothing$, for all $s$.
\end{proposition}

\begin{proof}
The set $\M^{J,H}\bigl([\x]\rel \y\bigr)$ is contained in the compact
set $\M^{J,H}(\bOn)$ (Proposition \ref{prop:CR4}). Let
$\{\u_m\}\subset  \M^{J,H}\bigl([\x]\rel \y\bigr)$ be a  sequence, then
for any compact interval  $I$, the limit $\u=\lim_{m'\to \infty}
\u_{m'}$ lies in $\M^{J,H}(\bOn)$ and has the property that
$\u(s,\cdot) \in  \cl\bigl([\x]\rel \y\bigr)$, for all $s \in I$. We will
show now that $\u(s,\cdot)$ is in the relative braid class $[\x] \rel
\y$, by eliminating the possible boundary behaviors.

If $|u^k(s_0,t_0)|=1$, for some $(s_0,t_0)$ and $k$, then
Proposition~\ref{prop:dislyap4} implies that $|u^k| \equiv 1$, hence
$|u^k_{m'}| \to 1$ as $m' \to \infty$ uniformly on compact sets in
$(s,t)$. This contradicts the fact that $[\x]\rel \y$ is proper, and
therefore the limit satisfies $|\u|<1$.

If $u^k(s_0,t_0) = u^{k'}(s_0,t_0)$ for some $(s_0,t_0)$ and some
pair $\{k,k'\}$, then by Proposition~\ref{prop:dislyap3} either
$\Cr(\u(s_0-\epsilon,\cdot)) > \Cr(\u(s_0+\epsilon,\cdot))$, for some
$0<\epsilon\le \epsilon_0$, or $u^k\equiv u^{k'}$. The former case
will be dealt with a little later, while in the latter case $\u \in
\Sigma_-^n\rel \y$, contradicting that $[\x]\rel\y$ is proper as before.

If $u^k(s_0,t_0) = y^{\ell}(t_0)$ for some $(s_0,t_0)$ and $k$ and
$y^\ell\in \y$, then by Proposition~\ref{prop:dislyap3} either
$\Cr(\u(s_0-\epsilon,\cdot)\cup \y) > \Cr(\u(s_0+\epsilon,\cdot)\cup
\y)$, for some $0<\epsilon\le \epsilon_0$, or $u^k\equiv y^{\ell}$.
Again, the former case will be dealt with below, while in the latter
case $\u \in \Sigma_-^n\rel \y$, contradicting that $\x\rel\y$ is
proper.

Finally, the two statements about the crossing numbers imply that
both $\u(s_0-\epsilon,\cdot), \u(s_0+\epsilon,\cdot) \in \Omega^n\rel
\y$, and thus $\u(s_0-\epsilon,\cdot), \u(s_0+\epsilon,\cdot)\in
[\x]\rel \y$. On the other hand, since at least one crossing number at
$s_0-\epsilon$ has strictly decreased at $s_0+\epsilon$, the braids
$\u(s_0-\epsilon,\cdot)$ and $\u(s_0+\epsilon,\cdot)$ cannot belong
to the same relative braid class, which is a contradiction. As a
consequence $\u(s,\cdot) \rel \y \in [\x]\rel \y$ for all $s$, which
proves that $\M^{J,H}\bigl([\x]\rel \y\bigr)$ is compact, and therefore
also $\cS^{J,H} \subset [\x]\rel \y$ is compact and isolated in
$[\x]\rel\y$.
\end{proof}


\section{The Maslov index for braids and Fredholm theory}
\label{sec:maslov} The action $\f_H$ defined on $\bOn$ has the
property that stationary braids have a doubly unbounded spectrum,
i.e.,  if we consider the $d^2\f_H(\x)$ at a stationary braid $\x$, then
$d^2\f_H(\x)$ is a self-adjoint operator whose (real) spectrum
consists of isolated eigenvalues unbounded from above or below. The
classical Morse index for stationary braids is therefore not
well-defined. The theory of the Maslov index for Lagrangian subspaces
is used to replace the classical Morse index \cite{Floer1,RobSal2,
RobSal1}, via Fredholm theory.

\subsection{The Maslov index}
Let $(E,\omega)$ be a (real) symplectic vector space of dimension
$\dim E = 2n$, with compatible almost complex structure $J\in
\text{Sp}^+(E,\omega)$.  An $n$-dimensional subspace $V\subset E$
is called  \style{Lagrangian} if $\omega(v,v') =0$ for all $v,v' \in V$.
Denote the space of Lagrangian subspaces of $(E,\omega)$ by
$\cL(E,\omega)$, or $\cL$ for short.

It is well-known that
a subspace $V\subset E$ is Lagrangian if and only if $V =
\text{range}(X)$  for some linear map $X: W \to E$ of rank $n$ and
some $n$-dimensional (real) vector space $W$, with $X$ satisfying
\begin{equation}
\label{eqn:lagr1}
    X^T J X =0,
\end{equation}
where the transpose is defined via the inner product $\langle
\cdot,\cdot \rangle \bydef \omega(\cdot, J \cdot)$.
%

The map $X$ is called a \style{Lagrangian frame} for $V$. If we
restrict to the special case $(E,\omega)  = (\R^{2n},\overline
\omega_0)$, with standard $J_0$, then for a point $x$ in $\R^{2n}$
one can choose symplectic coordinates
$x=(p^1,\cdots,p^n,q^1,\cdots, q^n)$ and the standard symplectic
form is given by $\overline \omega_0 = dp^1\wedge dq^1 +\cdots +
dp^n\wedge dq^n$. In this case a subspace $V\subset \R^{2n}$ is
Lagrangian if $X = \left(\begin{array}{c}P \\
Q\end{array}\right)$, with $P,Q$  $n\times n$ matrices satisfying
$P^T Q = Q^T P$, and $X$ has rank $n$. The condition on $P$ and $Q$
follows immediately from Eq.\ \rmref{eqn:lagr1}.

For any fixed $V\in \cL$, the space $\cL$ can be decomposed into
strata $\Xi_k(V)$:
$$
\cL = \bigcup_{k=0}^n \Xi_k(V).
$$
The strata $\Xi_k(V)$ of Lagrangian subspaces $V'$ which intersect
$V$ in a subspace of dimension~$k$ are submanifolds of co-dimension
$k(k+1)/2$. The \style{Maslov cycle} is defined as
$$
\Xi(V) = \bigcup_{k=1}^n \Xi_k(V).
$$
Let $\Lambda(t)$ be a smooth curve of Lagrangian subspaces and
$X(t)$  a smooth Lagrangian frame for $\Lambda(t)$. A crossing is a
number  $t_0$ such that $\Lambda(t_0) \in \Xi(V)$, i.e., $ X(t_0)w =
v \in V$, for some $w\in W$, $0\neq v\in V$. For a curve $\Lambda: [a,b] \to
\cL$, the set of crossings is compact, and for each crossing $t_0\in
[a,b]$ we can define the crossing form on $\Lambda(t_0)\cap V$:
\begin{eqnarray*}
\Gamma(\Lambda,V,t_0)(v) \bydef \omega\bigl(X(t_0)w,X'(t_0)w\bigr).
\end{eqnarray*}
A crossing $t_0$ is called \style{regular} if $\Gamma$ is a
nondegenerate form. If $\Lambda: [a,b] \to \cL$ is a Lagrangian curve
that has only regular crossings then the \style{Maslov index} of the
pair $(\Lambda,V)$ is defined by
$$
\mu(\Lambda,V) = \frac{1}{2} {\rm sign}~\Gamma(\Lambda,V,a) +
\sum_{a<t_0<b} {\rm sign}~\Gamma(\Lambda,V,t_0) +
\frac{1}{2} {\rm sign}~\Gamma(\Lambda,V,b),
$$
where $\Gamma(\Lambda,V,a)$ and $\Gamma(\Lambda,V,b)$ are
zero when $a$ or $b$ are not crossings. The notation `sign' is the
signature of a quadratic form, i.e. the number of positive minus the
number of negative eigenvalues and the sum is over the crossings
$t_0\in (a,b)$. Since the Maslov index is homotopy invariant and
every path is homotopic to a regular path the above definition
extends to arbitrary continuous Lagrangian paths, using property (iii)
below. In the special case of $(\R^{2n},\overline \omega_0)$ we have
that
\begin{eqnarray*}
\Gamma(\Lambda,V,t_0)(v) &=& \overline \omega_0\bigl(X(t_0)w,X'(t_0)w\bigr)\\
&=& \langle P(t_0)w,Q'(t_0)w\rangle -
\langle P'(t_0)w,Q(t_0)w\rangle.
\end{eqnarray*}
A list of properties of the Maslov index can be found (and is proved) in
\cite{RobSal2}, of which we mention the most important ones:
\begin{enumerate}
\item[(i)] for any $\Psi\in \Sp(E)$, $\mu(\Psi \Lambda,\Psi V) =
    \mu(\Lambda,V)$;\footnote{This property shows that we can
    assume $E$ to be the standard symplectic space without loss of
    generality.}
\item[(ii)] for $\Lambda: [a,b] \to \cL$ it holds that $\mu(\Lambda,V) =
    \mu(\Lambda|_{[a,c]},V)+\mu(\Lambda|_{[c,b]},V)$, for any
    $a<c<b$;
\item[(iii)]  two paths $\Lambda_0,\Lambda_1: [a,b] \to \cL$ with
    the same end points are homotopic if and only if
 $\mu(\Lambda_0,V) = \mu(\Lambda_1,V)$;
\item[(iv)] for any path $\Lambda: [a,b] \to \Xi_k(V)$ it holds that
    $\mu(\Lambda,V) =0$.
\end{enumerate}

The same can be carried out for pairs of Lagrangian curves
$\Lambda,\Lambda^\dagger: [a,b] \to \cL$. The crossing form on
$\Lambda(t_0)\cap \Lambda^\dagger (t_0)$ is then given by
$$
\Gamma(\Lambda,\Lambda^\dagger ,t_0)
\bydef
\Gamma(\Lambda,\Lambda^\dagger(t_0),t_0)
- \Gamma(\Lambda^\dagger,\Lambda(t_0),t_0).
$$
For pairs $(\Lambda,\Lambda^\dagger)$ with only regular crossings
the Maslov index $\mu(\Lambda,\Lambda^\dagger)$ is defined in the
same way as above using the crossing form for Lagrangian pairs. By
setting $\Lambda^\dagger(t) \equiv V$ we retrieve the previous case,
and $\Lambda(t) \equiv V$ yields $\Gamma(V,\Lambda^\dagger,t_0) =
-\Gamma(\Lambda^\dagger,V,t_0)$. Consider the symplectic space
$(\overline{E},\overline{\omega}) = (E\times E, (-\omega) \times
\omega)$, with almost complex structure $(-J) \times J$. A~crossing
$\Lambda(t_0)\cap \Lambda^\dagger(t_0) \not = \varnothing$ is
equivalent to a crossing
 $(\Lambda \times \Lambda^\dagger)(t_0)\in\Xi(\Delta)$,
where $\Delta \subset \overline{E}$ is the diagonal Lagrangian plane,
and $\Lambda\times\Lambda^\dagger$ a Lagragian curve in
$\overline{E}$,
which follows from Equation \rmref{eqn:lagr1} using the Lagrangian
frame $\overline{X}(t) =\left(\begin{array}{c}X(t)
\\X^\dagger(t)\end{array}\right)$. Let $\overline{v} = (v,v)
=\overline{X}(t_0) w$, then
\begin{eqnarray*}
\Gamma(\Lambda\times\Lambda^\dagger,\Delta,t_0)(\overline{v}) &=&
\overline{\omega}\bigl(\overline{X}(t_0)w,\overline{X'}(t_0)w\bigr)\\
&=& -\omega\bigl({X}(t_0)w,{X'}(t_0)w\bigr) + \omega\bigl({X^\dagger}(t_0)w,{X^\dagger}'(t_0)w\bigr)\\
&=& -\Gamma(\Lambda,\Lambda^\dagger(t_0),t_0)(v)
+ \Gamma(\Lambda^\dagger,\Lambda(t_0),t_0)(v).
\end{eqnarray*}
This justifies the identity
\begin{equation}
\label{eqn:prod1}
\mu(\Lambda,\Lambda^\dagger) = \mu(\Delta,\Lambda\times\Lambda^\dagger).
\end{equation}
Equation \rmref{eqn:prod1} is used to define the Maslov index for
continuous pairs of Lagrangian curves, and is a special case of the
more general formula below. For $\Psi: [a,b] \to \Sp(E)$ a symplectic
curve,
\begin{equation}
\label{eqn:prod2}
\mu(\Psi\Lambda,\Lambda^\dagger) = \mu({\rm gr}(\Psi),\Lambda\times\Lambda^\dagger),
\end{equation}
where ${\rm gr}(\Psi) = \{(x,\Psi x)~|~x\in E \}$ is the graph of $\Psi$.
The curve ${\rm gr}(\Psi)(t)$ is a Lagrangian curve in
$(\overline{E},\overline{\omega})$ and $X_\Psi(t) =
\left(\begin{array}{c}{\rm Id} \\\Psi(t)\end{array}\right)$ is a
Lagrangian frame for ${\rm gr}(\Psi)$. Indeed, via \rmref{eqn:lagr1}
we have
$$
\left(\begin{array}{cc}
{\rm Id} & \Psi^T(t)
\end{array}\right)
\left(\begin{array}{cc}
-J & 0 \\0 & J
\end{array}\right)
\left(\begin{array}{c}
{\rm Id} \\\Psi(t)
\end{array}\right)
= \Psi^T(t) J \Psi(t) -J =0,
$$
which proves that ${\rm gr}(\Psi)(t)$ is a Lagrangian curve in
$\overline{E}$. Via $\overline{E}\times \overline{E}$ the crossing form
is given by
\begin{equation*}
\Gamma\left({\rm gr}(\Psi),\Lambda\times\Lambda^\dagger,t_0\right)=
 \Gamma\left({\rm gr}(\Psi),(\Lambda\times\Lambda^\dagger)(t_0),t_0\right)
- \Gamma\left(\Lambda\times\Lambda^\dagger,{\rm gr}(\Psi)(t_0),t_0\right).
\end{equation*}
and upon inspection consists of the three terms making up the
crossing form of $(\Psi\Lambda,\Lambda^\dagger)$ in $\overline{E}$.
More specifically, let $\xi = X_\Psi(t_0)\xi_0 =
\overline{X}(t_0)\eta_0=\eta$, so that $\Psi X\eta_0=\Psi \xi_0 =
X^\dagger \eta_0$, which yields
\begin{eqnarray*}
\Gamma\left({\rm gr}(\Psi),(\Lambda\times\Lambda^\dagger)(t_0),t_0\right)(\xi)
&=& \omega(\Psi(t_0)\xi_0,\Psi'(t_0)\xi_0)\\
&=& \omega(\Psi(t_0)X(t_0)\eta_0,\Psi'(t_0)X(t_0)\eta_0),
\end{eqnarray*}
and
\begin{eqnarray*}
\lefteqn{\Gamma\left(\Lambda\times\Lambda^\dagger, {\rm gr}(\Psi)(t_0),t_0\right)(\eta)}\\
&=& -\omega\bigl({X}(t_0)\eta_0,{X'}(t_0)\eta_0\bigr) +
  \omega\bigl({X^\dagger}(t_0)\eta_0,{X^\dagger}'(t_0)\eta_0\bigr) \\
&=& -\omega\bigl(\Psi(t_0){X}(t_0)\eta_0,\Psi(t_0){X'}(t_0)\eta_0\bigr)
+ \omega\bigl({X^\dagger}(t_0)\eta_0,{X^\dagger}'(t_0)\eta_0\bigr)
\end{eqnarray*}
which proves Equation \rmref{eqn:prod2}. The crossing form for a
more  general Lagrangian pair $({\rm gr}(\Psi),\Lambda)$, where
$\overline{\Lambda}(t)$ is a Lagrangian curve in $\overline{E}$, is
given by $\Gamma\left({\rm gr}(\Psi),\overline{\Lambda},t_0\right)$
as described above. In the special case that
$\overline{\Lambda}(t)\equiv V \times V$, then
\[
\Gamma\left({\rm gr}(\Psi),\overline{\Lambda},t_0\right)(\overline{v})
=
\omega(\Psi(t_0)w,\Psi'(t_0)w),
\]
where $\overline{v} = X_\Psi(t_0) w$.

A particular example of the Maslov index for symplectic paths is the
\style{Conley-Zehnder index} on $(E,\omega) = (\R^{2n},\overline
\omega_0)$, which is defined as $\mu_{CZ}(\Psi) \bydef \mu\left({\rm
gr}(\Psi),\Delta\right)$ for paths $\Psi: [a,b] \to \Sp(2n,\R)$, with
$\Psi(a)  = {\rm Id}$ and ${\rm Id}-\Psi(b)$ invertible. It holds that $
\Psi'  = \overline J_0 K(t)\Psi$, for some smooth path $t\mapsto K(t)$
of symmetric matrices. An intersection of ${\rm gr}(\Psi)$ and
$\Delta$ is equivalent to the condition $\det(\Psi(t_0) - {\rm Id}) =0$,
i.e. for $\xi_0 \in \ker~(\Psi(t_0) - {\rm Id})$ it holds that
$\Psi(t_0)\xi_0 = \xi_0$. The crossing form is given by
\begin{eqnarray*}
\Gamma\left({\rm gr}(\Psi),\Delta,t_0\right)(\overline{\xi}_0)
&=& \overline \omega_0(\Psi(t_0)\xi_0,\Psi'(t_0)\xi_0)\\
&=& \langle \Psi(t_0)\xi_0,K(t_0)\Psi(t_0)\xi_0\rangle \\
&=& \langle \xi_0,K(t_0)\xi_0\rangle.
\end{eqnarray*}
In the case of a symplectic path $\Psi: [0,\tau]\to \Sp(2n,\R)$, with
$\Psi(0) = {\rm Id}$, the extended Conley-Zehnder index is defined as
$\mu_{CZ}(\Psi,\tau) = \mu({\rm gr}(\Psi),\Delta)$.

\subsection{The permuted Conley-Zehnder index}
We now define a variation on the Conley-Zehnder index suitable for
the application to braids. Consider the symplectic space
$$
 E = \R^{2n}\times \R^{2n},
 \qquad
 \omega = (-\overline \omega_0)\times\overline \omega_0.
$$
In $E$ we choose coordinates $(x,\tilde x)$, with $x =
(p^1,\cdots,p^n,q^1,\cdots, q^n)$ and $ \tilde x=(\tilde
p^1,\cdots,\tilde p^n,\tilde q^1,\cdots, \tilde q^n)$ both in
$\R^{2n}$. Let $\sigma\in S_n$ be a permutation, then the permuted
diagonal $\Delta_\sigma$ is defined by:
\begin{equation}
\label{e:sigmapq}
  \Delta_\sigma
  \bydef
  \bigl\{ (x,\tilde x)~|~ (\tilde p^k,\tilde q^k)
  =
  (p^{\sigma(k)},q^{\sigma(k)}),~~1\le k\le n\bigr\}.
\end{equation}
It holds that $\Delta_\sigma = {\rm gr}(\bm{\sigma})$, where
$\bm{\sigma} = \left(\begin{array}{cc}\sigma & 0 \\0 &
\sigma\end{array}\right)$ and the permuted diagonal $\Delta_\sigma$
is a Lagrangian subspace of $E$. Let  $\Psi: [0,\tau] \to \Sp(2n,\R)$ be
a symplectic path with $\Psi(0) = {\rm Id}$. A crossing $t=t_0$ is
defined by the condition $\ker~(\Psi(t_0) -\bm{\sigma}) \not = \{0\}$
and the crossing form is given by
 \begin{eqnarray}
\Gamma\left({\rm gr}(\Psi),\Delta_\sigma,t_0\right) (\xi_0^\sigma)
&=& \overline \omega_0(\Psi(t_0)\xi_0,\Psi'(t_0)\xi_0) \nonumber\\
&=& \langle \Psi(t_0)\xi_0,K(t_0)\Psi(t_0)\xi_0\rangle \nonumber\\
&=& \langle \bm{\sigma}\xi_0,K(t_0)\bm{\sigma}\xi_0\rangle = \langle
\xi_0,\bm{\sigma}^T K(t_0)\bm{\sigma}\xi_0\rangle, \label{e:crossK}
\end{eqnarray}
where $\xi_0^\sigma = X_\sigma\xi_0$, and $X_\sigma$ the frame for
$\Delta_\sigma$. The \style{permuted Conley-Zehnder index} is
defined as
\begin{equation}
\label{eqn:pCZ}
    \mu_\sigma(\Psi,\tau) \bydef \mu({\rm gr}(\Psi),\Delta_\sigma).
\end{equation}

Based on the properties of the Maslov index the following list of basic
properties of the index $\mu_\sigma$ can be derived.
\begin{lemma}
\label{lem:mas1} For $\Psi: [0,\tau]\to \Sp(2n,\R)$ a symplectic
path with $\Psi(0)=\textup{Id}$,  
\begin{enumerate}
\item[(i)] $ \mu_\sigma(\Psi \times \Psi^\dagger,\tau) =
    \mu_\sigma(\Psi,\tau) +  \mu_\sigma(\Psi^\dagger,\tau)$;
\item[(ii)] let $\Phi_k(t): [0,\tau]\to \Sp(2n,\R)$ be a symplectic
    loop (rotation) given by $\Phi_k(t) = e^{\frac{2\pi k}{\tau}
    \overline J_0 t}$, then $ \mu_\sigma(\Phi_k\Psi,\tau) =
    \mu_\sigma(\Psi,\tau) + 2kn$,
\end{enumerate}
\end{lemma}

\begin{proof}
Property (i) follows from the fact that the equations for the crossings
uncouple.
As for (ii), consider the symplectic curves (using $\Psi(0)=\text{Id}$)
$$
\Psi_0(t) =   \begin{cases}
   \Phi_k(t)\Psi(t)   & \text{ } t\in [0,\tau]\\
      \Psi(\tau) & \text{} t\in [\tau,2\tau],
\end{cases}\quad \Psi_1(t) = \begin{cases}
 \Phi_k(t)     & \text{ } t\in [0,\tau] \\
    \Psi(t-\tau)  & \text{ } t\in [\tau,2\tau].
\end{cases}
$$
The curves $\Psi_0$ and $\Psi_1$ are homotopic via the homotopy
$$
\Psi_\lambda(t) = \begin{cases}
   \Phi_k(t)\Psi((1-\lambda) t)   & \text{ } t\in [0,\tau]\\
      \Psi\bigl(\tau +\lambda (t-2\tau)\bigr) & \text{} t\in [\tau,2\tau],
\end{cases}
$$
with $\lambda \in [0,1]$, and $\mu_\sigma(\Psi_0,2\tau) =
\mu_\sigma(\Psi_1,2\tau) $. By definition of $\Psi_0$ it follows that
$\mu_\sigma(\Phi_k\Psi,\tau)= \mu_\sigma(\Psi_0,2\tau)$. Using
property (iii) of the Maslov index above, we obtain
\begin{eqnarray*}
\mu_\sigma( \Phi_k\Psi,\tau)
&=&
\mu_\sigma(\Psi_0,2\tau) =  \mu_\sigma(\Psi_1,2\tau) = \mu\left({\rm gr}(\Psi_1),\Delta_\sigma\right) \\
&=&
\mu\left({\rm gr}(\Phi_k)|_{[0,\tau]},\Delta_\sigma\right) +
\mu\left({\rm gr}(\Psi(t-\tau))|_{[\tau,2\tau]},\Delta_\sigma\right) \\
&=&
\mu\left({\rm gr}(\Phi_k)|_{[0,\tau]},\Delta_\sigma\right)+\mu_\sigma(\Psi,\tau).
\end{eqnarray*}
It remains to evaluate $\mu\left({\rm
gr}(\Phi_k)|_{[0,\tau]},\Delta_\sigma\right)$. Recall from
\cite{RobSal2}, Remark 2.6, that for a Lagrangian loop $\Lambda(t+1)
= \Lambda(t)$ and any Lagrangian subspace $V$ the Maslov index is
given by
$$
\mu(\Lambda,V)
=
\frac{\alpha(1)-\alpha(0)}{\pi},
\quad
\det\left( P(t) + iQ(t)\right) = e^{i\alpha(t)},
$$
where $X= (P,Q)^t$ is a unitary Lagrangian frame for $\Lambda$. In
particular, the index of the loop is independent of the Lagrangian
subspace $V$. From this we derive that
$$
\mu\left({\rm gr}(\Phi_k)|_{[0,\tau]},\Delta_\sigma\right)  =
\mu\left({\rm gr}(\Phi_k)|_{[0,\tau]},\Delta\right),
$$
and the latter is computed as follows. Consider the crossings of
$\Phi_k$: $\det\left( e^{{2\pi k/\tau}\overline J_0 t_0} - {\rm
Id}\right) =0$, which holds for $t_0 = {\tau n/k}$, $n=0,\cdots, k$.
Since $\Phi_k$ satisfies $\Phi_k' = {2\pi k/\tau} \overline J_0 \Phi_k$,
the crossing form is given by $\Gamma\left( {\rm
gr}(\Phi_k),\Delta,t_0\right)\xi_0 = \left<\xi_0,{2\pi k/\tau}
\xi_0\right> = {2\pi k/\tau}|\xi_0|^2$, with $\xi_0\in
\ker\left(\Psi(t_0) -{\rm Id}\right) \not = \{0\}$, and ${\rm
sign}~\Gamma\left( {\rm gr}(\Phi_k),\Delta,t_0\right) = 2n$ (the
dimension of the kernel is $2n$). From this we derive that
$\mu\left({\rm gr}(\Phi_k)|_{[0,\tau]},\Delta\right) = 2kn$ and
consequently $\mu\left({\rm
gr}(\Phi_k)|_{[0,\tau]},\Delta_\sigma\right) = 2kn$.
\end{proof}

\subsection{Fredholm theory and the Maslov index for closed braids}
The main result of this section concerns the relation between the
permuted Conley-Zehnder  index $\mu_\sigma$ and the Fredholm
index of the linearized Cauchy-Riemann operator
$$
 \partial_{K,\Delta_\sigma}
 =
 \frac{\partial}{\partial s}  - \overline J_0\frac{\partial}{\partial t} - K(s,t),
$$
where $K(s,t)$ is a family of symmetric
$2n\times 2n$ matrices parameterized by $\R\times \R/\Z$ and the
matrix $\overline J_0$ is standard.
The operator $\partial_{K,\Delta_\sigma}$ acts on functions satisfying
the non-local boundary conditions $\bigl(\xi(s,0),\xi(s,1)\bigr) \in
\Delta_\sigma$, or in other words $\xi(s,1) = \bm{\sigma} \xi(s,0)$.
On $K$ we impose the following hypotheses:
\begin{enumerate}
    \item[(k1)] there exist continuous functions $K_\pm: \R/\Z \to
        M(2n,\R)$
        such that $\lim_{s\to\pm \infty} K(s,t) = K_\pm(t)$, uniformly
        in $t\in [0,1]$;
    \item[(k2)] the solutions $\Psi_\pm$ of the initial value problem
        $$
            \frac{d}{dt} \Psi_\pm - \overline J_0 K_\pm(t) \Psi_\pm =0,
            \quad \Psi_\pm(0) = {\rm Id},
        $$
        have the property that ${\rm gr}\bigl(\Psi_\pm (1)\bigr)$ is
        transverse to $\Delta_\sigma$.
\end{enumerate}
Hypothesis (k2) can be rephrased as
$\det\bigl(\Psi_\pm(1)-\bm{\sigma}\bigr) \not = 0$. It follows from the
proof below that this is equivalent to saying that the mappings
$L_\pm = \overline J_0\frac{d}{dt} + K_\pm(t)$ are invertible.

In \cite{RobSal1} the following result was proved. Define the function
spaces
\begin{eqnarray*}
W^{1,2}_\sigma([0,1];\R^{2n})
&\bydef&
\bigl\{ \eta \in W^{1,2}(  [0,1])
~|~\bigl(\eta(0),\eta(1)\bigr) \in \Delta_\sigma\bigr\}\\
W^{1,2}_\sigma(\R\times [0,1];\R^{2n})
&\bydef&
\bigl\{ \xi \in W^{1,2}(\R\times [0,1]) ~|~
\bigl(\xi(s,0),\xi(s,1)\bigr) \in \Delta_\sigma\bigr\}.
\end{eqnarray*}
\begin{proposition}
\label{prop:mas2} Suppose that Hypotheses (k1) and (k2) are
satisfied. Then the operator $\bar\partial_{K,\Delta_\sigma}:
W^{1,2}_{\sigma} \to L^2$ is Fredholm and the Fredholm index
is given by
$$
    \ind ~\partial_{K,\Delta_\sigma} =  \mu_\sigma(\Psi_-,1) -  \mu_\sigma(\Psi_+,1).
$$
As a matter of fact $\bar\partial_{K,\Delta_\sigma} $ is a Fredholm
operator  from $W^{1,p}_\sigma$ to $L^p$, $1<p<\infty$, with the
same Fredholm index.
\end{proposition}
\begin{proof}
In \cite{RobSal1} this result is proved that under Hypotheses (k1) and
(k2) on the operator $\partial_{K,\Delta_\sigma} $. We will sketch the
proof  adjusted to the special situation here. Regard the linearized
Cauchy-Riemann operator as an unbounded operator
$$
    D_L =  \frac{d}{ds} - L(s),
$$
on $L^2\bigl(\R;L^2([0,1];\R^{2n})\bigr)$, where $L(s) = \overline
J_0\frac{d}{dt} + K(s,t)$ is a family of unbounded, self-adjoint
operators on $L^2([0,1];\R^{2n})$, with (dense) domain
$W^{1,2}_\sigma([0,1];\R^{2n})$. In this special case the result
follows from the spectral flow of $L(s)$: for the path $s\mapsto L(s)$
a number $s_0\in \R$ is a crossing  if ${\rm ker} ~L(s) \not = \{0\}$. On
${\rm ker} ~L(s) $ we have the crossing form
$$
    \Gamma(L,s_0)\xi \bydef ( \xi, L'(s) \xi )_{L^2} =
    \int_0^1 \Bigl\langle \xi(t),\frac{\partial K(s,t)}{\partial s} \xi(t)\Bigr\rangle dt,
$$
with $\xi\in {\rm ker} ~L(s)$. If the path $s\mapsto L(s)$ has only
regular crossings --- crossings for which $\Gamma$ is non-degenerate
--- then the main result in \cite{RobSal1} states that $D_L$ is
Fredholm with
$$
     \ind~D_L = -\sum_{s_0} \text{sign}~\Gamma(L,s_0) \bydef -\mu_{\rm spec}(L).
$$

Let $\Psi(s,t)$ be the solution of the $s$-parametrized family of ODEs
$$
  \left\{ \begin{array}{l} L(s)\Psi(s,t) = 0, \\ \Psi(s,0)=\text{Id}.
   \end{array} \right.
$$
Note that $\xi\in \ker~L(s)$ if and only if $\xi(t) =\Psi(s,t)\xi_0$ and
$\Psi(s,1)\xi_0=\bm{\sigma}\xi_0$, i.e., $\xi_0\in
\ker(\Psi(s,1)-\bm{\sigma})$. The crossing form for $L$ can be related
to the crossing form for $({\rm gr}(\Psi),\Delta_\sigma)$. We have
that $L(s) \Psi(s,\cdot) =0$ and thus by differentiating
$$
    \frac{\partial K(s,t)}{\partial s} \Psi(s,t) + K(s,t)\frac{\partial \Psi(s,t)}{\partial s} =
    -\overline J_0 \frac{\partial^2\Psi(s,t)}{\partial s\partial t}.
$$
From this we derive
\begin{eqnarray*}
    \lefteqn{-  \Bigl\langle \Psi(s,t)\xi_0,\frac{\partial K(s,t)}{\partial s} \Psi(s,t)\xi_0\Bigr\rangle }\\
    & & = \Bigl\langle \Psi(s,t)\xi_0,K(s,t)\frac{\partial \Psi(s,t)}{\partial s} \xi_0\Bigr\rangle
    +  \Bigl\langle \Psi(s,t)\xi_0,\overline J_0\frac{\partial^2 \Psi(s,t)}{\partial s\partial t} \xi_0\Bigr\rangle\\
    & & = \Bigl\langle K(s,t)\Psi(s,t)\xi_0,\frac{\partial \Psi(s,t)}{\partial s} \xi_0\Bigr\rangle
    +  \Bigl\langle \Psi(s,t)\xi_0,\overline J_0\frac{\partial^2 \Psi(s,t)}{\partial s\partial t} \xi_0\Bigr\rangle\\
    & & = - \Bigl\langle \overline J_0\frac{\partial \Psi(s,t)}{\partial t} \xi_0,\frac{\partial\Psi(s,t)}{\partial s} \xi_0\Bigr\rangle
    +  \Bigl\langle \Psi(s,t)\xi_0,\overline J_0\frac{\partial^2 \Psi(s,t)}{\partial s\partial t} \xi_0\Bigr\rangle,
\end{eqnarray*}
which yields that
$$
    - \Bigl\langle \Psi(s,t)\xi_0,\frac{\partial K(s,t)}{\partial s} \Psi(s,t)\xi_0\Bigr\rangle
    = \frac{\partial}{\partial t} \Bigl\langle \Psi(s,t)\xi_0,\overline J_0\frac{\partial \Psi(s,t)}{\partial s} \xi_0\Bigr\rangle.
$$
We substitute this identity in the integral crossing form for $L(s)$ at a
crossing $s=s_0$:
\begin{eqnarray*}
    \Gamma(L,s_0)(\xi) &=&
     \int_0^1 \Bigl\langle \xi(t),\frac{\partial K(s,t)}{\partial s} \xi(t)\Bigr\rangle dt\\
    &=& \int_0^1 \Bigl\langle \Psi(s,t)\xi_0,\frac{\partial K(s,t)}{\partial s} \Psi(s,t)\xi_0\Bigr\rangle dt\\
    &=& -\Bigl\langle \Psi(s,t)\xi_0,\overline J_0\frac{\partial \Psi(s,t)}{\partial s} \xi_0\Bigr\rangle\Bigl|_0^1
    = -\Bigl\langle \Psi(s,1)\xi_0,\overline J_0\frac{\partial \Psi(s,1)}{\partial s} \xi_0\Bigr\rangle\\
    &=&  \overline \omega_0\Bigl( \Psi(s,1)\xi_0,\frac{\partial \Psi(s,1)}{\partial s} \xi_0\Bigr)
    =  \Gamma\bigl({\rm gr}(\Psi(s,1),\Delta_\sigma,s_0\bigr)(\xi_0^\sigma).
\end{eqnarray*}
The boundary term at $t=0$ is zero since $\Psi(s,0) = {\rm Id}$ for all
$s$. The relation between the crossing forms proves that the curves
$s\mapsto L(s)$ and $s\mapsto \Psi(s,1)$ have the same crossings and
$\mu({\rm gr}(\Psi(s,1)),\Delta_\sigma) =  \mu_{\rm spec}(L)$. We
assume that $\Psi(\pm T,t) = \Psi_\pm(t)$, and that the crossings
$s=s_0$ are regular, as the general case follows from homotopy
invariance. The symplectic path along the boundary of the cylinder
$[-T,T]\times \R/\Z \subset \R\times \R/\Z$ yields
$$
-\mu(\Delta,\Delta_\sigma) - \mu_\sigma(\Psi_+,1) + \mu_{\rm spec}(L) + \mu_\sigma(\Psi_-,1)
= 0 .
$$
\begin{figure}[hbt]
\label{fig:fig-contour}
\begin{center}
\includegraphics[width=6in]{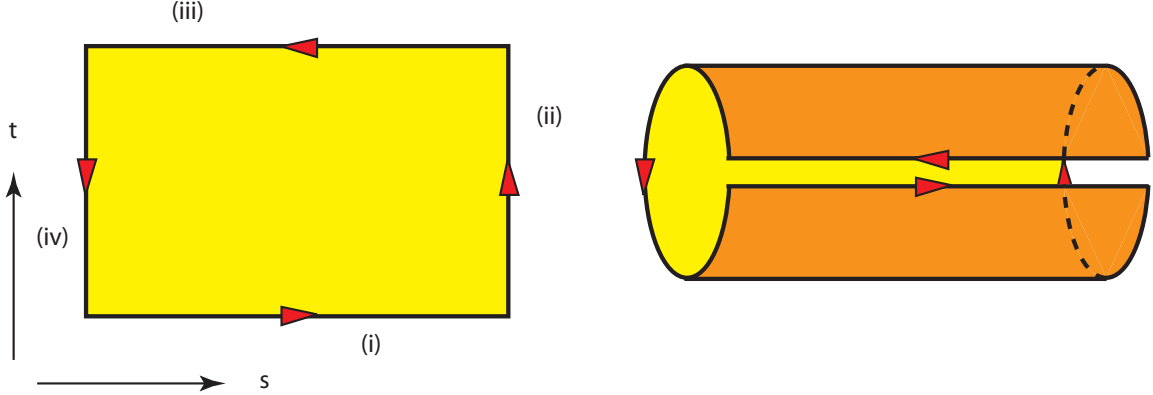}
\caption{The symplectic contour in $\R^2$ and as cylinder $[-T,T]\times \R/\Z$.}
\end{center}
\end{figure}
Indeed, since the loop is contractible the sum of the terms is zero.
The individual terms along the boundary components are found as
follows, see Figure \ref{fig:fig-contour}: (i) for $-T\le s\le T$, it holds
that $\Psi(s,0)= {\rm Id}$, and thus ${\rm gr}(\Psi(s,0)) = \Delta$ and
$\mu({\rm gr}(\Psi(s,0)),\Delta_\sigma) =
\mu(\Delta,\Delta_\sigma)$; (ii) for $0\le t\le 1$, we have $\Psi(T,t) =
\Psi_+(t)$, and therefore $\mu({\rm gr}(\Psi_+),\Delta_\sigma)
=\mu_\sigma(\Psi_+,1)$; (iii) for $-T\le s \le T$ (opposite direction)
the previous calculations with the crossing form for $L(s)$ show that
$\mu({\rm gr}(\Psi(s,1)),\Delta_\sigma) =
 \mu_{\rm spec}(L)$; (iv) for $0\le t\le 1$ (opposite
direction), it holds that $\Psi(-T,t) = \Psi_-(t)$, and therefore
$\mu({\rm gr}(\Psi_-),\Delta_\sigma) =\mu_\sigma(\Psi_-,1)$. Since
${\rm ind}~ D_L = -\mu_{\rm spec}(L)$ we obtain
$$
    \ind~D_L=  \ind ~\partial_{K,\Delta_\sigma} =
   \mu_\sigma( \Psi_-,1) -  \mu_\sigma( \Psi_+,1)
  + \mu(\Delta,\Delta_\sigma).
$$
Since $\Delta_\sigma$ and $\Delta$ are both constant Lagrangian
curves,
it follows that $\mu(\Delta,\Delta_\sigma)=0$, which concludes the
proof of the Theorem.
\end{proof}

We recall from Section \ref{sec:CR1} that the Hamiltonian for
multi-strand braids is defined as $\overline{H}(t,\x(t)) = \sum_{k=1}^n
H(t,x^k(t))$. The linearization  around a braid $\x$ is given by
\begin{equation}
\label{e:Lx}
  L_\x \bydef - d^2 \f_H(\x) = \overline J_0
\frac{d}{dt} + d^2\overline{H}(t,\x) .
\end{equation}
Define the symplectic path $\Psi: [0,1] \to \Sp(2n,\R)$ by
\begin{equation}
\label{e:Psix}
\frac{d \Psi}{dt}  - \overline J_0 d^2\overline{H}(t,\x(t)) \Psi=0,\qquad \Psi(0)= \text{Id}.
\end{equation}
For convenience we write $K(t) =  d^2\overline{H}(t,\x(t))$, so that
the linearized equation becomes $\frac{d}{dt} \Psi - \overline J_0 K(t)
\Psi =0$.

\begin{lemma}
If  $\det\bigl(\Psi(1)-\bm{\sigma}\bigr) \neq 0$, then
$\mu_\sigma(\Psi,1)$ is an integer.
\end{lemma}
\begin{proof}
Since crossings between ${\rm gr}(\Psi)$ and $\Delta_\sigma$ occur
when $\det\bigl(\Psi(t)-\bm{\sigma}\bigr) = 0$, the only endpoint that
may lead to a non-integer contribution is the starting point. There the
crossing form is, as in~(\ref{e:crossK}), given by
$$
  \Gamma({\rm gr}(\Psi),\Delta_\sigma,0)(\xi^\sigma)
      =\langle\xi,\bm{\sigma}^T K(0)\bm{\sigma}\xi\rangle,
$$
for all $\xi \in \ker ( \Psi(0)-\bm{\sigma} )$. The kernel of
$\Psi(0)-\bm{\sigma}=\text{Id}-\bm{\sigma}$ is even dimensional,
since in coordinates~(\ref{e:sigmapq}) it is of the form $
  \ker (\text{Id}_{2n}-\bm{\sigma}) = \ker (\text{Id}_{n}-\sigma) \times
  \ker (\text{Id}_{n}-\sigma).
$ Therefore, $\text{sign}~ \Gamma({\rm gr}(\Psi),\Delta_\sigma,0)$ is
always even, and $\mu_\sigma( \Psi,1)$ is an integer.
\end{proof}

The non-degeneracy condition leads  to an integer valued
Conley-Zehnder index for braids.
\begin{definition}
\label{defn:mas3} A stationary braid $\x$ is said to be non-degenerate
if $\det\bigl(\Psi(1)-\bm{\sigma}\bigr) \neq 0$.
The Conley-Zehnder index of a non-degenerate stationary braid $\x$
is defined by $\mu(\x) \bydef \mu_\sigma( \Psi,1)$, where $\sigma\in
S_n$ is the associated permutation of $\x$.
\end{definition}

\begin{remark}
\label{rmk:withmorse} If $\x=\{x^k(t)\}$ is a stationary non-degenerate
braid, then $\mu(\x) $ can be related to the Morse indices
$\mu_H(x^k)$ provided that the matrix norm of $K=d^2H(x^k)$ is not
too large, e.g. if $\Vert K\Vert<2\pi$. For $\mu_H(\x) \bydef \sum_k
\mu_H(x^k)$,
\begin{equation}
\label{eqn:mf1}
\mu(\x) =
\mu_\sigma( \Psi,1) = \sum_k\l(1-\mu_H(x^k)\r) = n- \mu_H(\x) .
\end{equation}
This relation can be useful in some instances for computing Floer
homology, see Sect.\ \ref{subsec:annulus}. Indeed, in dimension two
$\Psi$ satisfies: $\Psi(t) = \exp{\l(J_0Kt\r)}$, where $K=d^2H(x^k(t))$
is a constant matrix. Then $\mu(x^k) = \mu_{\rm CZ}(\Psi) =
1-\mu^-(K)$, where $\mu^-(K)$ is the number of negative eigenvalues
of eigenvalues. The latter equality follows from Thm.\ 3.3 in
\cite{SalZehn1}.
\end{remark}


\section{Transversality and connecting orbit spaces}
\label{sec:trans} Central to the analysis of the Cauchy-Riemann
equations are various generic non-degeneracy and transversality
properties. The first important statement in this direction involves
the generic non-degeneracy of critical points.

\subsection{Generic properties of critical points}
\label{subsec:points} Define $\P_H\left([\x]\rel\y\right)$ to be those
critical points in $\P_H$ that are contained in the braid class
$[\x]\rel\y$.
\begin{proposition}
\label{prop:trans1} Let $[\x]\rel\y$ be a proper relative braid class.
Then, for any Hamiltonian $H \in \cH$, with $\y\in \P_H(\Omega^{m})$,
there exists a $\delta_*>0$ such that for any $\delta<\delta_*$ there
exists a nearby Hamiltonian $H'\in \cH$ satisfying
\begin{enumerate}
\item [(i)] $\Vert H-H'\Vert_{C^\infty}<\delta$;
\item [(ii)] $\y\in \P_{H'}(\Omega^{m})$,
\end{enumerate}
such that $\P_{H'}\left([\x]\rel\y\right)$ consists of only finitely many
non-degenerate critical points for the action $\f_{H'}$.
\end{proposition}

We say that the property that $\P_H\left([\x]\rel\y\right)$ consists of
only non-degenerate critical points is a generic property, and is
satisfied by generic Hamiltonians in the above sense.

\begin{proof} Given $H\in \cH$ we start off with defining  a class
of perturbations. For a braid $\y\in \Omega^m$, define the tubular
neighborhood $N_\epsilon(\y)$ of $\y$ in $\R/\Z \times \D^2$ by :
$$
    N_\epsilon(\y) = \bigcup_{\stackrel{k=1,\cdots,m}{t\in [0,1]}} B_\epsilon(y^k(t)).
$$
If $\epsilon>0$ is sufficiently small, then a neighborhood
$N_\epsilon(\y)$ consists of $m$ disjoint cylinders. Let $D_\epsilon =
\{x\in \D^2~|~1-\epsilon < |x|\le 1\}$ be a small neighborhood of the
boundary, and define
$$
    A_\epsilon
    = N_\epsilon(\y) \cup \bigl(\R/\Z\times D_\epsilon\bigr),\quad
    A_\epsilon^c = \bigl(\R/\Z \times \D^2\bigr) \backslash A_\epsilon .
$$

Let $\cT^{J,H}([\x]\rel\y)$ represent the paths in the cylinder traced
out by the elements of $\cS^{J,H}([\x]\rel\y)$:
$$
\cT^{J,H}([\x]\rel\y) \bydef
  \bigl\{ (t,x^k(t)) \bigm| 1\leq k \leq n, \, t \in [0,1],\,  \x \in \cS^{J,H}([\x]\rel\y) \bigr\}.
$$
Since $[\x] \rel \y$ is proper, there exists an  $\epsilon_*>0$, such
that for all $\epsilon \le \epsilon_*$ it holds that
$\cT^{J,H}([\x]\rel\y) \subset \Int ( A_{2\epsilon}^c)$.
Now fix $\epsilon \in (0, \epsilon_*]$.
%
%
On $C^\infty(\R/\Z\times\D^2;\R)$ we define the norm
$$
\Vert \bh\Vert_{C^\infty}\bydef \sum_{k=0}^\infty \epsilon_k\Vert \bh\Vert_{C^k},
$$
for a sufficiently fast decaying sequence $\epsilon_k>0$, such that
$C^\infty$ equipped with this norm is a separable Banach space,
dense in $L^2$.
Let
\begin{eqnarray*}
  \cV_\epsilon
  &\bydef&
  \{ h\in C^\infty(\R/\Z\times \D^2;\R)~|~ \text{supp}\, h \subset A^c_\epsilon \} ,\\
  \cV_{\delta,\epsilon}
  &\bydef&
  \{ h \in \cV_\epsilon ~|~ \Vert  h \Vert_{C^\infty}<\delta \} ,
\end{eqnarray*}
and consider Hamiltonians of the form $H' = H + h_\delta \in \cH$,
with $h_\delta\in \cV_{\delta,\epsilon}$. Then, by construction,
$\y\in \P_{H'}(\bO{m})$, and by Proposition \ref{prop:isol1} the set
$\cS^{J,H'}([\x]\rel\y)$ is compact and isolated in the proper braid
class $[\x]\rel\y$ for all perturbation $h_\delta \in
\cV_{\delta,\epsilon}$. A straightforward compactness argument using
the compactness result of Proposition \ref{prop:isol1} shows that
$\cT^{J,H+h_\delta}([\x]\rel\y)$ converges to $\cT^{J,H}([\x]\rel\y)$
in the Hausdorff metric as $\delta\to 0$. Therefore, there exists a
$\delta_*>0$, such that $\cT^{J,H+h_{\delta}}([\x]\rel\y) \subset \Int
(A_{2\epsilon}^c)$, for all $0\le \delta\le \delta_*$. In particular
$\P_{H+h_{\delta,\epsilon}}\subset \Int (A_{2\epsilon}^c)$, for all
$0\le \delta\le \delta_*$. Now fix $\delta \in (0, \delta_*]$.

The Hamilton equations for $H'$ are $x_t^k - J_0 \nabla H(t,x^k) -
J_0 \nabla h (t,x) = 0$, with periodic boundary conditions in $t$.
Define $\cU_{\epsilon} \subset
W^{1,2}_\sigma([0,1];\R^{2n})$ to be the open subset of functions
$\x=\{x^k\}$ such that  $x^k(t) \in \Int (A_{2\epsilon}^c)$ and define
the nonlinear mapping
$$
\cG: \cU_{\epsilon} \times \cV_{\delta,\epsilon} \to L^2([0,1];\R^{2n}),
$$
which represents the above system of equations and boundary
conditions. Explicitly,
$$
\cG(\x,h) = \overline J_0\x_t + \nabla \overline{H}(t,\x) + \nabla \overline{h}(t,\x),
$$
where $\overline{H}(t,\x) = \sum_k H(t,x^k)$, and likewise for
$\overline{h}$. The mapping $\cG$ is linear in $h$. Since $\cG $ is
defined on $\cU_{\epsilon}$ and both $H$ and $h$ are of class
$C^\infty$, the mapping $\cG$ is of class $C^1$. The derivative with
respect to variations $(\xi,\delta h)\in
W^{1,2}_\sigma([0,1];\R^{2n})\times \cV_\epsilon$ is given by
\begin{eqnarray*}
    d\cG(\x,h)(\xi,\delta h)
    &=&
    \overline J_0 \xi_t +  d^2\overline{H}(t,\x)\xi +  d^2\overline{h}(t,\x)\xi
    +\nabla \delta \overline{h }(t,\x)\\
    &=&
    L_\x \xi + \nabla \delta \overline{h }(t,\x),
\end{eqnarray*}
where $L_\x = \overline J_0\frac{d}{dt} +
d^2\overline{H}(t,\x)+d^2\overline{h}(t,\x)$, by analogy with
Equation (\ref{e:Lx}). We see that there is a one-to-one
correspondence between elements $\Psi$ in the kernel of $L_\x$ and
symplectic paths described by Equation (\ref{e:Psix}) with
$\det\bigl(\Psi(1)-\bm{\sigma}\bigr) = 0$. In other words, the
stationary braid $\x$ is non-degenerate if and only if $L_\x$ has
trivial kernel.

The operator $L_\x$ is a self-adjoint operator on $L^2([0,1];\R^{2n})$
with domain $W^{1,2}_\sigma([0,1];\R^{2n})$ and is Fredholm with
$\ind(L_\x)=0$. Therefore $\cG_h\bydef \cG(\cdot,h)$ is a (proper)
nonlinear  Fredholm operator with
$$
\ind (\cG_h) = \ind(L_\x)=0.
$$
Define the set
$$
\cZ  = \bigl\{ (\x,h)\in \cU_{\epsilon} \times \cV_{\delta,\epsilon}~|~\cG(\x,h) =0\bigr\}=\cG^{-1}(0).
$$
We show that $Z$ is a Banach manifold by demonstrating that
$d\cG(\x,h)$ is surjective for all $(\x,h) \in Z$. Since
$d\cG(\x,h)(\xi,\delta h ) = L_\x \xi -\nabla \delta \overline{h}(t,\x)$,
and the (closed) range of $L_\x$ has finite codimension, we need to
show there is a (finite dimensional) complement of $R(L_\x)$ in the
image of $\nabla \delta \overline{h }(t,\x)$. It suffices to show that
$\nabla \delta \overline{h }(t,\x)$ is dense in $L^2([0,1];\R^{2n})$.

Recall that for any pair $(\x,h)\in Z$, it holds that $\x\in
\P_{H'}\subset \Int(A_{2\epsilon}^c)$. As before consider a
neighborhood $N_{\epsilon}(\x)$, so that $N_{\epsilon}(\x)\subset
\Int(A_{\epsilon}^c)$ and consists of $n$ disjoint cylinders
$N_{\epsilon}(x^k)$. Let $\phi_{\epsilon}^k(t,x) \in
C_0^\infty(N_{\epsilon}(x^k))$, such that $\phi_{\epsilon}^k\equiv 1$
on $N_{\epsilon/2}(x^k)$. Define, for arbitrary $f^k\in
C^\infty(\R/\Z;\R^2)$,
$$
(\delta h) (t,x) = \sum_{k=1}^n \phi_{\epsilon}^k(t,x) \langle f^k(t),x\rangle_{L^2}.
$$
Since $\phi^k_{\epsilon}(t,x^k(t)) \equiv 1$ it holds that $\delta
\overline{h }(t,\x) = \sum_{k=1}^n   \langle f^k(t),x^k\rangle_{L^2}$
for $\x \in \P_{H'}$, and therefore the gradient satisfies $\nabla \delta
\overline{h }(t,\x) = {\mathbf{ f}}=(f^k) \in C^\infty(\R/\Z;\R^{2n})$.
Moreover, $\delta h  \in \cV_\epsilon$ by construction, and because
$C^\infty(\R/\Z;\R^{2n})$ is dense in $L^2([0,1];\R^{2n})$ it follows
that $d\cG(\x,h)$ is surjective.

Consider the projection $\pi: \cZ \to \cV_{\delta,\epsilon}$, defined
by $\pi(\x,h) = h$. The projection $\pi$ is a   Fredholm operator.
Indeed, $d\pi: T_{(\x,h)} \cZ \to \cV_\epsilon$, with
$d\pi(\x,h)(\xi,\delta h ) = \delta h $, and
$$
T_{(\x,h)} \cZ  =
\left\{
(\xi,\delta h )\in W^{1,2}_\sigma\times \cV_\epsilon
~|~
L_\x\xi - \nabla\delta \overline{h } = 0
\right\}.
$$
From this it follows that $\ind(d\pi) = \ind(L_\x) = 0$. The Sard-Smale
Theorem 
implies that the set of perturbations $h\in
\cV_{\delta,\epsilon}^{\rm reg} \subset \cV_{\delta,\epsilon}$ for
which $h$ is a regular value of $\pi$ is an open and dense subset. It
remains to show that $h\in \cV_{\delta,\epsilon}^{\rm reg}$ yields
that $L_\x$ is surjective. Let $h\in \cV_{\delta,\epsilon}^{\rm reg}$,
and $(\x,h)\in Z$, then $d\cG(\x,h)$ is surjective, i.e., for any
$\zeta\in L^2([0,1];\R^{2n})$ there are $(\xi,\delta h )$ such that
$d\cG(\x,h)(\xi,\delta h )=\zeta$. On the other hand, since since
$h$~is a regular value for $\pi$, there exists a $\widehat\xi$ such
that $d\pi(\x,h)(\widehat\xi,\delta h )=\delta h $ ,
$(\widehat\xi,\delta h ) \in T_{(\x,h)} \cZ $, i.e. $L_\x\widehat\xi -
\nabla \delta \overline{h }=0$. Now
\begin{eqnarray*}
    L_\x(\xi-\widehat\xi) &=&d\cG(\x,h)(\xi-\widehat\xi,0)\\
    &=& d\cG(\x,h)\bigl((\xi,\delta h )-(\widehat\xi,\eta)\bigr) = \zeta -0 = \zeta,
\end{eqnarray*}
which proves that for all $h\in \cV_{\delta,\epsilon}^{\rm reg}$ the
operator $L_\x$ is surjective, and hence also injective, implying that
$\x$ is non-degenerate.
\end{proof}

For $\x_\pm \in \P_H$, let  $\M_{\x_-,\x_+}^{J,H}([\x]\rel\y)$ be the
space of all bounded solutions in $\u \in \M^{J,H}([\x]\rel\y)$ such
that $\lim_{s \to \pm \infty} \u(s,\cdot)=\x_{\pm}(\cdot)$, i.e.,
connecting orbits in the relative braid class $[\x]\rel \y$. If
$\x_-=\x_+$, then the set consists of just this one critical point. The
space $\cS_{\x_-,\x_+}^{J,H}([\x]\rel\y)$, as usual, consists of the
corresponding trajectories.
\begin{lemma}
\label{lem:trans2} Let $[\x]\rel\y$ be a proper braid class and let
$H\in \cH$ be a generic Hamiltonian. Then
$$
\M^{J,H}([\x]\rel\y) \subset \bigcup_{\x_\pm \in \P_H} \M_{\x_-,\x_+}^{J,H}([\x]\rel\y),
$$
where $\P_H = \P_H([\x]\rel\y)$.
\end{lemma}
See \cite{Sal} for a detailed proof.

\begin{corollary}
\label{cor:isol2} Let $[\x]\rel\y$ be a proper relative braid class and
let $H$ be a generic Hamiltonian with $\y \in \P_H(\bO{m})$. Then
the space of bounded solutions is given by the union $\M^{J,H}([\x]\rel
\y) = \bigcup_{\x_\pm\in\P_H } \M_{\x_-,\x_+}^{J,H}([\x]\rel\y)$.
\end{corollary}
\begin{proof} The key observation is that since $\x_\pm \rel \y
\in [\x]\rel\y$, also $\u(s,\cdot)\rel\y \in [\x]\rel\y$, for all $s\in \R$
(the crossing number cannot change). Therefore, any $\u \in
\M_{\x_-,\x_+}^{J,H}([\x]\rel\y)$ is contained in $[\x]\rel\y$, and thus
$\M_{\x_-,\x_+}^{J,H}([\x]\rel\y)\subset \M^{J,H}([\x]\rel\y)$. The
remainder of the proof follows from Lemma \ref{lem:trans2}.
\end{proof}

Note that the sets $\M_{\x_-,\x_+}^{J,H}([\x]\rel\y)$ are not
necessarily compact in $\bOn$. The following corollary gives a more
precise statement about the compactness of the spaces
$\M_{\x_-,\x_+}^{J,H}([\x]\rel\y)$, which will be referred to as
geometric convergence.
\begin{corollary}
\label{cor:geom1} Let $[\x]\rel\y$ be a proper relative braid class and
$H$ be a generic Hamiltonian with $\y \in \P_H(\bO{m})$. Then for
any sequence $\{\u_n\} \subset \M_{\x_-,\x_+}^{J,H}([\x]\rel\y)$
(along a subsequence) there exist  stationary braids $\x^i \in
\P_H([\x]\rel\y)$, $i=0,\dots,m$, orbits $\u^i \in
\M_{\x^i,\x^{i-1}}^{J,H}([\x]\rel\y)$ and times $s_n^i$,
$i=1,\dots,m$, such that
$$
\u_n(\cdot + s_n^i,\cdot)   \longrightarrow \u^i,\quad n\to\infty,
$$
in $C^r_{\rm loc}(\R\times \R/\Z)$, for any $r\ge 1$. Moreover, $\x^0
= \x_+$ and $\x^m = \x_-$ and $\f_H(\x^{i}) >\f_H(\x^{i-1})$ for
$i=1,\dots, m$. The sequence $\u_n$ is said to geometrically
converge to the broken trajectory $(\u^1,\dots,\u^m)$.
\end{corollary}
See, again, \cite{Sal} for a proof.

\subsection{Generic properties for connecting orbits}
\label{subsec:orbits} As for critical points, non-degeneracy can also
be defined for connecting orbits. This closely follows the ideas in the
previous subsection. Set $W^{1,p}_\sigma = W^{1,p}_\sigma(\R\times
[0,1];\R^{2n})$ and $L^p= L^p(\R\times [0,1];\R^{2n})$.

Let $\x_-, \x_+\in \P_H(\bOn)$ be non-degenerate stationary braids. A
connecting orbit $\u\in \M_{\x_-,\x_+}^{J,H}$ is said to be
\style{non-degenerate}, or \style{transverse}, if the linearized
Cauchy-Riemann operator
$$
{\partial\over\partial s} - \overline J{\partial\over\partial t}
+ \overline J \overline J_0 d^2 \overline H(t,\u(s,t)):~W^{1,p}_\sigma\to L^p,
$$
is a surjective operator (for all $1<p<\infty$).

As before we equip $C^\infty(\R/\Z\times \D^2;\R)$ with a Banach
structure, cf. Sect.\ \ref{subsec:points}.

\begin{proposition}
\label{prop:trans4} Let $[\x]\rel\y$ be a proper relative braid class,
and $H\in \cH$ be a generic Hamiltonian such that $\y \in
\P_H(\bO{m})$. Then, there exists a $\delta_*>0$ such that for any
$\delta\le \delta_*$ there exists a nearby Hamiltonian $H' \in \cH$
with $\Vert H-H'\Vert_{C^\infty}<\delta$ and $\y \in \P_{H'}(\bO{m})$
such that
\begin{enumerate}
\item[(i)] $ \P_{ H'}([\x]\rel\y) =  \P_{ H}([\x]\rel\y)$ and consists of
    only non-degenerate stationary points for the action $\f_{H'}$;
\end{enumerate}
and for any pair $\x_-,\x_+ \in \P_{H'}([\x]\rel\y)$
\begin{enumerate}
\item[(ii)] $\cS^{J,H'}_{\x_-,\x_+}([\x]\rel\y)$ is isolated in
    $[\x]\rel\y$;
\item[(iii)] $\M_{\x_-,\x_+}^{J,H'}([\x]\rel\y)$ consists of
    non-degenerate connecting orbits;
\item[(iv)] $\M^{J,H'}_{\x_-,\x_+}([\x]\rel\y)$ are smooth manifolds
    without boundary and
\end{enumerate}
$$
\dim  \M^{J,H'}_{\x_-,\x_+} ([\x]\rel\y)= \mu(\x_-) - \mu(\x_+),
$$
where $\mu$ is the Conley-Zehnder index defined in Definition
\ref{defn:mas3}.
\end{proposition}
\begin{proof}
Since $[\x]\rel\y$ is a proper braid class it follows from Proposition
\ref{prop:isol1} that $\cS^{J,H'}_{\x_-,\x_+}$ is isolated in $[\x]\rel\y$
for any $H'\in \cH$ provided $\y \in \P_{H'}$.

As for the transversality properties we follow Salamon and Zehnder
\cite{SalZehn1}, where perturbations in $\R^{2n}$ are considered. We
adapt the proof for Hamiltonians in $\R^2$. The proof is similar in
spirit to the genericity of critical points.

As in the proof of Proposition \ref{prop:trans1} we denote by
$\cV_\epsilon$ the set of perturbations $h\in C^\infty(\R/\Z\times
\D^2;\R)$ whose support is bounded away from $(t,\y(t))$,
$(t,\x_\pm(t))$ and $\partial\D^2$ (this yields a corresponding set
$A_\epsilon$ as in the proof of Proposition \ref{prop:trans1}). If we
choose $h\in \cV_{\delta,\epsilon}$ there exists a $\delta_*$ such
that $\P_{ H'}([\x]\rel\y) = \P_{ H}([\x]\rel\y)$ and consists of only
non-degenerate stationary points for the action $\f_{H'}$. For details
of this construction we refer to the proof of Proposition
\ref{prop:trans1}.

Define the Cauchy-Riemann operator
$$
    \cG(\u,\bh)=\u_{s}-\overline J \u_t
    +\overline{J}\overline J_0\nabla \overline{H}(t,\u)
    + \overline{J}\overline J_0\nabla  \overline \bh(t,\u).
$$
Based on the a priori regularity of bounded solutions of the
Cauchy-Riemann equations we define for $2<p<\infty$ the affine
spaces
\begin{equation}
\label{e:cU}
 \cU^{1,p}(\x_-,\x_+)\bydef \bigl\{  \gamma + \xi~|~\xi \in
 W^{1,p}_{\sigma}(\R\times [0,1];\R^{2n}) \bigr\},
\end{equation}
and balls $\cU_\eps^{1,p} = \{ \u \in \cU^{1,p} \,|\, \Vert
\xi\Vert_{W^{1,p}_\sigma}<\epsilon \}$, where $\gamma(s,t)\in
C^2\bigl(\R\times [0,1];(\D^2)^n\bigr)$ is a fixed connecting path
such that $\lim_{s\rightarrow\pm\infty} \gamma(s,\cdot)=\x_{\pm}$
and $\gamma(s,t) \in \Int (\D^2)^n$ for all $(s,t) \in \R\times [0,1]$.
Therefore, for $p>2$, functions $\u \in
\mathcal{U}_{\epsilon}^{1,p}(\x_-,\x_+)$ satisfy the limits
$\lim_{s\rightarrow\pm\infty}\u(s,\cdot)=\x_{\pm}$ and if
$\epsilon>0$ is chosen sufficiently small then also $\u(s,t) \in \Int
(\D^2)^n$ for all $(s,t) \in \R\times [0,1]$. The mapping
$$
    \cG:  \cU_{\epsilon}^{1,p}(\x_-,\x_+)\times \cV_{\delta,\epsilon}
    \rightarrow L^p(\R\times [0,1];\R^{2n}),
$$
is smooth. Define
$$
    \cZ_{\x_-,\x_+}\bydef\bigl\{(\u,\bh)\in
    \mathcal{U}_{\epsilon}^{1,p}(\x_-,\x_+)\times \cV_{\delta,\epsilon}\
    |\ \cG(\u,\bh)=0\bigr\}=\cG^{-1}(0),
$$
which is Banach manifold  provided that $d\cG(\u,\bh)$ is onto on for
all $(\u,\bh) \in \cZ_{\x_-,\x_+}$, where
$$
d\cG(\u,\bh)(\xi,\delta\bh) = d_1 \cG(\u,\delta\bh) \xi
+ \overline J \overline J_0 \nabla  \overline{\delta \bh}.
$$
Assume that $d\cG(\u,\bh)$ is not onto. Then there exists a non-zero
function $\eta\in L^{q}$ which annihilates the range of $d\cG(\u,\bh)$
and thus also the range of $d_1\cG(\u,\bh)$, which is a Fredholm
operator of index $\mu(\x_-)-\mu(\x_+)$; see
Proposition~\ref{prop:mas2}. The relation $\langle
d_1\cG(\u,\bh)(\xi),\eta\rangle=0$ for all $\xi$ implies that
$$
d_1\cG(\u,\bh)^*\eta = -  \eta_s - \overline J   \eta_t + \overline J \overline J_0 d^2  \overline
H(t,\u)  \eta =0.
$$
Since  $\langle d\cG(\u,\bh)(\xi,\delta\bh),\eta \rangle =0$ it follows
that
\begin{equation}
\label{TransConn:eq:etaorthogonalrange}
    \int_{-\infty}^{\infty}\int_{0}^{1}\langle \eta(s,t),
    \overline J\overline J_0\nabla \delta \overline{ \bh}
    \rangle_{\R^{2n}} dtds=0,\qquad \forall \delta \bh.
\end{equation}
Due to the assumptions on $\bh$ and $H$, the regularity theory for
the linear Cauchy-Riemann operator implies that $\eta$ is smooth. It
remains to show that no such non-zero function $\eta$ exists.

{\it Step $1$}: The function  $\eta$ satisfies the following perturbed
Laplace's equation: $\Delta \eta = \partial_{\overline J} \bar
\partial_{\overline J} \eta = \partial_{\overline J}\overline J\overline J_0
d^2\overline H(t,\u) \eta$. If at some $(s_0,t_0)$ all derivatives of $\eta$
vanish, it follows from Aronszajn's unique
continuation~\cite{aronszajn} that $\eta \equiv  0$ is a neighborhood
of $(s_0,t_0)$. Therefore $\eta(s,t)\neq0$ for almost all $(s,t)\in\R
\times [0,1]$.

{\it Step $2$}: The vectors $\eta(s,t)$ and $\u_{s}(s,t)$ are linearly
dependent for all $s$ and $t$. Suppose not, then these vector are
linearly independent at some point $(s_0,t_0)$. By Theorem 8.2 in
\cite{SalZehn1}   we may assume without loss of generality that
$\u_s(s_0,t_0)\not = 0$ and $\u(s_0,t_0)\not = \x_\pm(t_0)$ --- a
regular point. We now follow the   arguments  as in the proof of
Theorem 8.4 in \cite{SalZehn1} with some modifications. Since
$(s_0,t_0)$ is a regular point there exists a small neighborhood $ U_0
= I_{t_0} \times U_{x^1} \times \dots \times U_{x^n}, $ such that
$V_0 = \{(s,t)~|~(t,\u(s,t)) \in U_0\}$ is a neighborhood $V_0$ of
$(s_0,t_0)$ and $U_0 \cap A_\epsilon = \varnothing$. The sets
$U_{x^k}$ are neighborhoods of $u^k(s_0,t_0)$ and have the
important property that $U_{x^k} \cap U_{x^{k'}} = \varnothing$ for
all $k\not = k'$. The proof in  \cite{SalZehn1} shows that the map
$(s,t) \mapsto (t,\u(s,t))$ from $V_0$ to $U_0$ is a diffeomorphism.
By choosing $V_0$ small enough $\u_s$ and $\eta$ are linearly
independent on $V_0$. As in \cite{SalZehn1} this yields the existence
of coordinates $\phi_t: (\D^2)^n \to \R^{2n}$ in a neighborhood
$(t_0,\u(s_0,t_0))\in U_0$ such that
$$
\phi_t(\u(s,t))=(s-s_0,0,\dots,0), \qquad d\phi_t(\u(s,t))\eta(s,t)=(0,1,0,\dots,0).
$$
Define $g:\R^{2n}\times \Sb^1\rightarrow\R$ via
$g(y_1,\dots,y_{2n},t)=\beta(t-t_0)\beta(y_1)\beta(y_2)y_2$, where
$\beta\ge 0$ is a $C^\infty$ cutoff function such that $\beta=1$ on a ball
centered at zero $B_{\delta_1}(0)$ for sufficiently small positive
$\delta_1$ and $\beta=0$ outside of $B_{2\delta_1}(0)$. We define a
Hamiltonian $\chi:\Sb^1 \times (\D^2)^n \rightarrow\R$ via
$\chi(t,x_1,\dots,x_{2n})=g(\phi(x_1,\dots,x_{2n}),t)$. By
construction $\chi$ vanishes outside $U_0$,
$\chi(t,\u(s,t))=g((s-s_0,0,\dots,0),t)=0$, and $d\chi(t,\u(s,t))
\eta(s,t) =\beta(s-s_0)\beta(t-t_0)$ for all $(s,t) \in V_0$. In order to
have an admissible perturbation we need a Hamiltonian $\delta h \in
\cV_\epsilon$ such that $\chi = \delta\overline h$. Since $\chi$
vanishes outside $U_0$ and since $U_{x^k} \cap U_{x^{k'}} =
\varnothing$ for all $k\not = k'$ we can solve this equation. Set
$(\delta h)(t,u^k(s_0,t)) = \chi(t,\u(s_0,t_0))/n$ and define $\delta
h$ on the disjoint set $U_{x^k}$ as follows:
$$
(\delta h)(t,x^k) = \chi(t,u^1(s_0,t),\dots,x^k,\dots,u^n(s_0,t))
- \frac{n-1}{n} \chi(t,\u(s_0,t)).
$$
Since the sets $U_{x^k}$ are disjoint $\delta h$ is well-defined
$\Sb^1\times \D^2$ and zero outside $I_{t_0} \times (\cup_k U_{x^k})$. With this
choice of perturbation $\delta h$ the integral in Equation
(\ref{TransConn:eq:etaorthogonalrange}) is non-zero which
contradicts the assumption on $\eta$.

The remaining steps are identical to those in the proof of Theorem
8.4 in \cite{SalZehn1}: we outline these for completeness.

{\it Step $3$}: The previous step implies the existence of a function
$\lambda:\R\times[0,1]\rightarrow\R$ such that
$\eta(s,t)=\lambda(s,t)\frac{\partial \u}{\partial s}(s,t)$, for all $s,t$
for which $\eta(s,t) \not = 0$. Using a contradiction argument with
respect to Equation \rmref{TransConn:eq:etaorthogonalrange} yields
$\frac{\partial \lambda}{\partial s}(s,t)=0$, for almost all $(s,t)$. In
particular we obtain that $\lambda$ is $s$-independent and we can
assume that $\lambda(t)\ge \delta>0$ for all $t\in[0,1]$ (invoking
again unique continuation).

{\it Step $4$}: This final step provides a contradiction to the
assumption that $d\cG$ is not onto. It holds that
$$
    \int_0^1\left\langle \frac{\partial \u}{\partial s}(s,t),\eta(s,t)\right\rangle dt=
    \int_0^1\lambda(t)\left|\frac{\partial \u}{\partial s}(s,t)\right|^2dt
    \ge \delta \int_0^1\left|\frac{\partial \u}{\partial s}(s,t)\right|^2dt>0.
$$
The functions $\u_s$  and $\eta$ satisfy the equations
$d_1\cG(\u,\bh)\u_s =0,\quad d_1\cG(\u,\bh)^*\eta=0$, respectively.
From these equations we can derive expressions for $\u_{ss}$ and
$\eta_s$, from which:
$$
    \frac{d}{ds}\int_0^1\left\langle\frac{\partial \u}{\partial
    s}(s,t),\eta(s,t)\right\rangle dt=0.
$$
Combining this with the previous estimate yields that
$\int_{-\infty}^{\infty}\int_0^1| \u_{s}(s,t)|^2dt=\infty$, which,
combined with the compactness properties, contradicts the fact that
$\u\in\M_{\x_-,\x_+}$; thus $d\cG(\u,\bh)$ is onto for all $(\u,\bh) \in
\cZ_{\x_-,\x_+}$.

We can now apply the Sard-Smale theorem as in the proof of
Proposition \ref{prop:trans1}. The only difference here is that
application of the Sard-Smale requires
$(\mu(\x_-)-\mu(\x_+)+1)$-smoothness of $\cG$ which is guaranteed
by the smoothness of $\y$, $H$ and $\bh$.
\end{proof}

We can label a Hamiltonian to be generic now if both
$\P_H([\x]\rel\y)$ and $\cM_{\x_-.\x_+}([\x]\rel\y)$, $\x_\pm \in
\P_H([\x]\rel\y)$, are non-degenerate. The terminology `generic' is
justified since by the taken finitely many intersections for the
different pairs $\x_\pm$ we obtain a dense set of Hamiltonian,
denoted by $\cHH$. For generic Hamiltonians $H\in \cHH$ the
convergence of Corollary \ref{cor:geom1} can be extended with
estimates on the Conley-Zehnder indices of the stationary braids.
\begin{corollary}
\label{cor:geom2} Let $[\x]\rel\y$ be a proper relative braid class and
$H \in \cHH$ be a generic Hamiltonian with $\y \in \P_H(\bO{m})$. If
$\u_n$ geometrically converges to the broken trajectory
$(\u^1,\dots,\u^m)$, with $\u^i \in
\M_{\x^i,\x^{i-1}}^{J,H}([\x]\rel\y)$, $i=1,\dots,m$ and $\x^i \in
\P_H([\x]\rel\y)$, $i=0,\dots,m$, then
$$
   \mu(\x^i) >  \mu(\x^{i-1}),
$$
for $i=1,\dots,m$.
\end{corollary}

\begin{proof} See \cite{Sal} for a detailed proof of this statement.
\end{proof}

Since Proposition \ref{prop:trans4} provides a dense set of
Hamiltonians $\cHH$ the intersection of dense sets over all pairs
$(\x_-,\x_+)$ yields a dense set of Hamiltonians for which (i)-(iv) in
Proposition  \ref{prop:trans4} holds for all pairs pairs $(\x_-,\x_+)$
and thus for all of $\M^{J,H}([\x]\rel\y)$.

The above proof also carries over to the Cauchy-Riemann equations
with  $s$-dependent Hamiltonians $H(s,\cdot,\cdot)$. Exploiting the
Fredholm index property for the $s$-dependent case we obtain the
following corollary. Let $s \mapsto H(s,\cdot,\cdot)$ be a smooth
path in $\cH$ with the property $H_s = 0$ for $|s|\ge R$. We have
the following non-autonomous version of
Proposition~\ref{prop:trans4}, see~\cite{SalZehn1}.
\begin{corollary}
\label{cor:trans7} Let $[\x\rel\y]$ be a proper relative braid class with
fibers $[\x]\rel\y$, $[\x']\rel\y'$ in $[\x\rel\y]$. Let $s\mapsto
H(s,\cdot,\cdot)$ be a smooth path in $\cH$ as described above with
$H_\pm = H(\pm \infty,\cdot,\cdot) \in \cHH$ and $\y \in \P_{H_-}$,
$\y' \in \P_{H_+}$. Then there exists a $\delta_*>0$ such that for any
$\delta \le \delta_*$ there exist a path of Hamiltonians $s\mapsto
H'(s,\cdot,\cdot)$ in $\cH$, with $H'_s =0$ for $|s|\ge R$, $H(\pm
\infty,\cdot,\cdot) = H_\pm$ and $\Vert H-H'\Vert_{C^\infty}<\delta$
such that
\begin{itemize}
\item[(i)] $\cS^{J,H'} ([\x\rel\y])$ is isolated in $[\x \rel\y]$;
\item[(ii)] $\M^{J,H'}_{\x_-,\x'_+} ([\x\rel\y])$ consist of
    non-degenerate connecting orbits with respect to the
    $s$-dependent CRE;
\item[(iii)] $\M^{J,H'}_{\x_-,\x'_+} ([\x\rel\y])$ are smooth
    manifolds without boundary with
    $$
      \dim \M_{\x_-,\x'_+}^{J,H'}=\mu(\x'_-)-\mu(\x'_+)+1,
    $$
\end{itemize}
where $\mu$ is the Conley-Zehnder indices with respect to the
Hamiltonians $H_\pm$.
\end{corollary}


\section{Floer homology for proper braid classes}
\label{sec:defn-floer}

\subsection{Definition}
\label{subsec:defn} Let $\y\in \Omega^m$ be a smooth braid and
$[\x]\rel \y$ a proper relative braid class. Let $H\in \cHH$ be a
generic Hamiltonian with respect to the proper braid class $[\x]\rel\y$
(as per Proposition \ref{prop:trans1}). Then the set of bounded
solutions $\M^{J,H}([\x]\rel\y)$ is compact and non-degenerate,
$\P_H([\x]\rel\y)$ is non-degenerate, and $\cS^{J,H}([\x]\rel\y)$ is
isolated in $[\x]\rel\y$. Since $\P_H([\x]\rel\y)$ is a finite set we can
define the chain groups
\begin{equation}
\label{eqn:chaingroup}
    C_k\bigl([\x]\rel\y,H ;\Z_2\bigr)
    \bydef
    \bigoplus_{\x'\in \P_H([\x]\rel\y)\atop \mu(\x') =k} \Z_2 \cdot \x',
\end{equation}
as products of $\Z_2$. We define the boundary operator $\partial_k:
C_k \to C_{k-1}$ in the standard manner as follows. By Proposition
\ref{prop:trans4}, the orbits $\u\in \M^{J,H}_{\x_-,\x_+}([\x]\rel\y)$
are non-degenerate for all pairs $\x_-,\x_+\in \P_H([\x]\rel\y)$. Let
$\widehat \M^{J,H}_{\x_-,\x_+} =\M^{J,H}_{\x_-,\x_+}/\R$ be the
equivalence classes of orbits identified by translation in the
$s$-variable. Consequently, the $\wM^{J,H}_{\x_-,\x_+}$ are smooth
manifolds of dimension $\dim \wM^{J,H}_{\x_-,\x_+} =
\mu(\x_-)-\mu(\x_+) -1$.
\begin{lemma}
\label{lem:geom3} If  $\mu(\x_-)-\mu(\x_+) =1$, then
$\wM^{J,H}_{\x_-,\x_+}([\x]\rel\y)$   consists of finitely
many equivalence classes.
\end{lemma}
\begin{proof}
From the compactness Theorem \ref{prop:CR4},
 the geometric convergence in Corollaries \ref{cor:geom1}
and \ref{cor:geom2} and gluing we derive that any sequence $\{\u_n\}
\subset \M^{J,H}_{\x_-,\x_+}([\x]\rel\y)$ geometrically
converges to a broken trajectory  $(\u^1,\dots,\u^m)$, with
$\u^i \in \M_{\x^i,\x^{i-1}}^{J,H}([\x]\rel\y)$, $i=1,\dots,m$
and $\x^i \in \P_H([\x]\rel\y)$, $i=0,\dots,m$, such that $
\mu(\x^i) >  \mu(\x^{i-1})$, for $i=1,\dots,m$. Since by
assumption $\mu(\x_-) = \mu(\x_+)+1$, it follows that $m=1$ and
$\u_n$ converges to a single orbit $\u^1 \in
\M^{J,H}_{\x_-,\x_+} ([\x]\rel\y)$. Therefore, the set
$\wM^{J,H}_{\x_-,\x_+} ([\x]\rel\y)$ is compact. From
Proposition \ref{prop:trans4} it follows that the orbits in
$\wM^{J,H}_{\x_-,\x_+} ([\x]\rel\y)$ occur as isolated
points and therefore $\wM^{J,H}_{\x_-,\x_+} ([\x]\rel\y)$
is a finite set.
\end{proof}

Define the boundary operator by
\begin{equation}
\label{boundaryoperator}
\partial_k(J,H) \x
\bydef
\sum_{\x'\in \P_H([\x]\rel\y)\atop \mu(x') =k-1} n(\x,\x';J,H)  \x',
\end{equation}
where $ n(\x,\x';J,H) = \Bigl[\# ~\wM^{J,H}_{\x,\x'}\Bigr] ~{\rm mod}~
2 \in \Z_2$. The final property that the boundary operator has to
satisfy is $\partial_{k-1}\circ\partial_k =0$. The composition counts
the number of `broken connections' from $\x$ to $\x''$ modulo $2$.
\begin{lemma}
\label{lem:geom4} If $\mu(\x_-) - \mu(\x_+) =2$, then
$\wM^{J,H}_{\x_-,\x_+} ([\x]\rel\y)$ is a smooth 1-dimensional
manifold with finitely many connected components. The non-compact
components can be identified with $(0,1)$ and the closure with
$[0,1]$.
The limits $\{0,1\}$ correspond to  unique pairs of distinct broken
trajectories
$$
(\u^1,\u^2) \in  \M^{J,H}_{\x_-,\x'} ([\x]\rel\y) \times \M^{J,H}_{\x',\x_+} ([\x]\rel\y),
$$
and
$$
(\tilde \u^1,\tilde \u^2)\in \M^{J,H}_{\x_-,\x''} ([\x]\rel\y) \times \M^{J,H}_{\x'',\x_+} ([\x]\rel\y),
$$
with $\mu(\x'') = \mu(\x') = \mu(\x_-) -1$.
\end{lemma}

We point out that properness of $[\x\rel\y]$ and thus the isolation of
$\cS^{J,H}$ is crucial for the validity of Lemma \ref{lem:geom4}.
From Lemma \ref{lem:geom4} it follows that the total number of
broken connections from $\x$ to $\x''$ is even;
hence $\partial_{k-1}\circ\partial_k =0$, and consequently,
$$
\Bigl(C_*\bigl([\x]\rel\y,H;\Z_2\bigr),\partial_*(J,H)\Bigr)
$$
is a (finite) chain complex.
The Floer homology of $\bigl([\x]\rel\y,J,H\bigr)$ is the homology of
the chain complex $(C_*,\partial_*)$:
\begin{equation}
    \label{eqn:floerhomology}
    {\rm HF}_k\bigl([\x]\rel\y,J,H;\Z_2\bigr)
    \bydef
    \frac{\ker \partial_k}{{\rm im} \partial_{k+1}}.
\end{equation}
This Floer homology is finite. It is not yet established that ${\rm
HF}_*$ is independent of $J,H$ and whether ${\rm HF}_*$ is an
invariant for proper relative braid class $[\x\rel\y]$.

\subsection{Continuation}
\label{subsec:continuation} Floer homology has a powerful invariance
property with respect to `large' variations in its parameters
\cite{Floer1}. Let  $[\x]\rel\y$ be  a proper relative braid class and
consider almost complex structures $J,\widetilde J\in \cJ$, and
generic Hamiltonians $H,\widetilde H \in \cHH$ such that $\y\in \P_H
\cap \P_{\widetilde H}$ Then the Floer homologies $ {\rm
HF}_*\bigl([\x]\rel\y,J,H;\Z_2\bigr) $ and $ {\rm
HF}_*\bigl([\x]\rel\y,\widetilde J,\widetilde H;\Z_2\bigr) $ are
well-defined.

\begin{proposition}
\label{prop:conti1} Given a proper relative braid class $[\x]\rel\y$,
$$
{\rm HF}_*\bigl([\x]\rel\y,J,H;\Z_2\bigr)
\cong
{\rm HF}_*\bigl([\x]\rel\y,\widetilde J,\widetilde H;\Z_2\bigr),
$$
under the hypotheses on $(J,H)$ and $(\widetilde J,\widetilde H)$ as
stated above.
\end{proposition}

In order to prove the isomorphism we follow the standard procedure
in Floer homology. The main steps can be summarized as follows.
Consider the chain complexes
$$
\Bigl(C_*\bigl([\x]\rel\y,H;\Z_2\bigr),\partial_*(J,H)\Bigr)~~ {\rm and} ~~
\Bigl(C_*\bigl([\x]\rel\y,\widetilde H;\Z_2\bigr),\partial_*(\widetilde J,\widetilde H )\Bigr) ,
$$
and construct homomorphisms $h_k$ satisfying the commutative
diagram
$$
\begin{CD}
\cdots @>>>  C_k(H) @>\partial_k(J,H)>> C_{k-1}(H) @>\partial_{k-1}(J,H)>> C_{k-2}(H) @>>> \cdots\\
@.  @Vh_kVV  @Vh_{k-1}VV  @Vh_{k-2}VV @. @.\\
\cdots @>>>  C_k(\widetilde H) @>\partial_k(\widetilde J,\widetilde H)>> C_{k-1}(\widetilde H) @>\partial_{k-1}(\widetilde J,\widetilde H)>> C_{k-2}(\widetilde H) @>>> \cdots\\
\end{CD}
$$
To define $h_k$ consider the homotopies $\lambda \mapsto
(J_\lambda,H_\lambda)$ in $ \cJ\times \cH $  with $\lambda \in
[0,1]$. In particular choose $H_\lambda = (1-\lambda) H + \lambda
\widetilde H$ such that $\y \in \P_{H_\lambda}$ for all $\lambda \in
[0,1]$. Note that at the end points $\lambda =0,1$ the systems  are
generic, i.e., $H_0 = H \in \cHH$ and $H_1 = \widetilde H \in \cHH$;
this is not necessarily true for all $\lambda \in (0,1)$. Define the
smooth function $\lambda(s)$  such that $\lambda(s)=0$ for $s\le -R$
and $\lambda(s) =1$ for $s\ge R$, for some $R>0$ and $0\le
\lambda(s)\le 1$ on $\R$.
The non-autonomous  Cauchy-Riemann equations become
\begin{equation}
 \label{eqn:CR5}
 u_s -  J_{\lambda(s)}  u_t  -  \nabla_{g_s}   H_{\lambda(s)}(t,u)  =0.
\end{equation}
By setting $J(s,\cdot,\cdot) = J_{\lambda(s)}(\cdot,\cdot)$ and
$H(s,\cdot,\cdot)=H_{\lambda(s)}$, Equation \rmref{eqn:CR5} fits in
the framework of Equation \rmref{eqn:CR2s}. By Corollary
\ref{cor:trans7} the path $s\mapsto H(s,\cdot,\cdot)$ can be to
chosen be generic with the same limits.

As before, denote the space of bounded solutions by
$\M^{J_\lambda,H_\lambda} =
\M^{J_\lambda,H_\lambda}(\overline\Omega^m)$. The requisite basic
compactness result is as follows:
\begin{proposition}
\label{prop:CR4a} The space $\M^{J_\lambda,H_\lambda}$ is compact
in the topology of uniform convergence on compact sets in $(s,t)\in
\R^2$, with derivatives up to order $r$.  Moreover, $\f_{H_\lambda}$
is uniformly bounded along trajectories $\u\in
\M^{J_\lambda,H_\lambda}$, and
\begin{eqnarray*}
& &\lim_{s\to\pm \infty} |\f_{H_\lambda}(\u(s,\cdot))| = |c_\pm (\u)| \le C(J,\tilde J,H,\tilde H),\\
& &\int_\R\int_0^1 |\u_s|^2 dt ds = \sum_{k=1}^n\int_\R\int_0^1 |u_s^k|^2 dt ds \le C'(J'\tilde J,H,\tilde H).
\end{eqnarray*}
Moreover, $ \lim_{s\to -\infty} \u(s,\cdot) \in \P_H, \quad \text{and}
\quad \lim_{s\to +\infty} \u(s,\cdot) \in \P_{\widetilde H}$.
\end{proposition}
\begin{proof}
Compactness follows from the estimates in Section
\ref{subsec:comp2} and the compactness in Proposition
\ref{prop:CR4}. Due to genericity, bounded solutions have limits in
$\P_H\cup \P_{\widetilde H}$, see Corollary \ref{cor:isol2}.
\end{proof}

We define a homomorphism $h_k = h_k(J_\lambda,H_\lambda)$ as
follows:
$$
h_k \x = \sum_{\x'\in \P_{\widetilde H}\atop \mu (\x') =k}
                              n(\x,\x';J_\lambda,H_\lambda)  \x',
$$
where $n(\x,\x';J_\lambda,H_\lambda) = \Bigl[\#
~\wM^{J_\lambda,H_\lambda}_{\x,\x'}\Bigr] ~{\rm mod}~ 2 \in \Z_2$.
Using similar gluing constructions and the isolation of the sets
$\cS^{J,H}$ and $\cS^{J',H'}$,
it is straightforward to show that the
mappings $h_k$ are chain homomorphisms and induce a
homomorphisms $h_k^*$ on Floer homology:
$$
h_k^*(J_\lambda,H_\lambda):
{\rm HF}_*(\x\rel\y;J,H) \to {\rm HF}_*(\x\rel\y;\widetilde J,\widetilde H).
$$
Further analysis of the non-autonomous CRE and standard procedures
in Floer theory show that any two homotopies
$(J_\lambda,H_\lambda)$ and $(\hat J_\lambda,\hat H_\lambda)$
between $(J ,H)$ and $(\widetilde J,\widetilde H)$ descend to the
same homomorphism in Floer homology:

\begin{proposition}
\label{lem:fl2} For any two homotopies $(J_\lambda,H_\lambda)$ and
$(\hat J_\lambda,\hat H_\lambda)$ between $(J ,H)$ and
$(\widetilde J,\widetilde H)$,
$$
h_k^*(J_\lambda,H_\lambda) = h_k^*( \hat J_\lambda,\hat H_\lambda).
$$
Moreover, for a homotopy $(J_\lambda,H_\lambda)$ between
$(J,H)$ and $(\widetilde J,\widetilde H)$, and a homotopy $(\hat
J_\lambda,\hat H_\lambda)$ between  $(\widetilde J,\widetilde H)$
and $(\check{J},\check{H})$, the induced homomorphism between the
Floer homologies is given by
$$
h_k^*: {\rm HF}_*([\x]\rel\y,J,H)\rightarrow {\rm HF}_*([\x]\rel\y,\check{J},\check{H}),
$$
where $h_k^* = h_k^*( \hat J_\lambda,\hat H_\lambda)\circ
h_k^*(J_\lambda,H_\lambda)$ and $h_k^*$ is thus an isomorphism.
\end{proposition}

\begin{proof}[~ of Proposition \ref{prop:conti1}]
Consider the  homotopties
$$
h_k^*: {\rm HF}_*(\x\rel\y;J,H) \to {\rm HF}_*(\x\rel\y;\widetilde J,\widetilde H),
$$
and
$$
\bar h_k^* : {\rm HF}_*(\x\rel\y;\widetilde J,\widetilde H) \to {\rm HF}_*(\x\rel\y;J,H),
$$
then
$$
\bar h_k^* \circ h_k^*: {\rm HF}_*(\x\rel\y;J,H) \to {\rm HF}_*(\x\rel\y;J,H).
$$
Since a homotopy from $(J,H)$ to itself induces the identity
homomorphism on homology, it holds that $\bar h_k^* \circ h_k^* =
{\rm Id}$. By the same token it follows that $h_k \circ \bar h_k^* =
{\rm Id}$, which proves $\bar h_k^* =(h^*_k)^{-1}$ and thus the
proposition.
\end{proof}

\subsection{Admissible pairs and independence of the skeleton}
\label{sec:admis}

By Proposition \ref{prop:conti1} the Floer homology of $[\x] \rel\y$ is
independent of a generic pair $(J,H)$, which justifies the notation
${\rm HF}_*([\x]\rel\y;\Z_2)$. It remains to show that, firstly, for any
braid class $[\x]\rel\y$ a pair exists, and thus the Floer homology is
defined, and secondly that the Floer homology depends only on the
braid class $[\x\rel\y]$.

Given a skeleton $\y$,
Proposition \ref{prop:ham-class2} implies the
existence of an appropriate Hamiltonian $H$
having invariant set $A_m$ with $\psi_{t,H}(A_m)\sim \y(t)$, i.e.
$[\psi_{t,H}(A_m)] = [\y]$. If $\y$ is a smooth representative of its
braid class, then $H$ can be chosen such that $\psi_{t,H}(A_m) =
\y(t)$.\footnote{A Hamiltonian function can also be constructed by
choosing appropriate cut-off functions in a neighborhood of $\y$.}
This establishes that ${\rm HF}_*([\x]\rel\y)$ is well-defined for any
proper relative braid class $[\x] \rel\y \in \Omega^n\rel \y$, with $\y
\in \Omega^m\cap C^\infty$. We still need to establish independence
of the braid class in $[\x\rel\y]$, i.e., that the Floer homology is the
same for any two relative braid classes $[\x]\rel\y$, $ [\x']\rel\y'$
such that $[\x\rel\y]   = [\x'\rel \y']$. This leads to the first main
result of this paper.

\begin{THM}
\label{thm:main1}
Let $[\x\rel\y]$ be a proper relative braid class. Then,
$$
{\rm HF}_*([\x]\rel\y) \cong {\rm HF}_*([\x']\rel\y'),
$$
for any  two fibers $[\x]\rel\y$ and $ [\x']\rel\y'$ in $[\x\rel\y]$. In
particular,
$$
    \bH_*([\x\rel\y])\bydef {\rm HF}_*([\x]\rel\y)
$$
is an invariant of $[\x\rel\y]$.
\end{THM}
\begin{proof} Let $\y,\y'\in \Omega^m\cap C^\infty$ and let
$\left(\x(\lambda),\y(\lambda)\right)$, $\lambda \in [0,1]$ be
a smooth path
$[\x\rel\y]$ which connects the pairs $\x\rel\y$ and $\x'\rel\y'$. Since
$\x(\lambda) \rel \y(\lambda)\in [\x\rel\y]$, for all $\lambda\in
[0,1]$, the sets $\cN_\lambda  = [\x(\lambda)] \rel \y(\lambda)$ are
isolating neighborhoods for all $\lambda$. Choose smooth
Hamiltonians $H_\lambda$ such that $\y(\lambda) \in
\P_{H_\lambda}$. There are two philosophies one can follow to prove
this theorem. On the one hand, using the genericity  theory in Section
\ref{sec:trans} (Corollary~\ref{cor:trans7}), we can choose a generic
family $(J_\lambda,H_\lambda)$ for any smooth homotopy of almost
complex structures $J_\lambda$. Then by repeating the proof (of
Proposition~\ref{prop:conti1}) for this homotopy, we conclude that
${\rm HF}_*([\x]\rel\y) \cong {\rm HF}_*([\x']\rel\y')$. On the other
hand, without having to redo the homotopy theory we note that
$\cS^{J_\lambda,H_\lambda}([\x(\lambda)] \rel \y(\lambda))$ is
compact and isolated in $\cN_\lambda$; thus, there exists an
$\epsilon_\lambda$ for each $\lambda\in [0,1]$ such that
$\cN_\lambda$ isolates $\cS([\x(\lambda')] \rel \y(\lambda'))$ for all
$\lambda'$ in $[\lambda-\epsilon_\lambda,
\lambda+\epsilon_\lambda]$. Fix $\lambda_0 \in (0,1)$; then, by
arguments similar to those of Proposition \ref{prop:conti1}, we have
$$
    {\rm HF}_*(\cN_{\lambda_0},J_{\lambda_0},H_{\lambda_0} ) \cong
    {\rm HF}_*(\cN_{\lambda_0},J_{\lambda'},H_{\lambda'} ),
$$
for all $\lambda' \in  [\lambda_0-\epsilon_{\lambda_0},
\lambda_0+\epsilon_{\lambda_0}]$. A compactness argument shows
that, for $\epsilon'_{\lambda_0}$ sufficiently small, the sets of
bounded solutions $\M^{J_{\lambda'},H_{\lambda'} }(\cN_{\lambda'})$
and $\M^{J_{\lambda'},H_{\lambda'} }(\cN_{\lambda_0})$ are
identical, for all $\lambda' \in  [\lambda_0-\epsilon'_{\lambda_0},
\lambda_0+\epsilon'_{\lambda_0}]$. Together these imply that
$$
\bH_*\bigl([\x(\lambda') \rel \y(\lambda')]\bigr)
\cong\bH_*\bigl([\x(\lambda_0) \rel \y(\lambda_0)]\bigr)
$$
for $|\lambda'-\lambda_0| \leq
\min\{\epsilon_{\lambda_0},\epsilon'_{\lambda_0}\}$. Since $[0,1]$ is
compact, any covering has a finite subcovering, which proves that
${\rm HF}_*([\x]\rel\y) \cong {\rm HF}_*([\x']\rel\y')$.

Finally, since any skeleton $\y$ in $\pi\left( [ \x\rel\y]\right)$ can be
approximated by a smooth skeleton $\y'$, the isolating neighborhood
$\cN = \pi^{-1}(\y) \cap {[\x\rel\y]}$ is also isolating for $\y'$, i.e., we
can define ${\rm HF}_*(\cN)\bydef {\rm HF}_*(\cN')$. This defines
${\rm HF}_*([\x]\rel\y)={\rm HF}_*(\cN)$ for any $\y \in
\pi([\x\rel\y])$.
\end{proof}


\section{Properties and interpretation of the braid class invariant}
\label{sec:properties} The braid Floer homology $\bH_*([\x\rel\y])$
entwines braiding and dynamical features of solutions of the Hamilton
equations \rmref{eqn:HE} on the 2-disc $\D^2$. One such property ---
the non-triviality of the invariant yields braided solutions --- will form
the basis of a forcing theory.
\begin{THM}
\label{thm:exist1} Let $H\in \cH$ and let $\y \in \P_H(\overline
\Omega^m)$. Let $[\x\rel\y]$ be a proper relative braid class. If
$\bH_*([\x\rel\y]) \not = 0$, then $\P_H([\x]\rel\y) \not =
\varnothing$.
\end{THM}

\begin{proof}
Let $H_n \in \cH$ be a sequence of Hamiltonians such that $H_n \to
H$, i.e. $H_n -H \to 0$ in $C^\infty$, see Sect. \ref{sec:trans}. If
$\bH_*([\x\rel\y]) \not = 0$, then $C_*([\x]\rel\y,H_n;\Z_2) \not = 0$
for any $n$, since
$$
 H_*\bigl(C_*([\x]\rel\y,H_n;\Z_2),\partial_*\bigr) \cong {\rm {\rm FH}}_*([\x\rel\y]) \not = 0,
$$
where $\partial_* = \partial_*(J,H_n)$
(see Section \ref{sec:defn-floer}). Consequently, $\P_{H_n}([\x]\rel\y) \not =
\varnothing$. The strands $x_n^k$ satisfy the equation $(x_n^k)' =
X_{H_n}(t,x_n^k)$ and therefore $\Vert x_n^k\Vert_{C^1([0,1])} \le
C$. By the compactness of $C^1([0,1]) \hookrightarrow C^0([0,1])$ it
follows that (along a subsequence) $x^k_n \to x^k \in C^0([0,1])$.
The right hand side of the Hamilton equations now converges
$X_{H_n}(t,x_n^k(t)) \to X_H(t,x(t))$ pointwise in $t\in [0,1]$; thus
$x_n^k \to x^k$ in $C^1([0,1])$. This holds for any strand $x^k$ and
therefore produces a limit $\x\in \P_H([\x]\rel\y)$.
\end{proof}

Let $\beta_k = \dim \bH_k([\x\rel\y];\Z_2)$ be the $\Z_2$-Betti
numbers of the braid class invariant. Its Poincar\'e series s is defined
as
$$
    P_t([\x\rel\y]) = \sum_{k\in \Z} \beta_k([\x\rel\y]) t^k.
$$

\begin{THM}
 \label{thm:exist2}
The braid Floer homology of any proper relative braid class is finite.
\end{THM}
\begin{proof}
Assume without loss of generality that $\y$ is a smooth skeleton and
choose a smooth generic Hamiltonian $H$ such that $\y \in \P_H$.
Since the Floer homology is the same for all braid classes
$[\x]\rel\y\in [\x\rel\y]$ and all Hamiltonians $H$ satisfying the
above: ${\rm FH}_*([\x]\rel\y,J,H,\bh) \cong {\rm FH}_*([\x\rel\y])$.
Let $c_k = \dim C_k$; then,
$$
c_k([\x]\rel\y,H) \ge \dim \ker C_k \ge \beta_k([\x]\rel\y,J,H) = \beta_k([\x\rel\y]).
$$
Since $H$ is generic it follows from compactness that $\sum_k
c_k<\infty$. Therefore $c_k<\infty$ and $c_k\not = 0$ for finitely
many $k$. By the above bound $\beta_k\le c_k<\infty$.
\end{proof}

In the case that $H$ is a generic Hamiltonian a more detailed result
follows. Both $\bigoplus_k {\rm FH}_k([\x\rel\y];\Z_2)$ and
$\bigoplus_k C_k([\x]\rel\y,H;\Z_2)$ are graded $\Z_2$-modules,
their Poincar\'e series are well-defined, and
$$
    P_t\bigl( \P_H([\x]\rel\y)\bigr) = \sum_{k\in \Z} c_k([\x]\rel\y,H) t^k,
$$
where $c_k =  \dim C_k([\x]\rel\y,H;\Z_2)$.
\begin{THM}
 \label{thm:exist3}
Let $[\x\rel\y]$ be a proper relative braid class and $H$ a generic
Hamiltonian such that $\y\in \P_H$\footnote{We do not assume that
$\y$ is a smooth skeleton.} for a given skeleton $\y$. Then
\begin{equation}
 \label{eqn:MR1}
    P_t\bigl( \P_H([\x]\rel\y)\bigr) = P_t([\x\rel\y]) + (1+t) Q_t,
\end{equation}
where $Q_t \ge 0$. In addition, $\# ~\P_H([\x]\rel\y) \ge
P_1([\x\rel\y])$.
\end{THM}
\begin{proof}
Let $\y'$ be a smooth skeleton that approximates $\y$ arbitrarily
close in $C^2$ and let $H'$ be  an associated smooth generic
Hamiltonian. We start by proving \rmref{eqn:MR1} in the smooth
case. Define $Z_k = \ker \partial_k$, $B_k = {\rm im~} \partial_{k+1}$
and $Z_k\subset B_k \subset C_k([\x]\rel\y',H')$ by the fact that
$\partial_*$ is a boundary map. This yields the following short exact
sequence
$$
\begin{CD}
0 @>{\rm Id}>>  B_k @>i_k>> Z_k @>j_k>> {\rm FH}_k =\displaystyle\frac{Z_k}{B_k} @>0>> 0.\\
\end{CD}
$$
The maps $i_k$ and $j_k$ are defined as follows: $i_k(\x) = \x$ and
$j_k(\x) = \{\x\}$, the equivalence class in ${\rm FH}_k$. Exactness is
satisfied since $\ker i_k = 0 ={\rm im~} {\rm Id}$, $\ker j_k = B_k =
{\rm im~} i_k$ and $\ker 0 = {\rm FH}_k = {\rm im~}j_k$. Upon
inspection of the short exact sequence we obtain
$$
    \dim Z_k = \dim B_k + \dim {\rm FH}_k.
$$
Indeed, by exactness, $Z_k \supset \ker j_k = B_k$ and ${\rm im~}j_k
= {\rm FH}_k$ (onto) and therefore $\dim Z_k = \dim \ker j_k + \dim
{\rm im~} j_k =  \dim B_k + \dim {\rm FH}_k$. Since $C_k \cong Z_k
\oplus B_{k-1}$ it holds that
$$
    \dim C_k = \dim Z_k + \dim B_{k-1}.
$$
Combining these equalities gives $\dim C_k = \dim {\rm FH}_k + \dim
B_{k-1} + B_{k}$. On the level of Poincar\'e series this gives
$$
    P_t\l(\oplus_k C_k\r) =  P_t\l(\oplus_k {\rm FH}_k\r) + (1+t)  P_t\l(\oplus_k B_k\r),
$$
which proves \rmref{eqn:MR1} in the case of smooth skeletons.

Now choose sequences $\y_n \to \y$  and $H_n \to H$ in $C^\infty$
($H_n$ generic). For each $n$ the above identity is satisfied and
since also $H$ is generic it follows from hyperbolicity that $P_t\bigl(
\P_{H_n}([\x]\rel\y_n)\bigr) = P_t\bigl( \P_H([\x]\rel\y)\bigr) $ for $n$
large enough. This then proves \rmref{eqn:MR1}. By substitution
$t=1$ and using the fact that all series are positive  gives the lower
bound on the number of stationary braids.
\end{proof}

An important question is whether $\bH_*([\x\rel\y])$ also contains
information about $\P_H([\x]\rel\y)$ in the non-generic case besides
the result in Theorem \ref{thm:exist1}. In \cite{GVV} such a result
was indeed obtained and detailed study of the spectral properties of
stationary braids will most likely reveal a similar property. We
conjecture that $\# ~\P_H([\x]\rel\y) \ge {\rm
length}\bigl(\bH_*([\x\rel\y])\bigr)$, where ${\rm length}(\bH_*)$
equals the number of monomial terms in $P_t([\x\rel\y])\bigr)$.


\section{Homology shifts and Garside's normal form}
\label{subsec:twists}
\label{sec:garside}
\label{subsec:garside}
In this section we show that composing a braid
class with full twists yields a shift in braid Floer homology.
Consider the symplectic twist $S: [0,1] \to {\rm Sp}(2,\R)$ defined by
$S(t) = e^{2\pi J_0 t}$, which rotates the variables counter clock wise
over $2\pi$ as $t$ goes from $0$ to $1$. On the product $\R^2\times
\cdots \times \R^2 \cong \R^{2n}$ this yields the product rotation
$\overline S(t) = e^{2\pi\overline J_0 t}$ in  ${\rm Sp}(2n,\R)$.

Lifting to the Hamiltonian gives
$\overline S\x \in \P_{\overline S H}$, where the rotated Hamiltonian
$\overline S H \in \cH$ is given by
$\overline S H(t,\overline S\x) = H\bigl(t,\x\bigr) + \pi\bigl(|\overline S x|^2
-1\bigr)$. Substitution yields the transformed Hamilton equations:
\begin{equation}
\label{eqn:ham-trans}
(\overline S \x)_t  - \overline S J_0 \nabla H(t, \x) - 2\pi J_0 \overline S \x = 0,
\end{equation}
which are the Hamilton equations for $\overline S H$. This twisting induces a
a shift between the Conley-Zehnder indices $\mu(\x)$ and
$\mu(\overline S x)$:
\begin{lemma}
 \label{lem:shift1}
For $\x\in\P_H$, $\mu(\overline S \x) = \mu(\x) + 2n$,
where $n$ equal the number of strands in $\x$.
\end{lemma}

\begin{proof}
In Definition \ref{defn:mas3} the
Conley-Zehnder index of a stationary braid $\x\in \P_H$ was given
as the permuted Conley-Zehnder index of the symplectic path $\Psi:
[0,1] \to {\rm Sp}(2n,\R)$ defined by
\begin{equation}
\label{eqn:lin1}
\frac{d \Psi}{dt}  - \overline J_0 d^2\overline{H}(t,\x(t)) \Psi=0,\qquad \Psi(0)= {\rm Id}.
\end{equation}
In order to compute the Conley-Zehnder index of $\overline S \x$ we linearize
Eq.\ \rmref{eqn:ham-trans} in $\overline S \x$, which yields
$$
\frac{d(\overline S(t) \Psi)}{dt}  - \overline S(t) J_0 d^2\overline{H}(t, \x(t))
 \Psi - 2\pi \overline J_0 \overline S(t) \Psi = 0,
 \qquad \overline S(0)\Psi(0)= {\rm Id}.
$$
From Lemma \ref{lem:mas1}(ii)  and the fact that $\mu(e^{2\pi J_0 t})
=2$ it follows that
\begin{eqnarray*}
\mu(\overline S\x) &=&  \mu_\sigma(\overline S \Psi,1)\\
&=& \mu_\sigma( \Psi,1) + \mu(\overline S) = \mu_\sigma(  \Psi,1) + n\mu(e^{2\pi J_0 t}) \\
&=& \mu(\x) + 2n,
\end{eqnarray*}
which proves the lemma.
\end{proof}

We relate the Floer homologies of
$[\x\rel\y]$ with Hamiltonian  $H$ and $[\overline S \x \rel \overline S\y]$ with
Hamiltonian $\overline S H$ via the index shift in Lemma
\ref{lem:shift1}. Since the Floer homologies do not depend on the
choice of Hamiltonian we obtain the following:
\begin{THM}
\label{thm:shift5} {\bf [Shift Theorem]} Let $[\x\rel\y]$ denote a
braid class with $\x$ having $n$ strands. Then
\[
    \bH_*\bigl([(\x\rel\y)\cdot \Delta^{2}]\bigr) \cong
    \bH_{*-2n}([\x\rel\y]) .
\]
\end{THM}
\begin{proof}
It is clear that the application of $\overline S$ acts on braids by
concatenating with the full positive twist $\Delta^2$. As $\Delta^2$
generates the center of the braid group, we do not need to worry
about whether the twist occurs before, during, or after the braid. It
therefore suffices to show that $\bH_*([\overline S\x\rel\overline S\y])
\cong \bH_{*-2n}([\x\rel\y])$.

The Floer homology for $[\x\rel\y]$ is defined by choosing a generic
Hamiltonian $H$. From Lemma \ref{lem:shift1} we have that
$\mu(\overline S\x) = \mu(\x) + 2n$ and therefore
$$
C_k([\overline S\x]\rel\overline S\y,\overline S H;\Z_2) = C_{k-2n}([\x]\rel\y,H;\Z_2).
$$
Since the solutions in $\M^{J,\overline S H}$
are obtained via $\overline S$, it also also holds that
$$
\partial_k(J,\overline S H) = \partial_{k-2n}(J,H),
$$
and thus $FH_k([\overline S\x]\rel\overline S\y) \cong FH_{k-2n}([\x]\rel\y)$.
\end{proof}

Recall that a positive braid is one all of whose crossings are of the
same (`left-over-right') sign; equivalently, in the standard (Artin)
presentation of the braid group $\cB_n$, only positive powers of
generators are utilized. Positive braids possess a number of remarkable
and usually restrictive properties. Such is not the case for braid Floer
homology.

\begin{corollary}
\label{cor:garside}
Positive braids realize, up to shifts, all possible braid Floer homologies.
\end{corollary}
\begin{proof}
It follows from \style{Garside's Theorem} \cite{Garside,Birman} that
every braid $\beta\in\cB_n$ has a unique presentation as the product
of a positive braid along with a (minimal) number of {\em negative}
full twists $\Delta^{-2g}$ for some $g\geq 0$. From Theorem
\ref{thm:shift5}, the braid Floer homology of any given relative braid
class is equal to that of its (positive!) Garside normal form, shifted to
the left by degree $2gn$, where $n$ is the number of free strands.
\end{proof}

This reduces the problem of computing braid Floer homology to the subclass
of positive braid pairs. We believe this to be a considerable simplification.


\section{Cyclic braid classes and their Floer homology}
 \label{sec:cyclic}
In this section we compute examples of braid Floer homology for
cyclic type braid classes. The cases we consider can be computed by
continuing the skeleton and the Hamiltonians to a Hamiltonian
system for which the space of bounded solutions can be determined
explicitly: the integrable case.

\subsection{Single-strand rotations and symplectic polar coordinates}
 \label{subsec:single}
Choose complex coordinates $x=p+iq$ and consider Hamiltonians of
the form
\begin{equation}
\label{eqn:can-1}
    H(x) = F(|x|) + \phi_\delta(|x|)G\bigl(\arg(x)\bigr),
\end{equation}
where $\arg(x)=\theta$ is the argument and $G(\theta+2\pi) =
G(\theta)$. The cut-off function $\phi_\delta$ is chosen such that
$\phi_\delta(|x|) =0$ for $ |x|\le \delta$ and $|x|\ge 1-\delta $, and
$\phi_\delta(|x|) =1$ for $2\delta \le |x|\le 1-2\delta$. In the special
case that $G(\theta)\equiv 0$, then the Hamilton equations are given
by
$$
    x_t = i \nabla H(x) = i \frac{F'(|x|)}{|x|} x,
$$
Solutions of the Hamilton equations are given by $x(t) = r
\exp{\l(i\tfrac{F'(r)}{r} t\r)}$, where $r=|x|$. This gives the period $T = \frac{2\pi r}{F'(r)}$. Since
$H$ is autonomous all solutions of the Hamilton equations occur as
circles of solutions. The CRE are given by
\begin{equation}
\label{eqn:CR-aut}
u_s -i u_t - \nabla H(u) = 0.
\end{equation}
Consider the natural change to symplectic polar coordinates
$(I,\theta)$ via the relation $p = \sqrt{2I}\cos(\theta)$, $q =
\sqrt{2I}\sin(\theta)$, and define $\hat H(I,\theta) = H(p,q)$. In
particular, $ \hat H(I,\theta) = F(\sqrt{2I}) +
\phi_\delta(\sqrt{2I})G(\theta)$. The CRE become
\begin{eqnarray*}
\label{eqn:CR7}
    I_s +2I \theta_t - 2I\hat H_I (I,\theta)&=& 0,\\
    \theta_s - \frac{1}{2I} I_t - \frac{1}{2I} \hat H_\theta (I,\theta)&=& 0.
\end{eqnarray*}
If we restrict $x$ to the annulus $\A_{2\delta} = \{x\in \D^2:
2\delta\le |x| \le 1-2\delta\}$, the particular choice of $H$ described
above yields
\begin{eqnarray*}
    I_s + 2I \theta_t - \sqrt{2I} F'(\sqrt{2I})  &=&0,\\
    \theta_s - \frac{1}{2I} I_t - \frac{1}{2I} G'(\theta) &=& 0.
\end{eqnarray*}
Before giving a general result for braid classes for which $\x$ is a
single-strand rotation we employ the above model to get insight into
the Floer homology of the annulus.

\subsection{Floer homology of the annulus}
\label{subsec:annulus}
Consider an annulus $\A=\A_\delta$ with Hamiltonians $H$ satisfying
the hypotheses:
\begin{enumerate}
\item[(a1)] $H\in C^\infty(\R\times \R^2;R)$;
\item[(a2)] $H(t+1,x) = H(t,x)$ for all $t\in \R$ and all $x\in \R^2$;
\item[(a3)] $H(t,x) =0$ for all $x\in \partial\A$ and all $t\in \R$.
\end{enumerate}
This class of Hamiltonians is denoted by $\cH(\A)$. We consider Floer
homology of the annulus in the case that $H$ has prescribed behavior
on $\partial\A$. The boundary orientation is the canonical Stokes
orientation and the orientation form on $\partial \A$ is given by
$\lambda = i_{\mathbf n}\omega$, with ${\mathbf n}$ the outward
pointing normal. We will consider the Floer homology of the annulus
in the case that $H$ has prescribed behavior on $\partial \A$:
\begin{enumerate}
\item[(a4$^+$)] $i_{X_H}\lambda >0$ on $\partial \A$;
\item[(a4$^-$)] $i_{X_H}\lambda <0$ on $\partial \A$.
\end{enumerate}
The class of Hamiltonians that satisfy (a1)-(a3), (a4$^+$) is denoted
by $\cH^+$ and those satisfying (a1)-(a3), (a4$^-$) are denoted by
$\cH^-$. For Hamiltonians in $\cH^+$ the boundary orientation
induced by $X_H$ is coherent with the canonical orientation of
$\partial\A$, while for Hamiltonians in $\cH^-$ the boundary
orientation induced by $X_H$ is opposite to the canonical orientation
of $\partial\A$.

For pairs $(J,H) \in \cJ\times \cH^+(\A)$ let ${\rm HF}_*^+(\A;J,H)$
denote the Floer homology for contractible loops in $\A$. Similarly,
for $H \in \cH^-$ the Floer homology is denoted by ${\rm
HF}_*^-(\A;J,H)$. For Hamiltonians of the form \rmref{eqn:can-1} it
can also be interpreted as the Floer homology of the space of single
strand braids $\x$ that wind zero times around the annulus, i.e. any
constant strand is a representative. In view of (a4$^\pm$) this braid
class is proper.
\begin{THM}
 \label{thm:annulus1}
The Floer homology ${\rm HF}_*^+(\A;J,H)$ is independent of the pair
$(J,H) \in \cJ\times \cH^+(\A)$ and is denoted by ${\rm HF}_*^+(\A)$.
There is a natural isomorphism
$$
{\rm HF}_k^+(\A) \cong H_{k+1}(\A,\partial \A) =\begin{cases}
  \Z_2    & \text{for  } k=0,1 \\
    0  & \text{otherwise}.
\end{cases}
$$
Similarly, the Floer homology ${\rm HF}_*^-(\A;J,H)$ is independent
of the pairs $(J,H) \in \cJ\times \cH^-(\A)$ and is denoted by ${\rm
HF}_*^-(\A)$ and there is a natural isomorphism
$$
{\rm HF}_k^-(\A) \cong H_{k+1}(\A) =\begin{cases}
  \Z_2    & \text{for  } k=-1,0 \\
    0  & \text{otherwise},
\end{cases}
$$
where $H_*$ denotes the singular homology with coefficients in
$\Z_2$.
\end{THM}
\begin{proof}
Let us start with Hamiltonians in the class $\cH^+$. Consider $\A =
\A_\delta$ and choose $H = F + \phi_\delta G$, with $F(r) = \frac{1}{2}
\l(r-\frac{1}{2}\r)^2 - \frac{1}{2} \l(\delta-\frac{1}{2}\r)^2$ and
$G(\theta)  = \epsilon \cos(\theta)$. Using symplectic polar
coordinates we obtain that
$$
\nabla_g \hat H(I,\theta)
=
\left(\begin{array}{c}2I - \frac{1}{2} \sqrt{2I}
+ \epsilon \sqrt{2I} \phi_\delta'(\sqrt{2I}) \cos(\theta)
\\ -\frac{\epsilon}{2I} \phi_\delta(\sqrt{2I})\sin(\theta)
\end{array}\right).
$$
For  $\frac{1}{2}\delta^2\le I \le 2\delta^2$ and for
$\frac{1}{2}(1-2\delta)^2 \le I \le \frac{1}{2} (1-\delta)^2$ it holds
that $|\sqrt{2I} - \frac{1}{2}  | \ge \frac{1}{2}-2\delta$ and thus if we
choose $\epsilon<\frac{1}{4\delta}-1$ all zeroes of $\nabla_g \hat H$
lie in the annulus set $\A_{2\delta}\subset \A_\delta$. The zeroes of
$\nabla_g\hat H$ are found at $I =\frac{1}{8}$ and $\theta = 0,\pi$,
which are both non-degenerate critical points. Linearization yields
$$
d\nabla_g \hat H({1/8},0) = \left(\begin{array}{cc}1 & 0 \\0 & -4\epsilon\end{array}\right),
\quad d\nabla_g \hat H\l({1/8},\pi\r) = \left(\begin{array}{cc}1 & 0 \\0 & 4\epsilon\end{array}\right),
$$
i.e. a saddle point (index 1)  and a minimum (index 0) of $H$
respectively. For $\tau\le 1$  it follows from Remark
\ref{rmk:withmorse} that the Conley-Zehnder indices of the
associated symplectic paths defined by $\Psi_t = J_0 d\nabla_g \hat
H\Psi$ are given by $\mu_\sigma(\Psi,\tau) = 0,1$. Therefore the
index $  \mu_\sigma(\Psi,\tau) =  1- \mu_H = 0,1$ for $(I,\theta)$
equal to $(\frac{1}{8},0)$ and $(\frac{1}{8},\pi)$ respectively.

Next consider Hamiltonians of the form $\tau H$ and the associated
CRE $u_s - J_0 u_t - \tau \nabla H(u) =0$. Rescale $\tau s \to s$,
$\tau t \to t$ and $u(s/\tau,t/\tau) \to u(s,t)$; then $u$ satisfies
\rmref{eqn:CR-aut} again with periodicity $u(s,t+\tau) = u(s,t)$. The
1-periodic solutions of the CRE with $\tau H$ are transformed to
$\tau$-periodic solutions of \rmref{eqn:CR-aut}. Note that if $\tau$ is
sufficiently small then all $\tau$-periodic solutions of the stationary
CRE are independent of of $t$ and thus critical points of $H$.

If we linearize around $t$-independent solutions of
\rmref{eqn:CR-aut}, then $\frac{d}{ds} -d\nabla H(u(s))$ is Fredholm
and thus also
$$
 \partial_{K,\Delta} = \frac{\partial}{\partial s} - J\frac{\partial}{\partial t} - K,
$$
with $K = d\nabla H(u(s))$, is Fredholm, see \cite{SalZehn1}. We
claim that if $\tau$ is sufficiently small then all contractible
$\tau$-periodic bounded solutions $u(s,t+\tau) = u(s,t)$ of
\rmref{eqn:CR-aut} are $t$-independent, i.e. solutions of the
equation $u_s  = \nabla H(u)$. Let us sketch the argument following
\cite{SalZehn1}. Assume by contradiction that there exists a sequence
of $\tau_n\to 0$ and bounded solutions $u_n$ of Equation
\rmref{eqn:CR-aut}. If we embed $\A$ into the 2-disc $\D^2$ we can
use the compactness results for the 2-disc. One can assume without
loss of generality that $u_n \in \cM^{J_0,H}(\A;\tau) = \cM^{J_0,\tau
H}(\A)$. Following the proof in \cite{SalZehn1} we conclude that for
$\tau>0$ small all solutions in $\M$ are $t$-independent. The system
$(J,H)$ can be continued to $(J_0,\tau H)$ for which we know
$\M^{J_0,\tau H}(\A)$ explicitly via $u_s - \tau \nabla H =0$ and
therefore the desired homology is found as follows.

Note that $\A$ is an isolating neighborhood for the gradient flow
generated by $u_s = \tau \nabla H(u) =f(u)$, and for $H\in \cH^+$ the
exit set is $\partial \A$. Using the Morse  relations for the Conley
index we obtain for any generic $H\in \cH^+$ that
$$
\sum_{x\in {\rm Fix}(f)} t^{{\rm ind}(x)} = P_t(\A,\partial\A) + (1+t)Q_t,
$$
where ${\rm ind}(x) = \dim W^u(x)$. Using \cite{Sal} the
Poincar\'e polynomial follows, provided we choose the appropriate grading. By
the previous considerations on the Conley-Zehnder index in Remark
\ref{rmk:withmorse} we have that $ \mu_\sigma(\Psi,\tau) =  1 -
\mu_H(x) = 1 - (2- {\rm ind}(x)) = {\rm ind}(x) -1$. This yields ${\rm
ind}(x) = k+1$, where $k$ is the grading of Floer homology and
therefore $H_k(C,\partial;J_0,\tau H) \cong H_{k+1}(\A,\partial \A)$,
which proves the first statement.

As for Hamiltonians in $\H^-$ we choose $F(r) =  -\frac{1}{2}
\l(r-\frac{1}{2}\r)^2 + \frac{1}{2} \l(\delta-\frac{1}{2}\r)^2$. The proof
is identical to the previous case except for the indices of the
stationary points. The Morse indices of $(\frac{1}{8},0)$ and
$(\frac{1}{8},\pi)$ are $1$ and $2$ and $\mu_\sigma(\Psi,\tau) =
1-\mu_H =  0,-1$ for $(I,\theta)$ equal to $(\frac{1}{8},0)$ and
$({1}{8},\pi)$ respectively. As before $\A$ is an isolating
neighborhood for the gradient flow of $\tau H$ and for $H\in \cH^-$
the exit set is $\varnothing$. This yields the slightly different Morse
relations
$$
\sum_{x\in {\rm Fix}(f)} t^{{\rm ind}(x)} = P_t(\A) + (1+t)Q_t.
$$
By the same grading as before we obtain that
$H_k(C,\partial;J_0,\tau H) \cong H_{k+1}(\A)$, which proves the
second statement.
\end{proof}

\subsection{Floer homology for single-strand cyclic braid classes}
\label{subsec:cyclic2}

We apply the results in the previous subsection to compute the Floer
homology of families of cyclic braid classes $[ \x\rel\y]$. The
skeletons $\y$ consist of two braid components $\y^1$ and $\y^2$,
which are given by (in complex notation)
\begin{equation}
\label{eqn:def1}
\y^1 =   \Bigl\{  r_1 e^{\tfrac{2\pi n}{m} i t},\cdots, r_1 e^{\tfrac{2\pi n}{m} i (t-m+1)}\Bigr\},\quad
\y^2 =   \Bigl\{  r_2 e^{\tfrac{2\pi n'}{m'} i t},\cdots, r_2 e^{\tfrac{2\pi n'}{m'} i (t-m'+1)}\Bigr\},
\end{equation}
where $0<r_1<r_2\le 1$, $m,m'\in \N$, $n,n'\in \Z$ and $n\not = 0$,
$m\ge 2$. Without loss of generality we take both pairs $(n,m)$ and
$(n',m')$ relatively prime. In the braid group $\cB_m$ the braid
$\y^1$ is represented by the word $\beta^1 = (\sigma_1 \cdots
\sigma_{m-1})^n$, $m\ge 2$, and $n\in \Z$, and similarly for $\y^2$.
In order to describe the relative braid class $[\x\rel\y]$ with the
skeleton defined above we consider a single strand braid $\x =
\{x^1(t)\}$ with
$$
    x^1(t)=re^{2\pi \ell i t}
$$
where $0<r_1 <r<r_2<1$ and $\ell\in \Z$. We now consider two cases
for which $\x\rel\y$ is a representative.

\subsubsection{The case $\frac{n}{m}<\ell<\frac{n'}{m'}$}
The relative braid class $[\x\rel\y]$ is a proper braid class since the
inequalities are strict.
\begin{lemma}
\label{lem:braids1}
The Floer homology is given by
$$
\bH_k([\x\rel\y],\Z_2) = \begin{cases}
  \Z_2    & \text{for  } k= 2\ell,2\ell +1 \\
   0   & \text{otherwise}.
\end{cases}
$$
The Poincar\'e polynomial is given by $P_t([\x\rel\y]) =
t^{2\ell}+t^{2\ell +1}$.
\end{lemma}
\begin{proof}
Since $\bH_*([\x\rel\y],\Z_2) $ is independent of the representative
we consider the class $[x]\rel\y$ with $\x$ and $\y$ as defined above.
Apply $-\ell$ full twists to $\x\rel\y$: $(\tilde \x,\tilde \y) = \overline
S^{-\ell} (\x,\y)$. Then by Theorem \ref{thm:shift5}
\begin{equation}
\label{eqn:braids2}
\bH_k([\tilde\x\rel\tilde\y]) \cong \bH_{k+2\ell}([\x\rel\y]).
\end{equation}
We now compute the homology ${\rm HF}_k([\tilde\x\rel\tilde\y])$
using Theorem \ref{thm:annulus1}. The free strand $\tilde \x$, given
by $\tilde x^1(t) = r$, in $\tilde\x\rel\tilde\y$ is unlinked with the
$\tilde y^1$. Consider an explicit Hamiltonian $H(x) = F(|x|) +
\omega(|x|) G(\arg(x))$ and choose $F$ such that $F(r_1) = F(r_2) =
0$ and
\begin{equation}
\label{eqn:choices1}
\frac{F'(r_1)}{2\pi r_1} = \frac{n}{m}-\ell<0,
\quad{\rm and}\quad
\frac{F'(r_2)}{2\pi r_2} = \frac{n'}{m'}-\ell>0.
\end{equation}
Clearly $\y \in \P_H$ and the circles $|x| = r_1$ and $|x| = r_2$ are
invariant for the Hamiltonian vector field $X_H$. For $\tau>0$
sufficient small it holds that $\cM^{J_0,\tau H}([\tilde\x]\rel\y) =
\cM^{J_0,H}(\A)$, and from the boundary conditions in Eq.\
\rmref{eqn:choices1} it follows that $ H\in \cH^+$. From Theorem
\ref{thm:annulus1} we deduce that ${\rm
HF}_{0}([\tilde\x]\rel\tilde\y) \cong {\rm HF}_{0}^+(\A) = \Z_2$ and
${\rm HF}_{1}([\tilde \x]\rel\tilde\y) \cong {\rm HF}_{1}^+(\A) =
\Z_2$. This proves, using Eq.\ \rmref{eqn:braids2}, that ${\rm
HF}_{2\ell}([\x]\rel\y) = \Z_2$ and ${\rm HF}_{2\ell+1}([\x]\rel\y) =
\Z_2$, which completes the proof.
\end{proof}

\subsubsection{The case $\frac{n}{m}>\ell>\frac{n'}{m'}$}
The relative braid class $[\x\rel\y]$ with the reversed inequalities is
also a proper braid class. We have
\begin{lemma}
\label{lem:braids3}
The Floer homology is given by
$$
\bH_k([\x\rel\y],\Z_2) = \begin{cases}
  \Z_2    & \text{for  } k=2\ell-1,2\ell \\
   0   & \text{otherwise}.
\end{cases}
$$
The Poincar\'e polynomial is given by $P_t([\x\rel\y]) =
t^{2\ell-1}+t^{2\ell}$.
\end{lemma}

\begin{proof}
The proof is identical to the proof of Lemma \ref{lem:braids1}.
Because the inequalities are reversed we construct a Hamiltonian
such that
$$
\frac{F'(r_1)}{2\pi r_1} = \frac{n}{m}-\ell>0,
\quad{\rm and}\quad
\frac{F'(r_2)}{2\pi r_2} = \frac{n'}{m'}-\ell<0.
$$
This yields a Hamiltonian in $\cH^-$ for which we repeat the above
argument using the homology ${\rm HF}_*^-(\A)$.
\end{proof}

\subsection{Applications to disc maps}
\label{subsec:disc-maps}
We demonstrate how these simple computed examples of $\bH_*$
yield forcing results at the level of dynamics. The following results
are not so much novel (cf. Franks' work on rotation sets) as
illustrative of how one uses a Floer-type forcing theory.

\begin{THM}
\label{thm:fixed1} Let $f:\D^2 \to \D^2$ be an area-preserving
diffeomorphism with invariant set $A\subset \D^2$ having as braid
class representative $\y$, where $[\y]$ is as described in Eq.
\rmref{eqn:def1}, with $\frac{n}{m} \not = \frac{n'}{m'}$ relatively
prime. Then, for for each $l\in \Z$ and $k\in \N$, satisfying
$$
\frac{n}{m} <\frac{l}{k}< \frac{n'}{m'},
\quad {\rm or} \quad
\frac{n}{m} > \frac{l}{k}> \frac{n'}{m'},
$$
there exists a distinct period $k$ orbit of $f$. In particular, $f$ has
infinitely many distinct periodic orbits.
\end{THM}

\begin{proof}
By Proposition \ref{prop:ham-class2} there exists a Hamiltonian $H\in
\cH(\D^2)$ such that $f=\psi_{1,H}$, where $\psi_{t,H}$ the
Hamlitonian flow generated by the Hamiltonian system $x_t =
X_H(t,x)$ on $(\D^2,\omega_0)$. Up to full twists $\Delta^2$, the
invariant  set $A$ generates a braid $\psi_{t,H}(A)$ of braid class
$[\psi_{t,H}(A)] = [\y]~{\rm mod~} \Delta^2$ with
$$
 \psi_{t,H}(A) = \tilde \y= \{\tilde \y^1,\tilde \y^2\}.
$$

We begin with the case $k=1$.
There exists an integer $N$ (depending on the choice of $H$) such
that the numbers of turns of the strands $\tilde y^1$ and $\tilde y^2$
are $\tfrac{\tilde n}{m} = \tfrac{n}{m} +N$ and $\tfrac{\tilde n'}{m'} =
\tfrac{n'}{m'} +N$ respectively. Consider a free strand $\tilde \x$ such
that $\tilde\x\rel\tilde\y  \sim (\x\rel\y)\cdot \Delta^{2N}$, with
$[\x\rel\y]$ as in Lemmas \ref{lem:braids1} and \ref{lem:braids3}, and
with $l$ satisfying the inequalities above.\footnote{Geometrically
$\tilde \x$  turns $l+N$ times around $\tilde\y^1$ and each strand in
$\tilde \y^2$ turns $\tfrac{\tilde n'}{m'}$ times around $\tilde\x$.} By
Lemmas \ref{lem:braids1} and \ref{lem:braids3} the Floer homology of
$[\x\rel\y]$ is non-trivial, and Theorem \ref{thm:shift5} implies that
$\bH_k\l([\tilde\x\rel\tilde\y\r]) \cong \bH_{k-2nN}([\x\rel\y])$.
Therefore the Floer homology of $[\tilde \x\rel\tilde\y]$ is non-trivial.
From Theorem \ref{thm:exist1} the existence of a stationary relative
braid $\tilde \x$ follows, which yields a fixed point for $f$.

For the case $k>1$,
consider the Hamiltonain $kH$; the time-1 map associated with
Hamiltonian system $x_t = X_{kH}$ is equal to $f^k$. The fixed point implied
by the proof above descends to a $k$-periodic point of $f$.
\end{proof}

\begin{remark}
\label{rmk:fixed3} As pointed out in Section \ref{sec:properties}, we
conjecture that the Floer homology $\bH_*([\x\rel\y]) = \Z_2\oplus
\Z_2$ implies the existence of at least two fixed points of different
indices. This agrees with 
the generic setting
where centers and saddles occur in pairs.
\end{remark}


\section{Remarks and future steps}
\label{sec:rems}
In this section we outline a number of remarks and future directions.
The results in this paper are a first step to a more in depth theory.

\subsection{Floer homology, Morse homology and the Conley index}
\label{subsec:FMC} The perennial problem with Floer homologies is
their general lack of computability.\footnote{Much excitement in Heegaard 
Floer homology surrounds recent breakthroughs in combinatorial formulae for 
(still challenging!) computation of examples.} We outline a strategy for
algorithmic computation of braid Floer homology.
\begin{enumerate}
\item Use Garside's Theorem and Theorem \ref{thm:shift5} to
    reduce computation to the case of positive braids.
\item Prove that braid Floer homology is isomorphic to the Conley
    braid index of \cite{GVV} in the case of positive braids.
\item Invoke the computational results in \cite{GVV}.
\end{enumerate}
Steps 1 and 3 above are in place; Step 2 is conjectural.

To be more precise, let $[\x\rel\y]$ be a proper relative braid class. It
follows from the Garside theorem that for $g\ge 0$ sufficiently large,
the braid class $[\p\rel\q]=[\x\rel\y]\cdot\Delta^{2g}$ is positive ---
there exist representatives of the braid class which are positive
braids. Given such a positive braid, its \style{Legendrian
representative} (roughly speaking, an image of the braid under a
certain planar projection, lifted back via a 1-jet extension) captures
the braid class. From \cite{GVV}, one starts with a Legendrian braid
representative and performs a spatial discretization, reducing the
braid to a finite set of points which can be reconstructed into a
piecewise-linear braid in $[\p\rel\q]$.

There is an analogous \style{homological index} for Legendrian braids,
introduced in \cite{GVV}. This index, ${\rm HC}_*([\p\rel\q])$, is
defined as the (homological) Conley index of the discretized braid
class under an appropriate class of parabolic dynamics. This index has
finite-dimensionality built in and computation of several classes of
examples have been implemented using current homology
computational code.

This index shares some features with the braid Floer homology.
Besides finite dimensionality of the index, there is a precise analogue
of the Shift Theorem for products with full twists. Even the underlying
dynamical constructs are consonant. For Legendrian braid classes the
same construction as in this paper can be carried out using a
nonlinear heat equation instead of the nonlinear CRE. Consider the
scalar parabolic equation
\begin{equation}
\label{eqn:HFe}
u_s - u_{tt} - g(t,u) = 0,
\end{equation}
where $u(s,t)$ takes values in the interval $[-1,1]$. Such equations
can be obtained as a limiting case of the nonlinear CRE. For the
function $g$ we assume the following hypotheses:
\begin{itemize}
\item [(g1)] $g\in C^\infty(\R\times \R;\R)$;
\item [(g2)] $ g(t+1,q)  = g(t,q)$  for all $(t,q)\in \R\times \R$;
\item [(g3)] $g(t,-1) = g(t,1) = 0$ for $t\in \R$.
\end{itemize}
This equation generates a local semi-flow $\psi^s$ on periodic
functions in $C^\infty(\R/\Z;\R)$. For a braid diagram $\p$ we define
the intersection number  $I(\p)$ as the total number of intersections,
and since all intersections in a Legendrian braid of this type
correspond to positive crossings, the total intersection number is
equal to the crossing number defined above. The classical
\style{lap-number property} \cite{Ang2} of nonlinear scalar heat
equations states that the number of intersections between two
graphs can only decrease as time $s\to \infty$. As before let
$[\p]_L\rel\q$ be a relative braid class fiber with skeleton $\q$; we
can choose a nonlinearity $g$ such that the skeletal strands in $\q$
are solutions of the
equation $q_{tt} + g(t,q) =0$. Let $\uu(s)\rel\q$ 
denote a local solution (in $s$) of the Eq.\
(\ref{eqn:HFe}), then
$I(\uu\rel\q)|_{s_0-\epsilon}>I(\uu\rel\q)|_{s_0+\epsilon}$,
whenever $u(s_0,t_0) = u^k(s_0,t_0)$ for some $k$, and as before we
define the sets of all bounded solutions in $[\p]_L\rel\q$, which we
denote by $\cM([\p]_L\rel\q)$. Similarly,
$\cS([\p]_L\rel\q)\subset C^\infty(\R/\Z;\R)$, the image
under the map $u\mapsto u(0,\cdot)$, is compact in the
appropriate sense. We can now build a chain complex in the usual way
which yields the Morse homology ${\rm HM}_*([\p\rel\q])$.

\begin{conjecture}
\label{conj:compute}
Let $[\x\rel\y]$ be a proper relative braid class with $\x$ having $n$
strands. Let $[\p\rel\q]=[\x\rel\y]\cdot\Delta^{2g}$ be sufficiently
twisted so as to be positive. Then,
\begin{equation}
\label{eqn:iso1}
\bH_{*-2ng}([\x\rel\y])
\cong {\rm HM}_*([\p\rel\q])
\cong {\rm HC}_*([\p\rel\q]).
\end{equation}
\end{conjecture}

A result of this type would be crucial for computing the Floer and Morse
homology via the discrete invariants.

\subsection{Mapping tori and $J$-holomorphic curves}
\label{subsec:map-J}
In this paper we consider the Floer equations on the open symplectic
manifolds $C_n(\D^2)$. Another way to approach this problem is to
consider appropriate mapping tori. For simplicity consider the case
$n=1$ and an area-preserving mapping $f:\D^2\to \D^2$. We define
the mapping torus of $f$ by
$$
\D^2(f) \bydef \R \times \D^2/\bigl((t + 1, x)
\sim (t, f(x)) \bigr).
$$
Loops $\gamma: \R \to \D^2$ satisfying $\gamma(t+1) =
f(\gamma(t))$ for the space of twisted loops $\Omega(\D^2(f))$ and a
fixed point $x$ of $f$ can be regarded as a constant loop in
$\Omega(\D^2(f))$. Instead of considering the Floer equations in a
relative braid class we may consider $J$-holomorphic curves $u(s,t)$
satisfying $u(s,t+1) = f(u(s,t))$, $\lim_{s\to\pm\infty} u(s,\cdot) =
x_\pm$ and
$$
u_s + J(t,u) u_t = 0.
$$
For more details for general closed surfaces see \cite{DS1},
\cite{Seidel1}, \cite{CottonClay1}. This approach has various
advantages for defining braid class invariants and will be subject of
further study.

For the case $n=1$ the  nature of our problem  only requires periodic
boundary conditions on $u$. In order to study
anti-symplectomorphisms $f$, i.e. $f^*\omega = -\omega$, we can
invoke appropriate boundary conditions and go through the same
procedures to define Floer invariants.

\subsection{Further structures}
\label{subsec:fur}
In \cite{Seidel1} and \cite{CottonClay1} Floer homology for mapping
classes of  symplectomorphisms of closed surfaces are defined and
computed. Of particular interest is the work in \cite{CottonClay1} where
the Floer homology of pseudo-Anosov mapping classes is computed. In
this article we computed the Floer homology of relative braid classes
which can be realized by autonomous (integrable) Hamiltonians,
which is strongly related to mapping classes of finite type. In this
case one typically obtains homology for two consecutive indices (the
case $n=1$). Of particular interest are skeletons $\y$ that correspond
to pseudo-Anosov mapping classes. The Floer homology in this article
seeks associated proper relative braid classes $[\x\rel\y]$ for which the
Floer homology is well-defined. There are two questions that are
prominent in this setting: (1) Find the proper relative braid classes for
the given $\y$, and (2) compute the Floer homology. If we go the
route of the Sec.\ \ref{subsec:FMC} we can reduce this question to a
purely finite dimensional problem which can be solved using some
basic combinatorics and code for cubical homology. Another approach
is to use algorithms using representation of the braid group, as where
used in \cite{Jiang}, to determine the proper relative braid classes
$[\x\rel\y]$ for a given pseudo-Anosov braid $\y$. The arguments in
\cite{CottonClay1} suggest that the Floer homology of such relative
braid classes can be evaluated by considering appropriate
pseudo-Anosov representatives. Based on some earlier calculations in
\cite{GVV} we suspect that the Floer homology will be non-trivial for
exactly one index (the case $n=1$). In the case $n>1$ richer Floer
homology should be possible: a subject for further study.

\end{sloppypar}

\bibliography{gvvw-refs}{}
\bibliographystyle{plain}

\appendix


\section{Background: mapping classes of surfaces}
\label{sec:mapclass}
In this appendix, a number of well-known facts and results on
mapping classes of area-preserving diffeomorphisms of the 2-disc are
stated and proved, culminating in Proposition \ref{prop:ham-class2}.
This result can be found as, e.g., Lemma 1(b) in the paper of Boyland
\cite{Boyland2}. We include background and the proof for the sake of
readers wishing to see details in the (not so commonly stated) case of
invariant-but-not-pointwise-fixed boundaries.

\subsection{Isotopies of area forms}
\label{subsec:ar1}
Let $M$ be a smooth compact orientable 2-manifold. Two area forms
$\omega$ and $\omega'$ are \style{isotopic} if there exists a smooth
1-parameter family of area forms $\{\omega_t\}_{t\in [0,1]}$ such
that $\omega_t|_{t=0} = \omega$ and $\omega_t|_{t=1} = \omega'$,
and $\int_M \omega_t$ is constant for all $t\in [0,1]$. Two area
forms are \style{strongly isotopic} is if there exists an isotopy
$\{\psi_t\}_{t\in [0,1]}$ of $M$ such that $\psi_1^*\omega' = \omega$,
in which case it holds that $\psi_t^* \omega_t = \omega$ for all $t\in
[0,1]$ (viz. \cite{MS}).
\begin{lemma}
\label{lem:iso1}
Cohomologous area forms $\omega$, $\omega'$ on $\rD$ are strongly
isotopic.
\end{lemma}

\begin{proof}
Consider the homotopy $\omega_t = \omega + t(\omega'-\omega)$.
Denote the embedding $M\cong \rD\subset \R^2$ by $\varphi$. With
respect to the area forms $\widetilde \omega_t =
(\varphi^{-1})^*\omega_t$, the embeddings $\varphi: (M,\omega_t)
\to (\rD,\widetilde\omega_t)$ are symplectomorphisms for all $t\in
[0,1]$. Since $\varphi$ is a diffeomorphism it has degree 1 and
therefore $\int_{\rD} \widetilde \omega_t = \int_M \omega_t$ are
constant and the forms $\widetilde \omega_t$ are cohomologous for
all $t\in [0,1]$.

Write $\widetilde \omega = a(x) dp\wedge dq$, $\widetilde\omega'
=a'(x)dp\wedge dq$ with $a(x), a'(x)>0$ and $\widetilde\omega_t =
a_t(x) dp\wedge dq$ with $a_t(x) = a(x) + t(a'(x)-a(x))>0$ on
$[0,1]\times \rD$. We seek a smooth isotopy $\psi_t$ in ${\rm
Diff}_0(\rD)$ such that $\psi_t^* \widetilde\omega_t =
\widetilde\omega$ for all $t\in [0,1]$. Such an isotopy can be found
via Moser's stability principle, see \cite{MS}. In order to apply Moser's
stability principle we need to find a 1-parameter family of 1-forms
$\sigma_t$ on $\rD$ such that $\frac{d\widetilde \omega_t}{dt}  =
d\sigma_t$. Define the  vector field $X_t$  via the relation $\sigma_t
= -\iota_{X_t} \omega_t$ and consider the equation $\frac{d
\psi_t}{dt} = X_t \circ \psi_t$. This gives the desired isotopy provided
$\langle X_t(x),{\bf n}\rangle =0$, where ${\bf n}$ is the outward
pointing normal on $\partial\rD$.

Set $\sigma_t = -\partial_q \phi_t(x) dp + \partial_p \phi_t(x) dq$,
where $\phi_t(x)$ is a smooth potential function on $[0,1]\times
\rD$. The vector field $X_t$ is then given by $X_t = \frac{1}{a_t(x)}
\nabla \phi_t(x)$. The differential equation for  $\phi_t$   becomes
$\frac{d a_t(x)}{dt}  = (a'-a)(x) =  \Delta \phi_t(x)$ on $[0,1] \times
\rD$. A boundary condition is found as follows. The condition $\langle
X_t(x),{\bf n}\rangle =0$ is equivalent to the condition that
$\frac{\partial \phi_t(x)}{\partial {\bf n}} =0$.

The potential $\phi_t=\phi$ satisfies the Neumann boundary value
problem for the Laplacian. Since the forms $\widetilde\omega_t$ are
cohomologous, it follows that
$$
0 = \int_{\rD}  (a'-a)(x) dp\wedge dq =
 \int_{\rD} \Delta \phi(x) dp\wedge dq  = -\int_{\partial\rD}\frac{\partial \phi(x)}{\partial {\bf n}} =0.
$$
The Neumann problem is therefore well-posed and has a unique
smooth solution $\phi$ up to an additive constant. By construction we
have that $\psi_t^* \widetilde\omega_t =  \widetilde \omega$ and the
thus $\omega$ and $\omega'$ are strongly isotopic via $\varphi^{-1}
\psi_t \varphi$.
\end{proof}

\begin{corollary}
\label{cor:iso2} Let $M\cong \rD$, then any symplectomorphism
$f:(M,\omega) \to (M,\omega)$ is conjugate to a symplectomorphism
$\widetilde f: (\rD,\omega_0) \to (\rD,\omega_0)$.
\end{corollary}

\begin{proof}
let $\psi: M \to \rD$ be a diffeomorphism and $\widetilde \omega =
(\psi^{-1})^*\omega$. This yields a commutative diagram of
symplectomorphisms
$$
\begin{CD}
(M,\omega) @>\psi>> (\rD,\widetilde\omega)\\
@VVfV      @VV\psi f \psi^{-1}V\\
(M,\omega) @>\psi>> (\rD,\widetilde\omega)
\end{CD}
$$
with $(\psi f\psi^{-1})^* \widetilde \omega = (\psi^{-1})^* f^* \psi^*
\widetilde \omega = (\psi^{-1})^* f^* \omega =  (\psi^{-1})^* \omega =
\widetilde \omega$. By Lemma \ref{lem:iso1} there exists a
diffeomorphism $\varphi:\rD \to \rD$ such that $\varphi^* \widetilde
\omega = c_0\omega_0$ with $c_0 = \int_{\rD} \widetilde
\omega/\int_{\rD} \omega_0$. The mapping $\widetilde f
=\varphi^{-1} \psi f \psi^{-1}\varphi$ preserves $c_0\omega_0$ and
thus $\omega_0$ is the desired symplectomorphism.
\end{proof}

From this point on we  restrict ourselves to area-preserving
diffeomorphisms of the standard disc with $b$ inner discs removed.

\subsection{Symplectic mapping classes}
\label{subsec:sympmap1}
Two symplectomorphisms $f,g \in \symp(M,\omega)$ are
\style{symplectically isotopic} if  there exists an isotopy $\psi_t$ such
that $\psi_0 = g$, $\psi_1 = f$ and $\psi_t^*\omega = \omega$ for all
$t\in [0,1]$. The symplectic isotopy classes form a group under
composition and is called the symplectic mapping class group
$\Gamma_{\rm symp}(M,\omega) \bydef
\pi_0\bigl(\symp(M,\omega)\bigr)$. For any smooth surface $M \cong
\rD$ it follows from Corollary \ref{cor:iso2} that $\Gamma_{\rm
symp}(M,\omega) \cong \Gamma_{\rm symp}(\rD,\omega_0)$.

\begin{proposition}
\label{prop:mapclass1}
It holds that $\Gamma_{\rm symp}(\rD,\omega_0) \cong
\Gamma^+(\rD),$ where $\Gamma^+(\rD)$ is the mapping class group
of orientation-preserving diffeomorphisms of $\rD$.
\end{proposition}

\begin{proof}
Let $f,g \in\symp(\rD,\omega_0)$ be isotopic via a symplectic isotopy
$\psi_t$, so that $g^{-1}f \in \symp_0(\rD,\omega_0)$. Next let $f,g
\in \symp(\rD,\omega_0)$ be isotopic in ${\rm Diff}(\rD)$, then
$g^{-1}f \in {\rm Diff}_0(\rD)$. Let $h_t$ be a smooth isotopy
between ${\rm id}$ and $g^{-1}f$. The isotopy $h_t$ does not
necessarily preserve $\omega_0$ and we set $\omega_t \bydef
h_t^*\omega_0$. It holds that $\omega_t =\omega_0$ at $t=0$ and
$t=1$. We claim that $\int_{\rD} \omega_t = \int_{\rD} \omega_0 =
{\rm area}(\rD)$ for all $t\in [0,1]$. Indeed, since $h_t$ is a smooth
1-parameter family of diffeomorphisms it holds that
$$
1=\deg(h_t) = \frac{1}{{\rm area}(\rD)} \int_{\rD} \omega_t.
$$
Write   $\omega_t = a_t(x) dp\wedge dq$ with $a_t(x)>0$ on
$[0,1]\times \rD$ and $a_0=a_1 =1$. Define the 2-parameter family
of cohomologous area forms $\omega_t^s = \omega_0 + s(\omega_t -
\omega_0)$ and consider the equation  $\frac{d\omega^s_t}{ds}  =
d\sigma_t^s$. Define the vector field $X_t^s=\frac{1}{a_t^s(x)}\nabla
\phi_t^s$ via the relation $\sigma_t^s  = -\iota_{X_t^s} \omega_t^s$
for some potential function $\phi_t^s: \rD \to \R$ and consider the
equation $\frac{d \chi_t^s }{ds} = X_t^s \circ \chi_t^s$. This gives the
desired isotopy provided $\langle X_t^s(x),{\bf n} \rangle  =0$ for $x
\in \partial \rD$.  The potential satisfies the differential equation
$a_t(x) -1   =  \Delta \phi_t^s(x)$ on $[0,1] \times \rD$ and $\phi_t^s
= \phi_t$. As before  we assume the Neumann boundary conditions
$\frac{\partial \phi_t(x)}{\partial {\rm n}} =0$ on $\partial\rD$. This
problem is well-posed since the forms $\omega_t$ are cohomologous,
and the Neumann problem has, up to an additive constant, a unique
smooth 1-parameter family of solutions $\phi_t$ . It is clear that
$X_t^s = \frac{1}{a_t^s(x)}\nabla \phi_t = 0$ for $t=0$ and $t=1$ and
therefore $\chi_t^1={\rm id}$ for $t=0$ and $t=1$.

Via $\chi_t^s$ we have that $(\chi_t^1)^* \omega_t = (\chi_t^1)^*
h_t^* \omega_0 = \omega_0$ and thus the isotopy $h_t \circ
\chi_t^1$ is a symplectic isotopy. We have shown now that two
symplectomorphisms $f,g \in \symp(\rD,\omega_0)$ are
symplectically isotopic if and only if they are (smoothly) isotopic,
which proves the proposition.
\end{proof}

It is a simple generalization to show that in the case of $\D^2_{b,m}$,
that is, $\D^2_b$ with a set of $m$ marked points in the interior, the
mapping that leave the set of $m$ points invariant satisfies
$\Gamma_{\rm symp}(\D^2_{b,m},\omega_0) \cong
\Gamma^+(\D^2_{b,m})$.

\subsection{Hamiltonian diffeomorphisms and isotopies}
\label{subsec:hamdiff1}
Any path $\psi_t$ in $\symp_0(\rD,\omega_0)$ satisfies the initial
value problem $\frac{d}{dt} \psi_t = X_t \circ \psi_t$, for $X_t =
\frac{d}{dt} \psi_t \circ \psi_t^{-1}$. Since $\psi_t^*\omega_0 =
\omega_0$ it holds that $d\iota_{X_t}\omega_0=0$ for all $t\in
[0,1]$. A symplectic isotopy is \style{Hamiltonian} if this closed form
is exact: if there exists $H:[0,1]\times \rD \to \R$ such that
$\iota_{X_t}\omega_0 = -dH(t,\cdot)$. In this case $\psi= \psi_1$ is
called a Hamiltonian (or exact) symplectomorphism. The subgroup of
Hamiltonian symplectomorphisms is denoted $\ham(\rD,\omega_0)$.

\begin{proposition}
\label{prop:hamiso1}
$\symp_0(\rD,\omega_0) = \ham(\rD,\omega_0)$.
\end{proposition}

\begin{proof}
For $\psi_t \in \symp_0(\rD,\omega_0)$ a symplectic isotopy,
$\theta_t = \iota_{X_t}\omega_0$ is a closed 1-form and $i^* \theta_t
=0$, where $i: \partial \rD \to \rD$. In other words, the 1-forms
$\theta_t$ are normal with respect to $\partial \rD$. From the Hodge
decomposition theorem for 1-forms, $\theta_t = -dh_t +
\mathbf{h}_t$, where $h_t|_{\partial \rD} =0$ and $ \mathbf{h}_t$ is
a harmonic 1-field (i.e. $d \mathbf{h}_t = d^* \mathbf{h}_t =0$) with
$i^* \mathbf{h}_t =0$. By the fundamental theorem of Hodge theory
it follows that the space of harmonic 1-fields $\H^1$ is isomorphic to
$H^1(\rD,\partial \rD;\R) \cong \R^{b}$, see e.g. \cite{Morrey} and
\cite{Taylor}. In particular, the isomorphism is given via the periods.
To be more precise, let $\gamma_1,\cdots,\gamma_{b}$ be
generators for $H_1(\rD,\partial \rD;\R)$, then there exists a unique
harmonic 1-field $\mathbf{h}_t$ with prescribed periods
$\int_{\gamma_i} \mathbf{h}_t =c_i(t)\in \R$. We choose $\gamma_i$
as a connection between $\partial \D^2$ and the $i$th inner boundary
circle. The terms in the Hodge decomposition are found as follows.
Apply $d^*$, then $-\Delta h_t = d^* \theta_t$ with the Dirichlet
boundary conditions on $\partial\rD$, which provides a unique smooth
solution $h_t$. Next, set $\mathbf{h}_t = d\tilde h_t$, then, since
$\mathbf{h}_t$ is normal it holds that $\nabla \tilde h_t$ is normal to
$\partial \rD$ and therefore $\tilde h_t = {\rm const.}$ on $\partial
\rD$. Since $\mathbf{h}_t$ is harmonic, $\tilde h_t$ satisfies $\Delta
\tilde h_t =0$. If we apply Stokes' Theorem to the generators
$\gamma_i$ we obtain $c_i(t) \bydef\int_{\gamma_i} \theta_t =
\int_{\gamma_i} d\tilde h_t = \tilde h_t|_{\partial \gamma_i}$. On
$\partial \D^2$ set $\tilde h_t = 0$ and on the $i$th inner boundary
circles set $\tilde h_t = c_i(t)$. There exists a unique smooth
harmonic solution $\tilde h_t$. Since the Hodge decomposition is
unique this gives $\mathbf{h}_t$. Finally, define $H(t,\cdot) = -h_t +
\tilde h_t$ to be the desired Hamiltonian.
\end{proof}

Similarly,
$$
\symp_0(\D^2_{b,m},\omega_0) =
\ham(\D^2_{b,m},\omega_0).
$$

This result about the identity components being Hamiltonian is used
now to iterate symplectomorphisms via Hamiltonian systems. We
start with the subgroup of the symplectic mapping class group
$\Gamma_{\rm symp}(\D^2_{b,m},\omega_0)$ and let $g$ be a
representative of a mapping class. Since $\Gamma_{\rm
symp}(\D^2_{b,m},\omega_0)$ is a subgroup of
$\cB_{b+m}/Z(\cB_{b+m})$ we can represent a mapping class $[g]$ by
a braid consisting of strands and cylinders. Choose a symplectic
isotopy $g_t$ in $\symp(\D^2,\omega_0)$ tracing out a braid and not
deforming the boundary circles inside $\D^2$. Since
$\symp_0(\D^2,\omega_0) = \ham(\D^2,\omega_0)$,    the isotopy
$g_t$ is Hamiltonian and there exists a Hamiltonian $K(t,\cdot)$
which generates $g_t$. Now let $f \in [g] \in \Gamma_{\rm
symp}(\D^2_{b,m},\omega_0)$, then since
$\symp_0(\D^2_{b,m},\omega) = \ham(\D^2_{b,m},\omega)$ there
exists a Hamiltonian $\widehat H(t,\cdot)$ that generates a
Hamiltonian isotopy $\hat f_t$ between ${\rm id}$ and $g^{-1} f$.
Now extend the Hamiltonian $\widehat H$   to all of $\D^2$, then the
Hamiltonian $H(t,\cdot) = K(t,\cdot) + \widehat
H\bigl(t,(g^{-1}(\cdot)\bigr)$ generates the Hamiltonian isotopy
$g_t\hat f_t$ between ${\rm id} $ and $f$ in
$\ham(\D^2,\omega_0)$.

The above Hamiltonians $H(t,\cdot)$ can be chosen to be periodic in
$t$, i.e. $H(t+1,\cdot) = H(t,\cdot)$. Reparametrize time $t$ by
$\tau(t)$ where $\tau$ is  a smooth  increasing function of $t$ on
$[0,1]$ with support in $(0,1)$ and $\tau(0) = 0$ and $\tau(1) = 1$.
Then $\tau'(t) H(\tau(t),\cdot)$ can be extended periodically in $t$.
These consideration bring us to the following class of Hamiltonians
$\cH(\D^2)$:
\begin{list}{}{\setlength{\leftmargin}{6ex}}
\item[(h1)] $H\in C^\infty(\R\times \D^2;\R)$;
\item[(h2)] $H(t+1,x) = H(t,x)$ for all $t\in \R$ and  all $x\in \D^2$;
\item[(h3)] $H(t,x)=0$ for all $x\in {\partial\D^2}$ and for  all $t\in \R$.
\end{list}

\begin{proposition}
\label{prop:ham-class2}
Given a symplectomorphism $f  \in \symp(\D^2_{b,m},\omega_0)$, then
there exists a Hamiltonian $H \in \cH$ such that
$f = \psi_{1,H}$, where
$\psi_{t,H}$ is the Hamiltonian isotopy generated by $H$.
\end{proposition}

\end{document}

%% file: macros-theorems.tex
\newenvironment{definition}[1][{\bf Definition.}]
{
\medskip\goodbreak
\refstepcounter{theorem}
\noindent$\blacktriangleleft$~\textbf{\thetheorem}~\textbf{#1~}}
{\hfill$\blacktriangleright$ 
}

\newenvironment{lemma}[1][{\bf Lemma.}]
{
\medskip\goodbreak

\refstepcounter{theorem}
\noindent$\blacktriangleleft$~\textbf{\thetheorem}~\textit{#1~}}
{\hfill$\blacktriangleright$\vskip.1cm

}

\newenvironment{corollary}[1][{\bf Corollary.}]
{
\medskip\goodbreak

\refstepcounter{theorem}
\noindent$\blacktriangleleft$~\textbf{\thetheorem}~\textit{#1~}}
{\hfill$\blacktriangleright$\vskip.1cm

}

\newenvironment{conjecture}[1][{\bf Conjecture.}]
{
\medskip\goodbreak

\refstepcounter{theorem}
\noindent$\blacktriangleleft$~\textbf{\thetheorem}~\textit{#1~}}
{\hfill$\blacktriangleright$\vskip.1cm

}

\newenvironment{THM}[1][{\bf Theorem.}]
{
\medskip\goodbreak

\refstepcounter{theorem}
\noindent$\blacktriangleleft$~\textbf{\thetheorem}~\textit{#1~}}
{\hfill$\blacktriangleright$\vskip.1cm

}

\newenvironment{proposition}[1][{\bf Proposition.}]
{
\medskip\goodbreak

\refstepcounter{theorem}
\noindent$\blacktriangleleft$~\textbf{\thetheorem}~\textit{#1~}}
{\hfill$\blacktriangleright$\vskip.1cm

}

\newenvironment{example}[1][Example.]
{
\smallskip\goodbreak

\refstepcounter{theorem}
\noindent$\blacktriangleleft$~\textbf{\thetheorem}~\textit{#1~}}
{\hfill$\blacktriangleright$\smallskip

}

\newenvironment{remark}[1][Remark.]
{
\smallskip\goodbreak

\refstepcounter{theorem}
\noindent$\blacktriangleleft$~\textbf{\thetheorem}~\textit{#1~}}
{\hfill$\blacktriangleright$\smallskip

}

\newcounter{PROB}

\newcounter{SUBPROB}[PROB]


%% file: macros-gvvw.tex
\newcommand{\rmref}[1]{{\rm(\ref{#1})}}
\newcommand{\bydef}{\stackrel{\mbox{\tiny\textnormal{\raisebox{0ex}[0ex][0ex]{def}}}}{=}}

\newcommand{\Sb}{{\mathbb S}}
\newcommand{\R}{{\mathbb R}}
\newcommand{\A}{{\mathbb A}}
\newcommand{\CC}{{\mathbb C}}
\newcommand{\Z}{{\mathbb Z}}
\newcommand{\D}{{\mathbb{D}}}
\newcommand{\N}{{\mathbb{N}}}

\newcommand{\cB}{{\mathcal B}}

\newcommand{\cG}{{\mathcal G}}

\newcommand{\cL}{\mathcal{L}} 

\newcommand{\cM}{{\mathscr M}}
\newcommand{\M}{{\mathscr M}}
\newcommand{\wM}{ \widehat{\mathscr M}}
\newcommand{\cZ}{{\mathscr Z}}
\newcommand{\cT}{{\mathscr T}}
\newcommand{\cU}{{\mathscr U}}
\newcommand{\cV}{{\mathscr V}}
\newcommand{\cH}{{\mathscr H}}
\newcommand{\cS}{{\mathscr S}}
\newcommand{\cN}{{\mathscr N}}
\newcommand{\cJ}{{\mathscr J}}
\newcommand{\cHH}{{\mathscr H}_{\rm reg}}

\newcommand{\x}{\textsc{x} }

\newcommand{\f}{{\mathscr L}}
\newcommand{\y}{\textsc{y}}

\renewcommand{\u}{\textsc{u}}

\newcommand{\p}{\textsc{p}}
\newcommand{\q}{\textsc{q}}

\newcommand{\rel}{{~{\rm rel}~}}
\newcommand{\oH}{{{\overline H}}}
\newcommand{\rD}{{\mathbb D}^2_{b}}
\newcommand{\bH}{{\rm HB}}

\renewcommand{\P}{{\rm Crit}}
\newcommand{\Cr}{{\rm Cross}}

\renewcommand{\H}{{\mathcal H}}

\newcommand{\eps}{\varepsilon}
\newcommand{\uu}{{\bf v}}

\newcommand{\cre}{{\partial_{J,H}}}
\newcommand{\Int}{{\rm int}}

\newcommand{\Sp}{{\rm Sp}}
\newcommand{\ind}{{\rm ind}}
\newcommand{\bO}[1]{\overline{\Omega}^{#1}}
\newcommand{\bOn}{\bO{n}}

\newcommand{\bh}{{ h}}

\newcommand{\loc}{{\text{\textup{loc}}}}
\newcommand{\symp}{{{\rm Symp}}}
\newcommand{\ham}{{{\rm Ham}}}
\newcommand{\cl}{{{\rm cl}}}

\renewenvironment{proof}[1][]{\noindent{\it Proof{#1}.
\hspace{1em}}}{\mbox{ }\hfill $\Box$\par}

\newcommand{\mat}{\begin{pmatrix}}
\newcommand{\rix}{\end{pmatrix}}
\newcommand{\tmat}{\left(\begin{smallmatrix}}
\newcommand{\trix}{\end{smallmatrix}\right)}

\renewcommand{\emph}[1]{{\bfseries\itshape #1}}

\renewcommand{\r}{\right}
\renewcommand{\l}{\left}

%% file: floerbraidspapersubmission.bbl
\begin{thebibliography}{10}

\bibitem{AM1}
Alberto Abbondandolo and Pietro Majer.
\newblock Morse homology on {H}ilbert spaces.
\newblock {\em Comm. Pure Appl. Math.}, 54(6):689--760, 2001.

\bibitem{Ang2}
Sigurd Angenent.
\newblock The zero set of a solution of a parabolic equation.
\newblock {\em J. Reine Angew. Math.}, 390:79--96, 1988.

\bibitem{Ang3}
Sigurd Angenent.
\newblock Inflection points, extatic points and curve shortening.
\newblock In {\em Hamiltonian systems with three or more degrees of freedom
  ({S}'{A}gar\'o, 1995)}, volume 533 of {\em NATO Adv. Sci. Inst. Ser. C Math.
  Phys. Sci.}, pages 3--10. Kluwer Acad. Publ., Dordrecht, 1999.

\bibitem{AV}
Sigurd Angenent and Robertus van~der Vorst.
\newblock A superquadratic indefinite elliptic system and its
  {M}orse-{C}onley-{F}loer homology.
\newblock {\em Math. Z.}, 231(2):203--248, 1999.

\bibitem{aronszajn}
Nachman Aronszajn.
\newblock A unique continuation theorem for solutions of elliptic partial
  differential equations or inequalities of second order.
\newblock {\em J. Math. Pures Appl. (9)}, 36:235--249, 1957.

\bibitem{Birman}
Joan Birman.
\newblock {\em Braids, links, and mapping class groups}.
\newblock Princeton University Press, Princeton, N.J., 1974.
\newblock Annals of Mathematics Studies, No. 82.

\bibitem{Boyland2}
Philip Boyland.
\newblock Dynamics of two-dimensional time-periodic euler fluid flows.
\newblock {\em Topology Appl.}, 152(1-2):87--106, 2005.

\bibitem{CottonClay1}
Andrew Cotton-Clay.
\newblock Symplectic floer homology of area-preserving surface diffeomorphisms.
\newblock {\em arXiv}, math.SG, Jan 2008.
\newblock 47 pages, 4 figures.

\bibitem{DS1}
Stamatis Dostoglou and Dietmar Salamon.
\newblock Self-dual instantons and holomorphic curves.
\newblock {\em Ann. of Math. (2)}, 139(3):581--640, 1994.

\bibitem{EGH}
Yasha Eliashberg, Alexander Givental, and Helmut Hofer.
\newblock Introduction to symplectic field theory.
\newblock {\em Geom. Funct. Anal.}, (Special Volume, Part II):560--673, 2000.
\newblock GAFA 2000 (Tel Aviv, 1999).

\bibitem{FH1}
A~Floer and H~Hofer.
\newblock Symplectic homology. i. open sets in ${\bf c}^n$.
\newblock {\em Math. Z.}, 215(1):37--88, 1994.

\bibitem{Floer1}
Andreas Floer.
\newblock Symplectic fixed points and holomorphic spheres.
\newblock {\em Comm. Math. Phys.}, 120(4):575--611, 1989.

\bibitem{Floer5}
Andreas Floer.
\newblock Instanton homology, surgery, and knots.
\newblock In {\em Geometry of low-dimensional manifolds, 1 ({D}urham, 1989)},
  volume 150 of {\em London Math. Soc. Lecture Note Ser.}, pages 97--114.
  Cambridge Univ. Press, Cambridge, 1990.

\bibitem{Franks1}
John Franks.
\newblock Geodesics on {$S\sp 2$} and periodic points of annulus
  homeomorphisms.
\newblock {\em Invent. Math.}, 108(2):403--418, 1992.

\bibitem{Fulton}
William Fulton.
\newblock {\em Algebraic topology}, volume 153 of {\em Graduate Texts in
  Mathematics}.
\newblock Springer-Verlag, New York, 1995.
\newblock A first course.

\bibitem{Garside}
Frank Garside.
\newblock The braid group and other groups.
\newblock {\em Quart. J. Math. Oxford Ser. (2)}, 20:235--254, 1969.

\bibitem{GVV}
Robert Ghrist, Jan~Bouw Van~den Berg, and Robert Vandervorst.
\newblock Morse theory on spaces of braids and {L}agrangian dynamics.
\newblock {\em Invent. Math.}, 152(2):369--432, 2003.

\bibitem{Ginzburg}
Viktor Ginzburg.
\newblock The conley conjecture.
\newblock Preprint 2006.

\bibitem{Hingston}
Nancy Hingston.
\newblock Subharmonic solutions of the hamilton equations on tori.
\newblock Preprint 2004.

\bibitem{HZ}
Helmut Hofer and Eduard Zehnder.
\newblock {\em Symplectic invariants and {H}amiltonian dynamics}.
\newblock Birkh\"auser Advanced Texts: Basler Lehrb\"ucher. [Birkh\"auser
  Advanced Texts: Basel Textbooks]. Birkh\"auser Verlag, Basel, 1994.

\bibitem{Jiang}
Boju Jiang and Hao Zheng.
\newblock A trace formula for the forcing relation of braids.
\newblock {\em Topology}, 47(1):51--70, 2008.

\bibitem{McDuff2}
Dusa McDuff.
\newblock Singularities of {$J$}-holomorphic curves in almost complex
  {$4$}-manifolds.
\newblock {\em J. Geom. Anal.}, 2(3):249--266, 1992.

\bibitem{McDuff1}
Dusa McDuff.
\newblock Singularities and positivity of intersections of {$J$}-holomorphic
  curves.
\newblock In {\em Holomorphic curves in symplectic geometry}, volume 117 of
  {\em Progr. Math.}, pages 191--215. Birkh\"auser, Basel, 1994.
\newblock With an appendix by Gang Liu.

\bibitem{MS}
Dusa McDuff and Dietmar Salamon.
\newblock {\em Introduction to symplectic topology}.
\newblock Oxford Mathematical Monographs. The Clarendon Press Oxford University
  Press, New York, second edition, 1998.

\bibitem{Milnor}
John Milnor.
\newblock {\em Morse theory}.
\newblock Based on lecture notes by M. Spivak and R. Wells. Annals of
  Mathematics Studies, No. 51. Princeton University Press, Princeton, N.J.,
  1963.

\bibitem{Morrey}
Charles Morrey, Jr.
\newblock {\em Multiple integrals in the calculus of variations}.
\newblock Die Grundlehren der mathematischen Wissenschaften, Band 130.
  Springer-Verlag New York, Inc., New York, 1966.

\bibitem{Oancea}
Alexandru Oancea.
\newblock A survey of {F}loer homology for manifolds with contact type boundary
  or symplectic homology.
\newblock In {\em Symplectic geometry and {F}loer homology. {A} survey of the
  {F}loer homology for manifolds with contact type boundary or symplectic
  homology}, volume~7 of {\em Ensaios Mat.}, pages 51--91. Soc. Brasil. Mat.,
  Rio de Janeiro, 2004.

\bibitem{OS1}
Peter Ozsvath and Zoltan Szabo.
\newblock Holomorphic disks and three-manifold invariants: properties and
  applications.
\newblock {\em arXiv}, math.SG, May 2001.

\bibitem{RobSal2}
Joel Robbin and Dietmar Salamon.
\newblock The {M}aslov index for paths.
\newblock {\em Topology}, 32(4):827--844, 1993.

\bibitem{RobSal1}
Joel Robbin and Dietmar Salamon.
\newblock The spectral flow and the {M}aslov index.
\newblock {\em Bull. London Math. Soc.}, 27(1):1--33, 1995.

\bibitem{SalWeb}
Dietman Salamon and Joa Weber.
\newblock Floer homology and the heat flow.
\newblock {\em Geom. Funct. Anal.}, 16(5):1050--1138, 2006.

\bibitem{Sal}
Dietmar Salamon.
\newblock Morse theory, the {C}onley index and {F}loer homology.
\newblock {\em Bull. London Math. Soc.}, 22(2):113--140, 1990.

\bibitem{Sal2}
Dietmar Salamon.
\newblock Lectures on {F}loer homology.
\newblock In {\em Symplectic geometry and topology ({P}ark {C}ity, {UT},
  1997)}, volume~7 of {\em IAS/Park City Math. Ser.}, pages 143--229. Amer.
  Math. Soc., Providence, RI, 1999.

\bibitem{SalZehn1}
Dietmar Salamon and Eduard Zehnder.
\newblock Morse theory for periodic solutions of {H}amiltonian systems and the
  {M}aslov index.
\newblock {\em Comm. Pure Appl. Math.}, 45(10):1303--1360, 1992.

\bibitem{Schwarz}
Matthias Schwarz.
\newblock {\em Morse homology}, volume 111 of {\em Progress in Mathematics}.
\newblock Birkh\"auser Verlag, Basel, 1993.

\bibitem{Seidel1}
Paul Seidel.
\newblock Symplectic floer homology and the mapping class group.
\newblock {\em arXiv}, math.SG, Jan 2000.
\newblock 11 pages, Latex.

\bibitem{Taylor}
Michael Taylor.
\newblock {\em Partial differential equations}, volume~23 of {\em Texts in
  Applied Mathematics}.
\newblock Springer-Verlag, New York, 1996.
\newblock Basic theory.

\bibitem{V1}
Claude Viterbo.
\newblock Functors and computations in floer homology with applications, i.
\newblock {\em Geom. Funct. Anal.}, Jan 1999.

\end{thebibliography}
